\documentclass[onefignum,onetabnum]{siamonline190516} % add [...,review] for submission

\usepackage{lipsum}
\usepackage{amsfonts}
\usepackage{graphicx}
\usepackage{epstopdf}
\usepackage{algorithmic}
\usepackage{dsfont}
\usepackage{enumitem}
\setlist[enumerate]{leftmargin=.5in}
\setlist[itemize]{leftmargin=.5in}
\usepackage{multirow}

\usepackage{pdflscape}
\usepackage{dsfont}
\usepackage[caption=false]{subfig} 
\ifpdf
  \DeclareGraphicsExtensions{.eps,.pdf,.png,.jpg}
\else
  \DeclareGraphicsExtensions{.eps}
\fi

\newsiamremark{remark}{Remark}
\newsiamremark{hypothesis}{Hypothesis}
\crefname{hypothesis}{Hypothesis}{Hypotheses}
\newsiamthm{claim}{Claim}

\newtheorem*{theorem_no_number}{Theorem}

\theoremstyle{plain}

\numberwithin{assumption}{subsection}

\newtheorem{manualassumptioninner}{Assumption}
\newenvironment{manualassumption}[1]{%
  \renewcommand\themanualassumptioninner{#1}%
  \manualassumptioninner
}{\endmanualassumptioninner}

\makeatletter
\newcommand{\customlabel}[2]{%
   \protected@write \@auxout {}{\string \newlabel {#1}{{#2}{\thepage}{#2}{#1}{}} }%
   \hypertarget{#1}{}
}
\makeatother

\DeclareMathOperator*{\argmax}{arg\,max}
\DeclareMathOperator*{\argmin}{arg\,min}

\numberwithin{equation}{subsection}
\numberwithin{theorem}{section}

\headers{Parameter Estimation and Optimal Sensor Placement}{L. Sharrock and N. Kantas}

\title{Joint Online Parameter Estimation and Optimal Sensor Placement for the Partially Observed Stochastic Advection-Diffusion Equation}
\author{Louis Sharrock\thanks{
Department of Mathematics, Imperial College London, South Kensington, London, SW7 2AZ, UK
  (\email{louis.sharrock16@imperial.ac.uk}, \email{n.kantas@imperial.ac.uk}) \newline
  \color{header1} \textbf{Funding:} \color{black} The first author was funded by the EPRSC CDT in the Mathematics of Planet Earth (grant number EP/L016613/1) and the National Physical Laboratory. The second author was partially funded under a J.P. Morgan A.I. Research Award (2019).}%, \url{http://www.imag.com/\string~ddoe/}).}
\and Nikolas Kantas\footnotemark[1]%\thanks{Department of Mathematics, Imperial College London
 }

\usepackage{amsopn}
\DeclareMathOperator{\diag}{diag}

\ifpdf
\hypersetup{
  pdftitle={Joint Online Parameter Estimation and Optimal Sensor Placement for the Partially Observed Stochastic Advection-Diffusion Equation},
  pdfauthor={L. Sharrock, N. Kantas}
}
\fi

\begin{document}

\maketitle

% REQUIRED
\begin{abstract}
In this paper, we consider the problem of jointly performing online parameter estimation and optimal sensor placement for a partially observed infinite dimensional linear diffusion process. We present a novel solution to this problem in the form of a continuous-time, two-timescale stochastic gradient descent algorithm, which recursively seeks to maximise the asymptotic log-likelihood of the observations with respect to the unknown model parameters, and to minimise the expected mean squared error of the hidden state estimate with respect to the sensor locations. We also provide extensive numerical results illustrating the performance of the proposed approach in the case that the hidden signal is governed by the two-dimensional stochastic advection-diffusion equation. 
\end{abstract}

% REQUIRED


\section{Introduction}
\label{sec:introduction}

The study of partially observed stochastic dynamical systems is old, but remains relevant to numerous applications in fields as diverse as acoustics and signal processing, image analysis and computer vision, automatic control and robotics, economics and finance, computational biology and bioinformatics, environmental monitoring, and meteorology (e.g., \cite{Bain2009,Cappe2005,Doucet2001a,Elliott1995}). In this paper, we consider a partially observed stochastic process governed by 
a particular dissipative stochastic partial differential equation (SPDE), namely, 
the stochastic advection-diffusion equation.
%This process is assumed to generate a continuous sequence of noisy observations via a finite set of measurement sensors.  
This equation, \color{black} typically as one component of a hierarchical model (see, e.g., \cite{Cressie2011a}), \color{black} is frequently used in environmental monitoring applications to model phenomena such as precipitation \cite{Brown2001,Liu2019,Sigrist2015}, air pollution \cite{Barrera1997,Liu2016}, %chemical contamination of surface soil \cite{Mohapl2000}, 
groundwater flow \cite{Kumar1991,Unny1989}, and sediment transport \cite{Man2007}. 

We will consider the case in which the hidden state of interest is a space-time varying scalar field, $u(\boldsymbol{x},t)$, on some bounded two-dimensional domain $\Pi\subseteq\mathbb{R}^2$. This state is modelled using the stochastic advection-diffusion equation, which can be written as
\begin{equation}
\frac{\partial u(\boldsymbol{x},t)}{\partial t} = -\boldsymbol{\mu}%_{\boldsymbol{x},t}
^T 
\nabla u(\boldsymbol{x},t) + \nabla\cdot\Sigma%_{\boldsymbol{x},t}
\nabla u(\boldsymbol{x},t)-\zeta%_{{\boldsymbol{x}},t} 
u(\boldsymbol{x},t) 
%+ f(\boldsymbol{x},t) 
+ b(\boldsymbol{x})\varepsilon(\boldsymbol{x},t) \label{SPDE}
\end{equation}
where $\boldsymbol{x}=(x_1,x_2)^T\in\Pi$, $\nabla = (\partial/\partial x_1, \partial/\partial x_2)^T$ is the gradient operator, $\nabla\cdot$ is the divergence operator, 
%$f(\boldsymbol{x},t)$ is a deterministic forcing, 
and ${\varepsilon}(\boldsymbol{x},t)$ is a Gaussian noise process which is temporally white and spatially coloured. This might appear as a restrictive choice for the dynamics, but much of the subsequent methodology is generic, and could thus theoretically be applied to other models of interest (see \cite{Sharrock2020a} for a rigorous treatment). Moreover, this model results in a tractable but non-separable space-time covariance operator \cite{Sigrist2015}, and thus its spatiotemporal dynamics are interpretable for practitioners. \color{black} In addition, the terms in this equation can, if desired, be given a clear physical interpretation. In particular, the first term 
 describes transport effects, also termed {convection} or {advection}, with $\smash{\boldsymbol{\mu}%_{\boldsymbol{x},t}
 =(\mu_1%_{\boldsymbol{x},t}^
 %_x
 ,\mu_2%_{\boldsymbol{x},t}^
 %_y
 )
^T}$ $\smash{\in\mathbb{R}^2}$ the {drift} or {velocity} field. The second term
 describes a possibly anisotropic {diffusion}, with $\smash{\Sigma = [\Sigma_{i,j}]_{i,j=1,2}\in\mathbb{R}^{2\times 2}}$ %_{\boldsymbol{x},t}
 the {diffusivity} or {diffusion matrix}. This matrix can further be parametrised as (e.g., \cite{Sigrist2015})
\begin{equation}
\Sigma^{-1}: =\Sigma^{-1}(\rho_1,\gamma,\alpha)%_{\boldsymbol{x},t}
 = \frac{1}{\rho_{%\boldsymbol{x},t,
1}^2}\left(\begin{array}{cc} \cos{\alpha%_{\boldsymbol{x},t}
} & \hspace{-1mm} \sin{\alpha%_{\boldsymbol{x},t}
} \\ -\gamma%_{\boldsymbol{x},t}
\sin{\alpha%_{\boldsymbol{x},t}
}  & \gamma \cos{\alpha%_{\boldsymbol{x},t}
} \end{array}\right)^T\left(\begin{array}{cc} \hspace{-1mm} \cos{\alpha%_{\boldsymbol{x},t}
} & \sin{\alpha%_{\boldsymbol{x},t}
} \\ -\gamma%_{\boldsymbol{x},t}
\sin{\alpha%_{\boldsymbol{x},t}
} & \gamma\cos{\alpha%_{\boldsymbol{x},t}
} \end{array}\right)
\end{equation}
in which case $\rho_{%\boldsymbol{x},t,
1}\in\mathbb{R}_{+}$ can be viewed as the {range}, which determines the amount of diffusion; $\gamma%_{\boldsymbol{x},t}
\in\mathbb{R}_{+}$ as the {anisotropic amplitude}, which determines the amount of anisotropy; and $\alpha%_{\boldsymbol{x},t}
\in[0,\frac{\pi}{2}]$ as the {anisotropic direction}, which determines the direction of the anisotropy. 
In the case that $\gamma%
%_{\boldsymbol{x},t}
\equiv 1$, this matrix is symmetric, and the diffusion is isotropic. 
%; this term is thus proportional to the standard Laplacian. % $\Delta u(\boldsymbol{x},t)$, where $\Delta$ is the standard Laplacian.  
Finally, the third term describes {damping}, with $\zeta%_{\boldsymbol{x},t}
\in\mathbb{R}_{+}$ the {damping rate}, or {damping coefficient}. \color{black}

The central problem underlying this partially observed stochastic dynamical system is that of optimal state estimation, or filtering. This consists in determining the conditional probability distribution of the latent signal process (i.e., the filter), given the history of observations, under the assumption that any model parameters are known, and the locations of the measurement sensors are fixed (e.g., \cite{Bashirov2003,Curtain1975,Falb1967}). In practical applications, however, it is often the case that the parameters of this model are unknown, and must be inferred from the data. Indeed, inferring the model parameters is often the primary problem of interest (e.g., \cite{Kumar1991,Mohapl2000}). Moreover, the locations of the measurement sensors are typically not fixed, and thus it may be possible to improve upon the optimal state estimate by determining an `optimal sensor placement'. This is particularly relevant to applications in engineering and the applied sciences. %, including meteorology, environmental monitoring, and fluid dynamics. 
In such applications, the process of interest, even if defined continuously over space and in time, can only be measured at a finite number of spatial locations. Moreover, the spatial density of observations is generally very low, due either to prohibitive expense (i.e., the sensors are expensive, or expensive to place), 
or geographical inaccessibility (i.e., the sensors cannot be placed in particular locations). 
Furthermore, measurements at certain points in the domain may yield more information about the system than measurements at other points, due to correlations in the signal. Thus, to a greater or lesser extent, the accuracy of the estimate of the signal is dependent on the number and location of the measurement sensors. 

\color{black}
In this paper, we address, for the first time, the problems of parameter estimation and optimal sensor placement (for the purpose of optimal state estimation) together. This represents a significant departure from the existing literature, in which these two problems have, until now, been studied separately. There is clear motivation for this combined approach. In the vast majority of practical applications, both of these problems are relevant. It would thus be highly convenient to solve them simultaneously, and, if possible, in an online fashion (i.e., in real time). Moreover, they are often interdependent, in the sense that the optimal sensor placement can vary significantly according to the current parameter estimate (see Section \ref{sec:opt_sens}). Thus, tackling them together can result in significant performance improvements (see Figure \ref{fig0}). Before we provide further details on our approach, let us briefly review the existing literature on these two important problems.  \color{black}

\begin{figure}[!h]
\vspace{-2mm}
\centering
\captionsetup[subfloat]{captionskip=-1.5pt}
\subfloat[
%The latent state $u(\boldsymbol{x},t)$.
]{\label{fig_0a}\includegraphics[width=.26\linewidth]{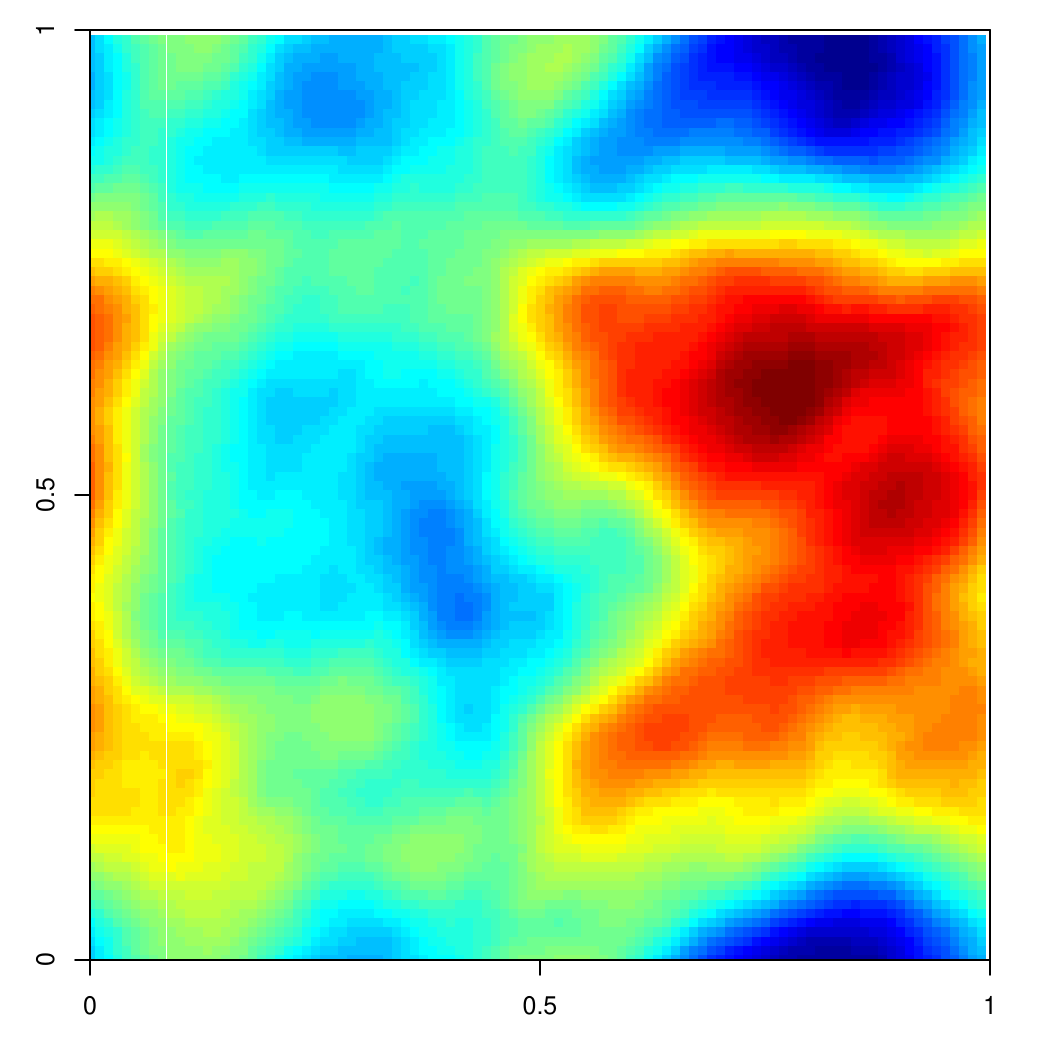}}
  \\[-2mm] %\hspace{6mm}
\subfloat[
%The optimal state estimate $\hat{u}_{\theta,\boldsymbol{o}}(\boldsymbol{x},t)$, with parameter estimation and optimal sensor placement.
]{{\label{fig_0b}\includegraphics[width=.2\linewidth]{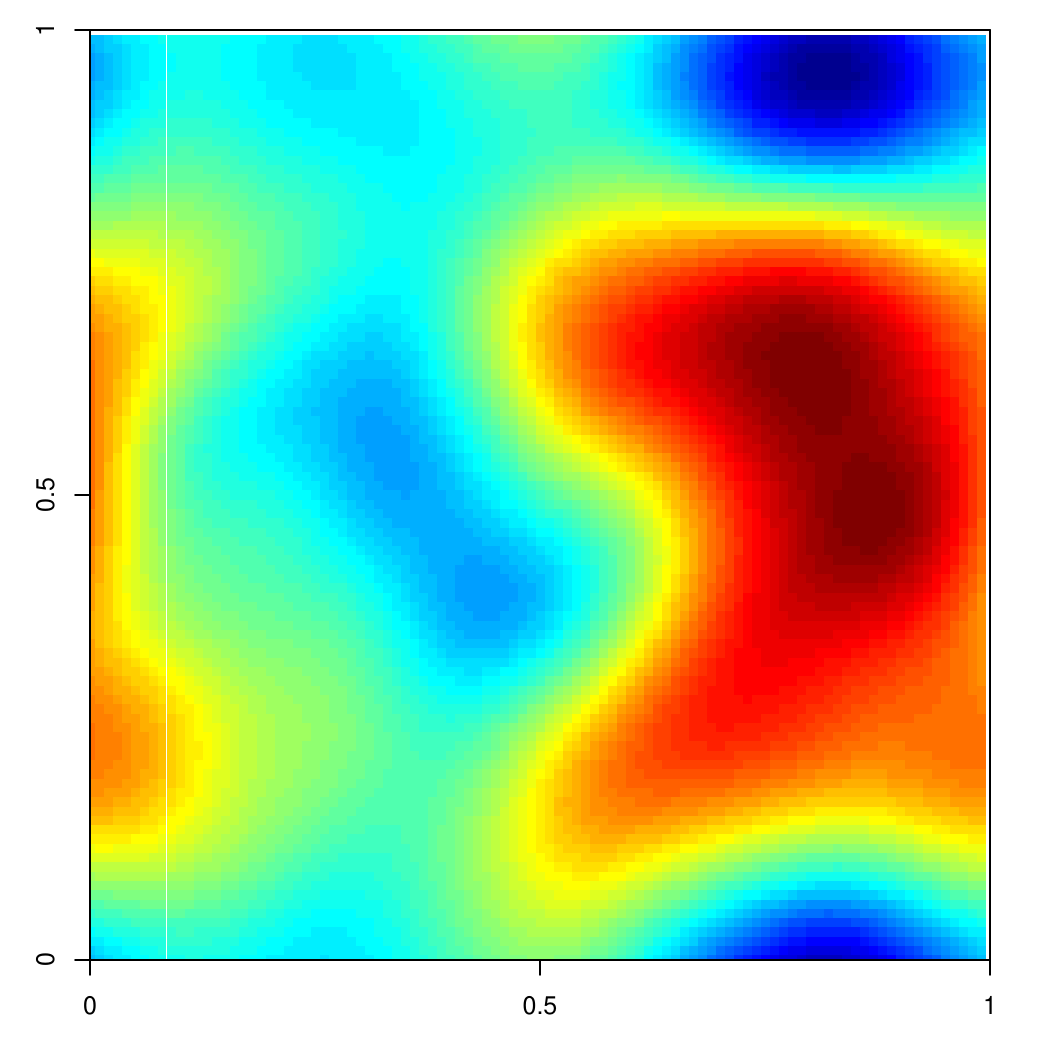}}} %\\[-3mm]
 \hspace{6mm}
\subfloat[
%The optimal state estimate $\hat{u}_{\theta,\boldsymbol{o}}(\boldsymbol{x},t)$, with parameter estimation but no optimal sensor placement.
]{{\label{fig_0c}\includegraphics[width=.2\linewidth]{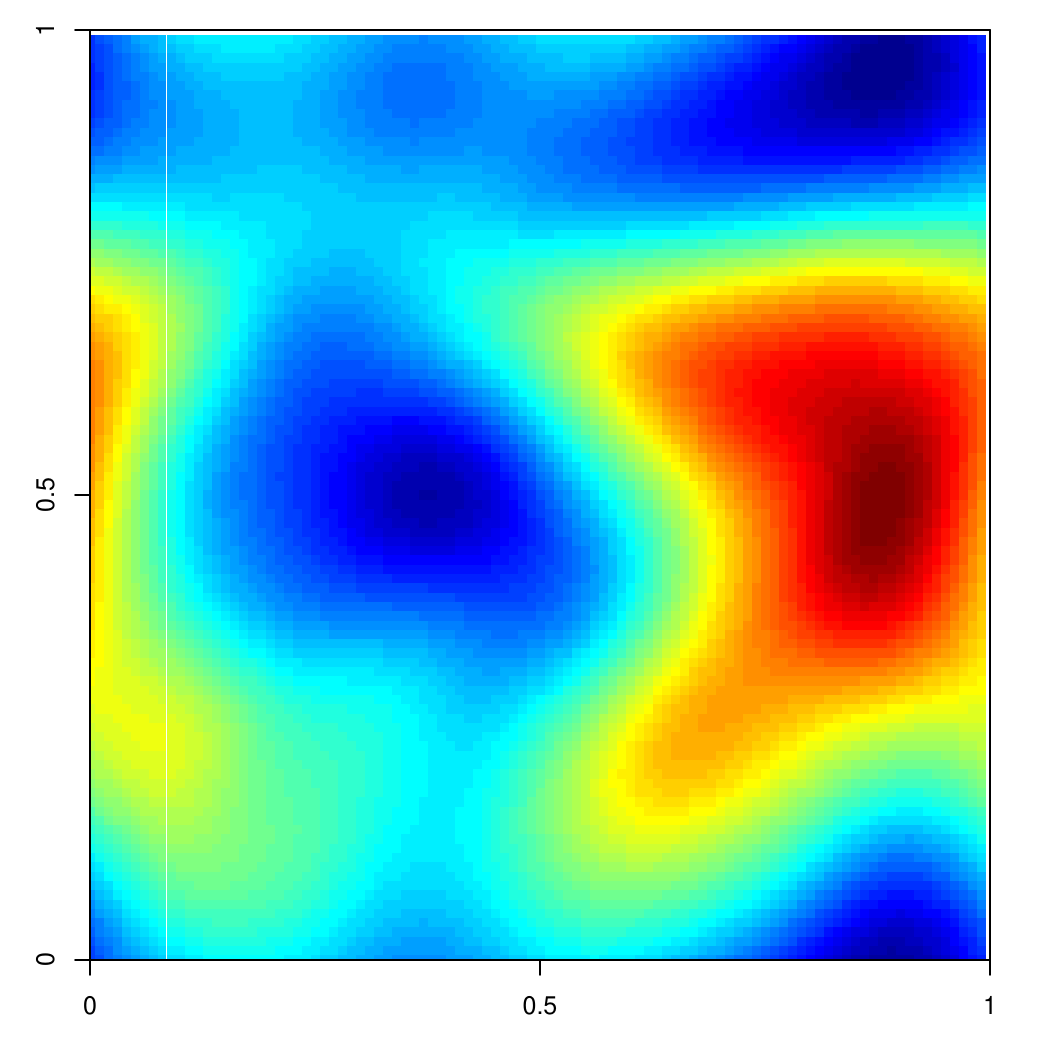}}}
\hspace{6mm}
\subfloat[
%The optimal state estimate $\hat{u}_{\theta,\boldsymbol{o}}(\boldsymbol{x},t)$, with optimal sensor placement but no parameter estimation.
]{\label{fig_0d}\includegraphics[width=.2\linewidth]{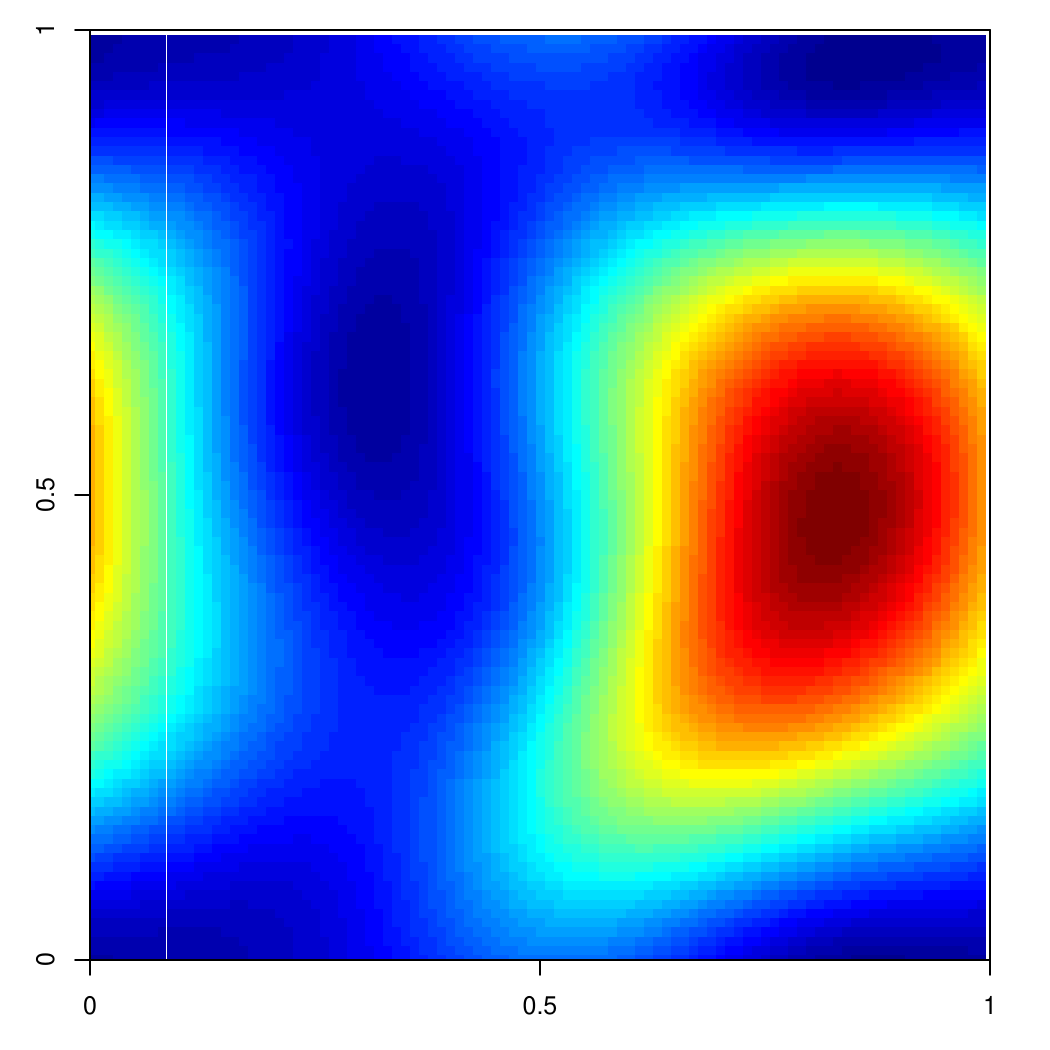}}
% \hspace{15mm}
%\subfloat[No parameter estimation or optimal sensor placement.]{\label{fig_0e}\includegraphics[width=.35\linewidth]{intro_plot_5}}
\captionsetup{aboveskip=8pt}
\caption{A comparison of the true, hidden state $u(\boldsymbol{x},t)$ (Fig. \ref{fig_0a}) and the optimal state estimate $\hat{u}_{\theta,\boldsymbol{o}}(\boldsymbol{x},t)$ obtained in three possible scenarios: using the true parameters and an optimal sensor placement (Fig. \ref{fig_0b}), using the true parameters but a sub-optimal sensor placement (Fig. \ref{fig_0c}), and using an optimal sensor placement but the incorrect parameters (Fig. \ref{fig_0d}). In this example, the hidden state is only accurately reconstructed in the first scenario, when the model parameters are successfully estimated \textbf{and} the sensors are optimally placed. \color{black} The parameters values and sensor placements used to generate these figures are provided in Appendix \ref{app:fig0}. \color{black}}
%\caption{A comparison of the true, hidden state $u(\boldsymbol{x},t)$ (\ref{fig_0a}) and the optimal state estimates $\hat{u}_{\theta,\boldsymbol{o}}(\boldsymbol{x},t)$ obtained in three possible scenarios:  In \ref{fig_0b}, the parameters are estimated and the sensors are optimally placed. In Figure \ref{fig_0c}, the parameters are estimated, but the sensors are not optimally placed. In Figure \ref{fig_0d}, the sensors are optimally placed, but the parameters are not estimated.}
\label{fig0}
\end{figure}

 %This represents the main novelty of our paper. %That is, as outlined below, to simultaneously and recursively maximise the asymptotic log-likelihood with respect to the model parameters, and minimise the asymptotic objective function with respect to the sensor locations. 
%Our proposed approach, extending the work of \cite{Surace2019,Zhang2018}, among others, is based on the use of a two-timescale stochastic gradient descent algorithm in continuous time. 

%This setting is familiar in classical filtering theory, where the problem is to obtain the conditional probability law of the latent signal process, given the history of observations, under the assumption that any model parameters are known, and the sensor locations are fixed (e.g., \cite{Bashirov2003,Curtain1975,Falb1967}). It has also been studied in the context of parameter estimation, whereby the problem is to estimate the true model parameters from the sequence of observations (e.g., \cite{Bagchi1984,Balakrishnan1975,Borkar1982,Kubrusly1977}), and from the perspective of optimal sensor placement, in which case one is interested in determining the locations of measurement sensors which are optimal with respect to some pre-determined criteria (e.g., \cite{Bensoussan1972,Burns2015,Curtain1978,Wu2016}).

\subsection{Literature Review}
\label{subsec:related_works}
\subsubsection{Parameter Estimation}
%As noted above, in many practical situations of interest, the parameters of this model are unknown, and must be inferred from the data, in either an online or an offline fashion. In fact, inferring the model parameters is often the primary problem of interest (e.g., \cite{Baras1987}). We are primarily concerned with online parameter estimation methods, which recursively estimate the unknown model parameters based on the continuous stream of observations. In comparison to classical methods, which process the observed data in a batch fashion, online methods perform inference in real time, can track changes in parameters over time, are more computationally efficient, and have significantly smaller storage requirements.

\color{black} The study of parameter estimation for infinite dimensional stochastic differential equations was initiated in the late 1960s (e.g., \cite{Kubrusly1977,Polis1982}), and has since been the subject of numerous papers and several monographs (e.g.,  \cite{Banks2012,Cialenco2018,Lototsky2009} and references therein). Although the majority of literature on this subject has been written for fully observed processes, several authors have also consider the `partially observed' case, in which observations of the infinite dimensional system are corrupted by some additional noise process \cite{Aihara1992,Aihara1988,Aihara1994,Aihara1991,Bagchi1981,Bagchi1984}.
%Instead, most of the literature on this subject has been written for discrete-time, partially observed processes (e.g., \cite{Cappe2005,Kantas2015}), or continuous-time, fully observed processes 
%(e.g., \cite{Bishwal2008,Borkar1982,Kutoyants1984,Levanony1994}). 
\color{black} Among the methods considered for this problem, those based on the maximum likelihood principle are perhaps the most ubiquitous. 
%, although several others have also been considered (e.g., \cite{Campillo1989,Dembo1986}). 
In the offline setting, maximum likelihood (ML) methods seek the value of $\theta$ that maximises the likelihood of the observations, or incomplete data log-likelihood, after some fixed time-period $T$. In particular, the maximum likelihood estimator (MLE) is defined as $\smash{\hat{\theta}_T =  \argmax_{\theta\in\Theta} \mathcal{L}_T(\theta,\boldsymbol{o})}$.
\iffalse
\begin{equation}
\theta^{*}_T =  \argmax_{\theta\in\Theta} \mathcal{L}_T(\theta,\boldsymbol{o}).
\end{equation} 
\fi
%%
%% either here on in section 3 %%
%%
The asymptotic properties of this estimator, including asymptotic consistency, asymptotic efficiency, and asymptotic normality, have been studied by various authors 
%for both finite-dimensional (e.g., \cite{Bagchi1975,Bagchi1980,Bagchi1981a,Balakrishnan1973,Tugnait1980,Tugnait1985}) and infinite-dimensional 
(e.g., \cite{Aihara1991,Bagchi1981,Bagchi1984,Balakrishnan1975,Kim1996,Koski1986,Mohapl1994}).
%[\textbf{To add}: offline parameter estimation for the stochastic advection diffusion equation.]

In this paper, we will primarily be concerned with online parameter estimation methods, which recursively estimate the unknown model parameters based on the continuous stream of observations (e.g., \cite{Baumeister1997}). In comparison to classical methods, which process the observed data in a batch fashion, online methods perform inference in real time, can track changes in parameters over time, are more computationally efficient, and have significantly smaller storage requirements. We will focus, in particular, on recursive maximum likelihood (RML) methods, which recursively seek the value of $\theta$ which maximises the asymptotic log-likelihood, defined as
$\smash{\tilde{\mathcal{L}}(\theta,\boldsymbol{o}) = \lim_{t\rightarrow\infty}\frac{1}{t}\mathcal{L}_t(\theta,\boldsymbol{o})}$. 

\color{black} Perhaps somewhat surprisingly, a rigorous treatment of this approach in the infinite-dimensional case remains an open problem. With this in mind, let us provide a brief overview of existing results in the finite-dimensional case (e.g., \cite{Gerencser1984,Gerencser2009,Surace2019}).  \color{black}
\iffalse
\begin{equation}
\tilde{\mathcal{L}}(\theta,\boldsymbol{o}) = \lim_{t\rightarrow\infty}\frac{1}{t}\mathcal{L}_t(\theta,\boldsymbol{o}). 
\end{equation}
\fi
\color{black}
The asymptotic properties of such methods for partially observed, discrete-time systems have been studied extensively (e.g., \cite{Collings1998,Doucet2003,Krishnamurthy2002,LeGland1995,LeGland1997,Poyiadjis2011,Schwartz2019,Tadic2010,Tadic2018}). In comparison, the partially observed, continuous-time case has received relatively little attention.\footnote{We do not attempt to review the fully observed, continuous-time case here. The interested reader is referred to \cite{Levanony1994,Sirignano2017a} and references therein.} This method was first proposed by Gerencser et al. \cite{Gerencser1984}, who derived a RML estimator for the parameters of a partially observed linear diffusion process using the It\^o-Venzel formula (e.g., \cite{Ventzel1965}), and provided an almost sure convergence (a.s.) result for this estimator without proof. %\footnote{Namely, that the estimator $\theta^{*}(t)\rightarrow\theta^{*}$ almost surely on the event $\smash{\Pi = \{\theta^{*}(t)<\infty \text{ for all } t\geq t_0\}}$, for some fixed, non-random initial time $t_0$.} 
This analysis was later extended in \cite{Gerencser2009}, in which the authors established the a.s. convergence of a modified version of the estimator in \cite{Gerencser1984}, which included an additional resetting mechanism. %\footnote{See also \cite{Gerencser2010}.}\textsuperscript{,}
%\footnote{The evolution equation for the estimators considered in \cite{Gerencser1984,Gerencser2009} include an additional, second order term. This arises when the It\^o-Venzel formula is applied to the score function (i.e., the derivative of the log-likelihood).}
The use of a continuous-time RML method for online parameter estimation was more recently revisited by Surace and Pfister \cite{Surace2019}, who derived an RML estimator for the parameters of a partially observed non-linear diffusion process, and established the a.s. convergence of this estimator under appropriate conditions on the latent state, the filter, and the filter derivative. This extended the results in \cite{Sirignano2017a} to the partially observed setting. 
%In this paper, we extend the recursive maximum likelihood estimator proposed by these authors to the infinite-dimensional setting. 
The use of a continuous-time RML method for non-linear partially observed diffusion processes was also considered in  \cite{Ljung1987,Moura1986}. In these papers, however, in addition to the model parameters, the hidden state was estimated via maximum likelihood, rather than the usual filtering paradigm.

\subsubsection{Optimal Sensor Placement}

%Aside from parameter estimation, one of the most important problems associated with state estimation in partially observed stochastic dynamical systems is that of optimal sensor placement. Indeed, the problem of determining the optimal placement of a finite number of measurement sensors, possibly subject to constraints, is relevant to a huge number of applications in engineering and the applied sciences, including meteorology, environmental monitoring, and fluid dynamics [\textbf{refs}]. In such applications, the process of interest, even if defined continuously over space and in time, can only be measured at a finite number of spatial locations and time points [refs]. Moreover, the spatial and temporal density of observations is generally very low, due either to prohibitive expense %(i.e., the sensors are expensive, they are are expensive to place, or both), 
%or geographical inaccessibility. % (i.e., the sensors cannot be placed in particular locations). 
%Furthermore, measurements at certain points in the spatio-temporal domain may yield more information about the system than measurements at other points, due to spatio-temporal correlations in the signal. Thus, to a greater or lesser extent, the accuracy of the estimate of the latent state is dependent on the number and locations of the measurement sensors.

The problem of optimal sensor placement for state estimation in infinite dimensional, partially observed, linear stochastic dynamical systems, has been studied by a large number of authors, and in a wide variety of contexts. Arguably the first mathematically rigorous treatment of this problem was provided by Bensoussan \cite{Bensoussan1971,Bensoussan1972}, who formulated it as an application of optimal control on the infinite dimensional Ricatti equation governing the covariance of the optimal filter. This extended the results in \cite{Athans1972}, in which similar conditions were obtained in the finite dimensional case. For a review of other early results on the sensor placement problem, we refer to \cite{Kubrusly1985}. Under this framework, sensor locations are treated as control variables, and the optimal sensor locations are defined as the minima of a suitable objective function, typically defined as the trace of the covariance at some finite time (e.g., \cite{Bensoussan1971,Bensoussan1972,Cannon1971,Chen1975,Curtain1978,Omatu1978,Wu2016}), or the integral of the trace of the covariance over some finite time interval (e.g. \cite{Burns2015,Chen1975,Herring1974,Korbicz1994}). Alternatively, one can define cost functions with respect to parameter estimation considerations (e.g., \cite{Patan2012,Ucinski2005}). In any case, one then has $\smash{\hat{\boldsymbol{o}}_T = \argmin_{\boldsymbol{o}\in\mathcal{O}}\mathcal{J}_T(\theta,\boldsymbol{o})}$.
\iffalse
\begin{equation}
\hat{\boldsymbol{o}}_T = \argmin_{\boldsymbol{o}\in\mathcal{O}}\mathcal{J}_T(\theta,\boldsymbol{o}).
\end{equation}
\fi

 Recently, and in the spirit of Bensoussan's original approach, Burns et al. \cite{Burns2015a,Burns2011,Burns2015,Burns2009,Hintermuller2017,Rautenberg2010} have provided a rigorous general framework for determining optimal location and trajectories of sensor networks for linear stochastic distributed parameter systems. In particular, in \cite{Burns2015}, the optimisation problem is precisely formulated as the minimisation of a functional involving the trace of a solution to the integral Ricatti equation, with constraints given by the allowed trajectories of the sensor network. The existence of Bochner integrable solutions to this equation, and thus the existence of optimal sensor locations, is established in \cite{Burns2015a,Burns2011}. Meanwhile, in \cite{Burns2015}, a Galerkin type numerical scheme for the finite dimensional approximation of these solutions is proposed, for which convergence is proved in $\smash{L_p}$ norm. %In these authors' most recent contribution, the existing objective functional is combined with a `worst case scenario' functional, which involves a further optimisation problem for directional sensitivities over a set of admissible perturbations \cite{Hintermuller2017}. 

Several authors have also considered optimal sensor placement with respect to asymptotic versions of these objective functions (e.g., \cite{Aidarous1978,Amouroux1978,Rafajlowicz1984,Tang2017,Yu1973,Zhang2018}). In this case, the optimal sensor placements are obtained, possibly recursively, as the minima of the asymptotic objective function $\smash{\tilde{\mathcal{J}}(\theta,\boldsymbol{o}) = \frac{1}{t}\lim_{t\rightarrow\infty}\mathcal{J}_t(\theta,\boldsymbol{o})}$. 
\iffalse
\begin{equation}
\tilde{\mathcal{J}}(\theta,\boldsymbol{o}) = \frac{1}{t}\lim_{t\rightarrow\infty}\mathcal{J}_t(\theta,\boldsymbol{o}).
\end{equation}
\fi
Most recently, Zhang and Morris \cite{Zhang2018} consider minimisation of the trace of the (mild) solution of the infinite dimensional algebraic Ricatti equation as a sensor placement criterion. In particular, they prove that the trace of the solution to this equation minimises the steady-state error variance, and thus represents an appropriate design objective. They also establish the existence of solutions (i.e., optimal sensor locations) to the corresponding optimal sensor placement problem, as well as the convergence of appropriate finite-dimensional approximations to these solutions, using results previously obtained for optimal actuator locations in \cite{Morris2011}. %Tang and Morris \cite{Tang2017} have since extended this approach to the case that the form of the observation operator is not assumed fixed.%; i.e., the shape as well as the location of the sensors is a design variable.

\subsection{Contributions}
In this paper, we present a principled method for performing joint online parameter estimation and optimal sensor placement for a partially observed infinite-dimensional linear diffusion process.
%This represents a significant departure from the existing literature, in which these two problems have, until now, been studied separately. There is strong motivation for this combined approach. Indeed, in the vast majority of practical applications, both of these problems are relevant. It would thus be highly convenient to solve them simultaneously, and, if possible, in an online fashion. Moreover, they are often interdependent, in the sense that the optimal sensor placement can vary significantly according to the current parameter estimate (see Section \ref{sec:opt_sens}). Thus, tackling them together can result in significant performance improvements (see Figure \ref{fig0}).  
%This represents the main novelty of our paper. %That is, as outlined below, to simultaneously and recursively maximise the asymptotic log-likelihood with respect to the model parameters, and minimise the asymptotic objective function with respect to the sensor locations. 
%Our proposed approach, extending the work of \cite{Surace2019,Zhang2018}, among others, is based on the use of a two-timescale stochastic gradient descent algorithm in continuous time. 
\color{black} We formulate this as a bilevel optimisation problem, in which the objective is to obtain estimates $\hat{\theta}\in\Theta$ and $\hat{\boldsymbol{o}}(\hat{\theta})\in\Pi^{n_{y}}$ 
%which simultaneously maximise the log-likelihood of the observations and minimise a approriately chosen sensor placement objective function. 
such that
%\footnote{Depending on our primary objective, we may also consider the dual bilevel optimisation problem, in which case the objective is to obtain the values of $\theta^{*}\in\Theta$, $\hat{\boldsymbol{o}}\in\Pi^{n_{y}}$ such that 
%\begin{alignat}{2}
%\hat{\boldsymbol{o}}&= \argmin_{\boldsymbol{o}\in\Pi^{n_{y}}} \tilde{\mathcal{J}}\big(\argmax_{\theta\in\Theta}\tilde{\mathcal{L}}(\theta,\boldsymbol{o}),\boldsymbol{o}\big)~~~,~~~\hat{{\theta}}&&= \argmax_{\theta\in\Theta} \tilde{\mathcal{L}}\big(\theta,\hat{\boldsymbol{o}}\big).
%\end{alignat}} 
\begin{alignat}{2}
\hat{\theta}&= \argmax_{\theta\in\Theta} \tilde{\mathcal{L}}\big({\theta},\hat{\boldsymbol{o}}(\theta)\big)~~~,~~~\hat{\boldsymbol{o}}(\theta)&&= \argmin_{\boldsymbol{o}\in\Pi^{n_{y}}} \tilde{\mathcal{J}}\big(\theta,\boldsymbol{o}\big). 
\vspace{-1mm}
\end{alignat}
\\[-2mm]
\color{black}
where, as previously, $\tilde{\mathcal{L}}(\cdot)$ denotes the asymptotic log-likelihood of the observations, and $\tilde{\mathcal{J}}(\cdot)$ denotes an asymptotic sensor placement objective function. %(see Section \ref{subsec:related_works}). 
It should be noted that the choice of $\tilde{\mathcal{L}}(\cdot)$ rather than $\tilde{\mathcal{J}}(\cdot)$ as the upper-level objective function here is arbitrary. In particular, our methodology remains valid if we consider the bilevel optimisation problem in which the upper-level and lower-level objective functions are reversed. The upper-level objective function can thus be chosen on a case-by-case basis, based on which criterion (parameter estimation or optimal sensor placement) one wishes to prioritise. 

To solve this problem, we propose a novel continuous-time, two-timescale stochastic gradient descent algorithm. This algorithm can be seen as a formal extension of the authors' previous work in \cite{Sharrock2020a} to the setting in which the latent state is infinite-dimensional. We establish, using the theoretical results in \cite{Sharrock2020a}, almost sure convergence of the online parameter estimates and recursive optimal sensor placements generated by a suitable finite dimensional approximation of this algorithm to the stationary points of the asymptotic log-likelihood and the asymptotic filter covariance, respectively.\footnote{Under reasonable additional assumptions, it is possible to extend this analysis to show that the online parameter estimates are local maxima of the asymptotic log-likelihood, and the recursive optimal sensor placements are local minima of the asymptotic filter covariance \cite{Sharrock2020a}.} We then provide several detailed numerical case studies, illustrating the performance of this method for the partially observed stochastic advection diffusion equation. Our numerical results indicate that the algorithm is highly effective, and can be applied in cases involving static and dynamic parameters, multiple noise and bias parameters, different specifications of the sensor placement objective, and different specifications of the upper and lower-level objective functions.

\subsection{Paper Organisation}
The remainder of this paper is organised as follows. \color{black} In Section \ref{sec:equation}, we precisely formulate the stochastic advection-diffusion equation as a functional stochastic differential equation on an appropriate separable Hilbert space. In Section \ref{sec:param}, we present the two-timescale stochastic gradient descent algorithm for joint parameter estimation and optimal sensor placement. We present our methodology in a generic abstract framework, and thus in principle it could be applied to partially observed stochastic partial differential equations other than the stochastic advection-diffusion equation. \color{black} In Section \ref{sec:numerics}, we provide several numerical examples illustrating the application of the proposed methodology. Finally, in Section \ref{sec:conclusions}, we provide some concluding remarks.

\section{The Partially Observed Stochastic Advection-Diffusion Equation}
\label{sec:equation}

In this section, we provide some background on the stochastic advection-diffusion equation. In particular, we outline how this SPDE can be defined as a functional evolution equation on an appropriate separable Hilbert space (see also, e.g., \cite{Banks1984,Burns2015,Wu2016}). We restrict our attention to the case of periodic boundary conditions following the treatment in \cite{Sigrist2015}. \color{black} This choice is largely motivated by expositional convenience, and the ability to efficiently perform numerical approximations using the Fast Fourier Transform (FFT). %We will also assume that the deterministic forcing $f(\boldsymbol{x},t)=0$. This is largely motivated by expositional and numerical convenience. 
%We should emphasise, however, that 
We should emphasise, however, that the subsequent joint online parameter estimation and optimal sensor placement algorithm in Section \ref{sec:param} is generic, and does not require this assumption. \color{black}

\subsection{Preliminaries}
We will suppose that the region of interest is the unit torus $\Pi := [0,1]^2$, with $\boldsymbol{x}=(x_1,x_2)^T\in\Pi$ a point on this space. We are interested in a space-time varying scalar field $u:\Pi\times [0,\infty)\rightarrow\mathbb{R}$, and will write $u(\boldsymbol{x},t)$ to denote the value of the field at spatial location $\boldsymbol{x}\in\Pi$ and time $t\in[0,\infty)$. We will assume that this field satisfies periodic boundary conditions. 
%, and thus
%, namely
%\begin{align}
%u\left((x,0)^T,t\right) = u\left((x,1)^T,t\right) ~\forall x\in[0,1]&\\% \begin{pmatrix} x \\ 0 \end{pmatrix}, t\right)=u\left(\begin{pmatrix} x \\ 1 \end{pmatrix}, t\right) ~\forall x\in[0,1]& \\
%u\left((0,y)^T,t\right) = u\left((1,y)^T,t\right) ~\forall y\in[0,1]&. %u\left(\begin{pmatrix} 0 \\ y \end{pmatrix}, t\right)=u\left(\begin{pmatrix} 1 \\ y \end{pmatrix}, t\right) ~\forall y\in[0,1]&.
%\end{align}
The function space of interest is given by $\mathcal{H} = L^{\text{per.}}_2(\Pi)$, the space of periodic square-integrable functions on $\Pi=[0,1]^2$. 

\color{black} It is natural to work with the Fourier characterisation of this space. In particular, suppose we write %$\{\phi_{\boldsymbol{k}}\}_{\boldsymbol{k}\in\mathbb{Z}^2}$  
$\smash{\{\phi_{\boldsymbol{k}}\}_{\boldsymbol{k}\in\mathbb{Z}^2}}$ for the set of orthonormal Fourier basis functions for $\mathcal{H}$, namely, $\smash{\phi_{\boldsymbol{k}}(\boldsymbol{x}) = \exp(i\boldsymbol{k}^T \boldsymbol{x})}$. We can then write \color{black}
\color{black}
\begin{equation}
\mathcal{H}= \bigg\{\varphi : \varphi(\boldsymbol{x})= \sum_{\mathbf{k}\in\mathbb{Z}^2} \alpha_{\boldsymbol{k}} \phi_{\boldsymbol{k}}(\boldsymbol{x}) :\alpha_{-\boldsymbol{k}}=\overline{\alpha_{\boldsymbol{k}}}, \sum_{\mathbf{k}\in\mathbb{Z}^2}\alpha^2_{\boldsymbol{k}}(t)<\infty\bigg\}. \label{fourier_basis}
\end{equation}
%\begin{equation}
%\phi_{\boldsymbol{k}}(\boldsymbol{x})=\exp\left[2\pi i \boldsymbol{k}^T\boldsymbol{x}\right]~,~~~\boldsymbol{k}\in\mathbb{Z}^2/\{\boldsymbol{0}\}.
%\end{equation} 
%\subsection{The Deterministic Advection-Diffusion Equation}
\subsection{The Signal Equation}
\color{black}
Using standard results on infinite dimensional systems (e.g., \cite{Curtain1995,Omatu1989}), we can formulate the stochastic advection-diffusion partial differential equation \eqref{SPDE} as a functional evolution equation on $\mathcal{H}$. Let $u(t)=u(\cdot,t) = \{u(\boldsymbol{x},t): \boldsymbol{x}\in\Pi \} \in\mathcal{H}$ denote the state of the infinite-dimensional system. Then we can write
% , of the abstract equation
\begin{equation}
\mathrm{d}u(t) = \mathcal{A}(\theta) u(t)\mathrm{d}t + \mathcal{B}\mathrm{d}v_{\theta}(t), ~~~u(0) = u_0 \in\mathcal{H} \label{eq_sig}%\sim\mathcal{N}(u_0(\theta),\Sigma_0(\theta)), \label{eq205}
\end{equation}
% (see, e.g., \cite{Pazy1983}), 
where $\theta\in\Theta\subset\mathbb{R}^{n_{\theta}}$ is an $n_{\theta}$-dimensional parameter, $\mathcal{A}(\theta)$ and $\mathcal{B}$ are abstract operators to be defined below, $v_{\theta}(t) = v_{\theta}(\cdot,t) = \{v_{\theta}(\boldsymbol{x},t):\boldsymbol{x}\in\Pi\}$ is a space-time Brownian motion, and $u_0$ is a $\mathcal{H}$-valued Gaussian random variable, which is independent of $v_{\theta}$. 
We are interested in weak solutions of this equation, to be understood path-wise on the complete probability space $(\Omega,\mathcal{F},\mathbb{P})$. 
%We are interested in weak solutions of this equation, 

The terms in this equation are defined explicitly as follows. Firstly, $\mathcal{A}(\theta):\mathcal{D}(\mathcal{A}(\theta))\rightarrow\mathcal{H}$ is the two-dimensional advection diffusion operator, defined according to
\begin{equation}
\mathcal{A}(\theta) \varphi = -\sum_{i=1}^2 \mu_{i}\frac{\partial \varphi}{\partial x_i}+\sum_{i,j=1}^{2}\frac{\partial}{\partial x_i} \left(\Sigma_{i,j} \frac{\partial \varphi}{\partial x_j}  \right) -\zeta\varphi%+ f(t)
~,~~~\varphi\in\mathcal{D}(\mathcal{A}(\theta))
\label{A_operator}
\end{equation} 
where $\smash{\mathcal{D}(\mathcal{A}(\theta)) = \{\varphi\in\mathcal{H}: \tfrac{\partial \varphi}{\partial x_i},\tfrac{\partial^2\varphi}{\partial x_i\partial x_j}\in L_2(\Pi)~, i,j=1,2\}}$, and $(\mu_i)_{i=1,2}$, $(\Sigma_{i,j})_{i,j=1,2}$, $\zeta$ are the parameters defined in Section \ref{sec:introduction}. We remark that this operator generates an exponentially stable $C_0$-semigroup over $\mathcal{H}$ (e.g., \cite{Burns2015,Sell2002}).\footnote{For ease of exposition, we have assumed here that $\mathcal{A}(\theta)$ is time-invariant. It is straightforward, however, to extend all of the results in this paper to the case of a time-dependent operator $\mathcal{A}(\theta,t)$. In this case, we would also require some standard assumptions on the regularity of the map $t\rightarrow\mathcal{A}(\theta,t)$ (e.g., \cite{Lions1971,Pazy1983,Tanabe1960}).}
%\begin{equation}
%\mathcal{D}(\mathcal{A}(\theta)) = \{\varphi\in\mathcal{H}: \tfrac{\partial \varphi}{\partial x_i},\tfrac{\partial^2\varphi}{\partial x_i\partial x_j}\in L_2^{\text{periodic}}(\Pi)~, i,j=1,2\}
%\end{equation}
Meanwhile, $\mathcal{B}:\mathcal{H}\rightarrow\mathcal{H}$ is a disturbance operator defined, for some $b\in\mathcal{H}$, via
\begin{equation}
\mathcal{B} \varphi = b(\boldsymbol{x}) \varphi~,~~~\varphi\in\mathcal{H}.
\end{equation}
%\subsection{The Noise Process}
\iffalse
We have 
\begin{equation}
\mathcal{A}(\theta) \phi_{\boldsymbol{k}}(\boldsymbol{s}) = -\left(\sum_{i=1}^2 i k_i\mu_i(\boldsymbol{x})- \sum_{i,j=1}^2 k_{i}\Sigma_{i,j}(\boldsymbol{x})k_{j}-\zeta(\boldsymbol{x})\right)\phi_{\boldsymbol{k}}(\boldsymbol{s}) = -\lambda_{\boldsymbol{k}}(\theta)\phi_{\boldsymbol{k}}(\boldsymbol{s})
\end{equation}
with 
\begin{equation}
\mathrm{Re}\left(\lambda(\theta)\right) = -\sum_{i,j=1}^2 k_{i}\Sigma_{i,j}(\boldsymbol{x})(\boldsymbol{x}) k_j - \zeta(\boldsymbol{x})<0.
\end{equation}
\fi
Finally, $v_{\theta}(t)$ is a $\mathcal{H}$-valued Wiener process with incremental covariance operator $\smash{\mathcal{Q}}(\theta):\mathcal{H}\rightarrow\mathcal{H}$. We will assume that this covariance operator satisfies $\smash{\mathcal{Q}(\theta)\phi_{\boldsymbol{k}}= \eta_{\boldsymbol{k}}^2(\theta)\phi_{\boldsymbol{k}}}$ for all $\boldsymbol{k}\in\mathbb{Z}^2$, for some bounded sequence of real numbers $\smash{\{\eta^2_{\boldsymbol{k}}(\theta)\}_{\boldsymbol{k}\in\mathbb{Z}^2}}$ satisfying $\eta_{-{\boldsymbol{k}}}(\theta)=\eta_{\boldsymbol{k}}(\theta)$ and $\sum_{\boldsymbol{k}}\eta_{\boldsymbol{k}}^2(\theta)<\infty$. We thus work with a diagonal covariance operator with respect to the Fourier basis, although other choices could easily be considered. It follows from standard results (see, e.g., \cite{Curtain1978a,DaPrato2014}) that 
\begin{equation}
v_{\theta}(\boldsymbol{x},t) := \sum_{\boldsymbol{k}\in\mathbb{Z}^2}\eta_{\boldsymbol{k}}(\theta)\phi_{\boldsymbol{k}}(\boldsymbol{x}) z_{\boldsymbol{k}}(t), %\label{eq203}
\end{equation}
where $\smash{\{z_{\boldsymbol{k}}(t)\}_{\boldsymbol{k}\in\mathbb{Z}}}$ are a set of suitably defined independent Brownian motions (see, e.g., \cite{Llopis2017} for a precise definition). Following \cite{Lindgren2011,Sigrist2015}, we will assume that $\smash{\{\eta^2_{\boldsymbol{k}}(\theta)\}_{\boldsymbol{k}\in\mathbb{Z}^2}}$ are given by 
%the $\mathcal{Q}(\theta)$-Weiner process $v_{\theta}=\{v_{\theta}(t)\}_{t\geq 0}$ is defined according to
\begin{equation}
\eta_{\boldsymbol{k}}(\theta) = \frac{\sigma}{2\pi}\left(\boldsymbol{k}^T\boldsymbol{k}+\frac{1}{\rho_0^2}\right)^{-\nu},
\end{equation}
where $\sigma>0$ is a {marginal variance} parameter, $\rho_0>0$ is a {spatial range} parameter, and $\nu>0$ is a smoothness parameter. This yields a noise process with the Mat\'ern covariance function in space, which is perhaps the most widely used covariance function in spatial statistics \cite{Cressie1993,Handcock1993,Stein1999}. In many applications, the smoothness parameter is not identifiable \cite{Lindgren2011,Sigrist2015}. Thus, as in \cite{Sigrist2015}, we will assume that $\nu=1$. In principle, however, other values of $\nu$ could also be considered. This particular choice corresponds to the so-called Whittle covariance function in space, which can arguably be regarded as `the elementary correlation in two dimensions' \cite{Whittle1954}. \color{black} 

%\subsection{The Stochastic Advection Diffusion Equation}
\iffalse
We can thus now define the stochastic advection-diffusion equation as a functional stochastic differential equation on $\mathcal{H}$, namely, 
\begin{equation}
\mathrm{d}u(t) = \mathcal{A}(\theta) u(t)\mathrm{d}t + \mathcal{B}\mathrm{d}v_{\theta}(t), ~~~u(0) = u_0 \label{eq_sig}%\sim\mathcal{N}(u_0(\theta),\Sigma_0(\theta)), \label{eq205}
\end{equation}
where $\mathcal{A}(\theta):\mathcal{D}(\mathcal{A}(\theta))\rightarrow \mathcal{H}$ is the stochastic advection-diffusion operator defined by (\ref{A_operator}), and $\mathcal{B}\in \mathcal{L}_2(\mathcal{H})$ is a spatial weighting operator, defined according to $\mathcal{B}\varphi(\boldsymbol{x}) = b(\boldsymbol{x})\varphi(\boldsymbol{x})$ for some $b\in\mathcal{H}$. 
%\begin{equation}
%\mathcal{B}\varphi = b(\boldsymbol{x})\varphi~,~~~\varphi\in\mathcal{H}.
%\end{equation} 
\fi

\subsection{The Spectral Signal Equation}
Using the Fourier characterisation, it is possible to write \color{black} $u(\boldsymbol{x},t)$ \color{black} as
\begin{equation}
\color{black} u(\boldsymbol{x},t) = \sum_{\textbf{k}\in\mathbb{Z}^2} \alpha_{\boldsymbol{k}}(t)\phi_{\boldsymbol{k}}(\boldsymbol{x}) \color{black}
%~,~~~\alpha_{-\boldsymbol{k}}(t)\equiv -\overline{\alpha}_{\boldsymbol{k}}(t),
~,~~~\alpha_{\boldsymbol{k}}(t) = \langle u(t),\phi_{\boldsymbol{k}} \rangle = \int_{\Pi} \color{black} u(\boldsymbol{x},t) \color{black} \overline{\phi_{\boldsymbol{k}}}(\boldsymbol{x})\mathrm{d}\boldsymbol{x}. \label{fourier_solution}
\end{equation}
\iffalse
It follows that 
\begin{equation}
\mathrm{d}\left[\sum_{\textbf{j}\in\mathbb{Z}^2} \alpha_{\boldsymbol{j}}(t)\phi_{\boldsymbol{j}}\right] = \mathcal{A}(\theta)\left[\sum_{\textbf{j}\in\mathbb{Z}^2} \alpha_{\boldsymbol{j}}(t)\phi_{\boldsymbol{j}}\right]\mathrm{d}t + \mathcal{B}\left[\sum_{\textbf{j}\in\mathbb{Z}^2} \eta_{\boldsymbol{j}}(\theta)\phi_{\boldsymbol{j}}\mathrm{d}z_{\boldsymbol{j}}(t)\right]
\end{equation}
or, simplifying,
\begin{equation}
\sum_{\textbf{j}\in\mathbb{Z}^2} \phi_{\boldsymbol{j}} \mathrm{d}\alpha_{\boldsymbol{j}}(t)= \sum_{\textbf{j}\in\mathbb{Z}^2}\mathcal{A}(\theta)\phi_{\boldsymbol{j}} \alpha_{\boldsymbol{j}}(t)\mathrm{d}t + \sum_{\textbf{j}\in\mathbb{Z}^2} \mathcal{B} \phi_{\boldsymbol{j}} \eta_{\boldsymbol{j}}(\theta)\mathrm{d}z_{\boldsymbol{j}}(t)
\end{equation}
\fi
%where the coefficients $\{\alpha_{\boldsymbol{k}}(t)\}_{\boldsymbol{k}\in\mathbb{Z}^2/\{\boldsymbol{0}\}}$ are given by
%\begin{equation}
%\alpha_{\boldsymbol{k}}(t) = \langle u(t),\phi_{\boldsymbol{k}} \rangle = \int_{\Pi} u(t) \overline{\phi_{\boldsymbol{k}}}(\boldsymbol{x})\mathrm{d}\boldsymbol{x}.%~~~,~~~\alpha_{-\boldsymbol{k}}(t)\equiv -\overline{\alpha}_{\boldsymbol{k}}(t),
%\end{equation}
%where the %It is thus equivalent to consider the parameterisation of $u(t)$ via the set of 
 %coefficients %$\smash{\{
%$\alpha_{\boldsymbol{k}}(t)$
 %\}_{\boldsymbol{k}\in\mathbb{Z}^2/\{\boldsymbol{0}\}}}$
%. 
\color{black} It is thus equivalent to consider the parametrisation of $u(t)$ via the set of 
 Fourier coefficients $\smash{\{
\alpha_{\boldsymbol{k}}(t)
 \}_{\boldsymbol{k}\in\mathbb{Z}^2}}$. \color{black} Taking the inner product of both sides of the signal equation with $\phi_{\boldsymbol{k}}$, 
 %\footnote{Since $\smash{u=\{u(t)\}_{t\geq 0}}$ is a real field, $\smash{\alpha_{-\boldsymbol{k}}(t) = -\bar{\alpha}_{\boldsymbol{k}}(t)}$ for $\smash{\boldsymbol{k}\in\mathbb{Z}^2_{\uparrow}}$. It is thus sufficient to consider $\smash{\boldsymbol{k}\in\mathbb{Z}_{\uparrow}^2}$.} 
we see that the $\alpha_{\boldsymbol{k}}$'s 
obey the following infinite dimensional stochastic differential equation
%, obtained upon %substituting \eqref{fourier_solution} into the signal equation, and 
%taking the inner product of both sides of the signal equation with $\phi_{\boldsymbol{k}}$,
%\footnote{We remark that, since $\smash{u=\{u(t)\}_{t\geq 0}}$ is a real field, we have $\smash{\alpha_{-\boldsymbol{k}}(t) = -\bar{\alpha}_{\boldsymbol{k}}(t)}$ for $\smash{\boldsymbol{k}\in\mathbb{Z}^2_{\uparrow}}$. It is thus sufficient to consider $\smash{\boldsymbol{k}\in\mathbb{Z}_{\uparrow}^2}$.}
\color{black}
\begin{equation}
\mathrm{d} \alpha_{\boldsymbol{k}}(t) = 
\sum_{\boldsymbol{j}\in\mathbb{Z}^2}  \lambda_{\boldsymbol{j},\boldsymbol{k}}(\theta)\alpha_{\boldsymbol{j}}(t) \mathrm{d}t+\sum_{\boldsymbol{j}\in\mathbb{Z}^2}  \xi_{\boldsymbol{j},\boldsymbol{k}} \eta_{\boldsymbol{j}}(\theta)\mathrm{d}z_{\boldsymbol{j}}(t)~,~~~\boldsymbol{k}\in\mathbb{Z}^2, \label{inf_dim_fourier}
\end{equation}
\color{black}
where \color{black} $\smash{\lambda_{\boldsymbol{j},\boldsymbol{k}}(\theta)=\langle  \mathcal{A}(\theta)\phi_{\boldsymbol{j}},\phi_{\boldsymbol{k}}\rangle = -\left(i\boldsymbol{j}^T\boldsymbol{\mu}+ \boldsymbol{j}^T\Sigma\boldsymbol{j}+\zeta\right)\delta_{\boldsymbol{j},\boldsymbol{k}}}$ and $\smash{\xi_{\boldsymbol{j},\boldsymbol{k}}=\langle  \mathcal{B}\phi_{\boldsymbol{j}},\phi_{\boldsymbol{k}}\rangle}$.
\iffalse
 and we have defined
\begin{align}
\lambda_{\boldsymbol{k},\boldsymbol{j}}(\theta)&=\langle  \mathcal{A}(\theta)\phi_{\boldsymbol{j}},\phi_{\boldsymbol{k}}\rangle = -\int_{\Pi}\left(i\boldsymbol{j}^T\boldsymbol{\mu}(\boldsymbol{x})- \boldsymbol{j}^T\Sigma(\boldsymbol{x})\boldsymbol{j}-\zeta(\boldsymbol{x})\right)\phi_{\boldsymbol{j}}(\boldsymbol{x})\overline{\phi_{\boldsymbol{k}}}(\boldsymbol{x})\mathrm{d}\boldsymbol{x}, \hspace{-2mm} \\
\xi_{\boldsymbol{k},\boldsymbol{j}}&=\langle  \mathcal{B}\phi_{\boldsymbol{j}},\phi_{\boldsymbol{k}}\rangle=\int_{\Pi}b(\boldsymbol{x})\phi_{\boldsymbol{j}}(\boldsymbol{x})\overline{\phi_{\boldsymbol{k}}}(\boldsymbol{x})\mathrm{d}\boldsymbol{x}.
\end{align}
\fi
We will occasionally refer to this as the `spectral' signal equation. In our numerics, we will often assume that $\mathcal{B}$ is given by the identity operator, in which case we also have $\xi_{\boldsymbol{j},\boldsymbol{k}} = \delta_{\boldsymbol{j},\boldsymbol{k}}$. In this case, the spectral signal equation diagonalises completely; that is, the $\alpha_{\boldsymbol{k}}$'s evolve independently of one another. \color{black}
%We will sometimes refer to this as the `spectral' signal equation. 
This parametrisation of the signal process is highly convenient, as it allows us to perform inference on a vector whose coordinates evolve according to a SDE, even if this vector happens to have infinite length.

\subsection{The Observation Equation} 
We assume that the signal process cannot be observed directly, \color{black} but that instead we obtain a continuous sequence of noisy observations $y = \{y(t)\}_{t\geq 0}$, taking values in $\mathbb{R}^{n_y}$, via a set of $n_y$ sensors located at $\smash{\boldsymbol{o}=\{\boldsymbol{o}_i\}_{i=1}^{n_{y}}\in\Pi^{n_{y}}}$. In particular, the observations are generated \color{black} according to \color{black}
\begin{equation}
\mathrm{d}y(t) = \mathcal{C}(\theta,\boldsymbol{o})u(t)\mathrm{d}t+ \mathrm{d}w_{\boldsymbol{o}}(t)~,~~~y(0)=0, \label{obs_eq}
\end{equation}
where $\mathcal{C}(\theta,\boldsymbol{o}):\mathcal{H}\rightarrow\mathbb{R}^{n_y}$ is a bounded linear operator to be specified below, and $w_{\boldsymbol{o}}(t)$ is a $\mathbb{R}^{n_y}$ valued Wiener process with incremental covariance $\mathcal{R}(\boldsymbol{o})\in\mathbb{R}^{n_{y}\times n_{y}}$, %\footnote{Thus, in particular,  $\mathcal{R}$ is symmetric and positive-definite.}
which is independent of both $v_{\theta}$ and $u_0$.
While the use of a linear observation equation is somewhat restrictive, it does encompass most typical observation schemes used in practice. 
%We will further restrict our attention to the scenario in which we obtain noisy, possibly biased observations of the field at a set of $n_{y}$ locations $\boldsymbol{o}=\{\boldsymbol{o}_i\}_{i=1}^{n_{y}}\in\Pi^{n_y}$, where $\boldsymbol{o}_i\in\Pi$ for all $i=1,\dots,n_y$. %This is often referred to as Eulerian data assimilation.
 
\color{black} We will further assume that each sensor provides a noisy, possibly biased, average of the latent signal around its current location, in which \color{black} case the observation operator is defined by 
\begin{align}
\mathcal{C}(\theta,\boldsymbol{o})\varphi= \begin{pmatrix} 
 \mathcal{C}_1(\theta,\boldsymbol{o}_1)\varphi  \\ \vdots \\ \mathcal{C}_{n_{y}}(\theta,\boldsymbol{o}_{n_y})\varphi %\right)^T
\end{pmatrix}
~,~~\mathcal{C}_i(\theta,\boldsymbol{o}_i)\varphi = \frac{\int_{\Pi} \mathcal{K}_{\boldsymbol{o}_i}(\boldsymbol{x})\varphi(\boldsymbol{x})\mathrm{d}\boldsymbol{x}}{\int_{\Pi} \mathcal{K}_{\boldsymbol{o}_i}(\boldsymbol{x})\mathrm{d}\boldsymbol{x}}+\beta_i~,~~~\varphi\in\mathcal{H}
\end{align}
%where $\mathcal{C}_i(t):\mathcal{H}\rightarrow\mathbb{R}$, $i=1,\dots,n_{y}$, are defined according to
%\begin{equation}
%\mathcal{C}_i(t)\varphi = \frac{\int_{\Pi} \mathcal{K}_{\boldsymbol{o}_i(t)}(\boldsymbol{x})\varphi(\boldsymbol{x})\mathrm{d}\boldsymbol{x}}{\int_{\Pi} \mathcal{K}_{\boldsymbol{o}_i(t)}(\boldsymbol{x})\mathrm{d}s},~~~\varphi\in\mathcal{H}
%\end{equation}
%for some suitably chosen weighting functions $\mathcal{K}_{\boldsymbol{o}_i(t)}\in L_2(\Pi)$, which decrease as $|\boldsymbol{x}-\boldsymbol{o}_i|$ increases. 
where $K_{\boldsymbol{o}_i}:\Pi\rightarrow\Pi$ are suitably chosen weighting functions, which decrease as $|\boldsymbol{x}-\boldsymbol{o}_i|$ increases, and $\beta_i\in\mathbb{R}$ are bias terms. \color{black} In our numerics, \color{black} we restrict our attention to the case in which each sensor provides an \color{black} unweighted \color{black} average of the latent signal process within a \color{black} small \color{black} fixed region of its current location (e.g., \cite{Burns2015,Llopis2017}). %\footnote{An alternative, and even more simplistic assumption, is that each of the sensors provides a point measurement of the latent signal process at its current location [\textbf{refs}]. %This implies, in particular, that 
%\begin{equation}
%\mathcal{C}_i(t)\varphi = \varphi(\boldsymbol{o}_i(t)),~~~\varphi\in\mathcal{H}
%\end{equation}
%This corresponds to the choice $\mathcal{K}_{\boldsymbol{o}_i(t)}(\boldsymbol{x}) = \delta\left[\boldsymbol{x}-\boldsymbol{o}_i(t)\right]$, 
%\begin{equation}
%\mathcal{K}_{\boldsymbol{o}_i(t)}(\boldsymbol{x}) = \delta\left[\boldsymbol{x}-\boldsymbol{o}_i(t)\right],
%\end{equation}
%where $\delta:C_0^{\infty}(\Pi)\rightarrow\mathbb{R}$ denotes the usual Dirac delta function.}
This corresponds to the choice 
\iffalse
$\smash{\mathcal{K}_{\boldsymbol{o}_i}(\boldsymbol{x}) = \mathds{1}\left\{\boldsymbol{x}\in\Pi:|\boldsymbol{x}-\boldsymbol{o}_i|\leq r \right\}}$, for some $r>0$. 
\fi
%\iffalse
\begin{equation}
\mathcal{K}_{\boldsymbol{o}_i}(\boldsymbol{x}) = \mathds{1}\left\{\boldsymbol{x}\in\Pi:|\boldsymbol{x}-\boldsymbol{o}_i|\leq r \right\},~~~r>0.
\end{equation}
%\fi
For simplicity, we will also assume that the sensors are {independent}, in which case the covariance matrix $\mathcal{R}(\boldsymbol{o})$ reduces to a diagonal matrix; and that the sensors can be categorised into $p_1$ distinct `noise classes', and into $p_2$ distinct `bias' classes, where $1\leq p_1,p_2\leq n_y$. By this, we mean that all observations generated by sensors belonging to a particular class have the same variance (or the same bias).

\newpage
\section{Joint Online Parameter Estimation and Optimal Sensor Placement}
\label{sec:param}

\color{black} In this section, we present the joint online parameter estimation and optimal sensor placement algorithm for a generic \color{black} partially observed infinite dimensional linear diffusion process \color{black} of the form \color{black}
\begin{alignat}{2}
\mathrm{d}u(t) &= \mathcal{A}(\theta)u(t)\mathrm{d}t + \mathcal{B}\mathrm{d}v_{\theta}(t)~,~~~&&u(0)=u_0 \label{signal_finite_dim_inf} \\
\mathrm{d}y(t) &= \mathcal{C}(\theta,\boldsymbol{o})u(t)\mathrm{d}t + \mathrm{d}w_{\boldsymbol{o}}(t)~,~~~&&y(0)=0,\label{obs_finite_dim_inf}
\end{alignat}
\color{black} where $\mathcal{A}(\theta):\mathcal{H}\rightarrow\mathcal{H}$ is the infinitesimal generator of a $C_0$-semigroup over a separable Hilbert space $\mathcal{H}$, $\mathcal{B}:\mathcal{H}\rightarrow\mathcal{H}$ and $\mathcal{C}(\theta,\boldsymbol{o}):\mathcal{H}\rightarrow\mathbb{R}^{n_y}$ are bounded linear operators, $v_{\theta}(t)$ and $w_{\boldsymbol{o}}(t)$ are independent $\mathcal{H}$ and $\mathbb{R}^{n_y}$-valued Wiener processes with incremental covariances $\mathcal{Q}(\theta)$ and $\mathcal{R}(\boldsymbol{o})$, and $u_0$ is a $\mathcal{H}$-valued Gaussian random variable, independent of $v_{\theta}$ and $w_{\boldsymbol{o}}$. Clearly, this abstract framework includes the partially observed stochastic advection diffusion equation defined in Section \ref{sec:equation}. \color{black}

\subsection{The Infinite-Dimensional Kalman-Bucy Filter}
\label{sec:filtering}
We begin by briefly review the infinite dimensional filtering problem. This refers to the problem of determining the conditional law of the latent signal process, given the history of observations $\smash{\mathcal{F}_t^Y\hspace{-1mm} =\hspace{-.2mm} \sigma\{y(s)\hspace{-.2mm}:\hspace{-.2mm}s\in[0,t]\}}\hspace{-.2mm}$. In the linear Gaussian case, it is well known that the conditional distribution of the latent signal is Gaussian, and thus determined uniquely by its mean and covariance. These quantities \color{black} are given in closed form by \color{black} the infinite-dimensional Kalman-Bucy filter (e.g., \cite{Bensoussan1971,Curtain1975,Curtain1978a}). 

In particular, suppose that we write $\hat{{u}}(\theta,\boldsymbol{o},t)$ for the conditional mean of the latent signal, and $\Sigma(\theta,\boldsymbol{o},t)$ for its conditional covariance. Then
$\hat{u}(\theta,\boldsymbol{o},t)$ is a mild solution of the stochastic evolution equation \cite{Curtain1978a} %\cite[Theorem 6.21]{Curtain1978a} 
\begin{align}
\mathrm{d}\hat{u}(\theta,\boldsymbol{o},t) &= \mathcal{A}(\theta)\hat{{u}}(\theta,\boldsymbol{o},t)\mathrm{d}t \label{inf_dim_mean}  \\
&\hspace{5mm}+\Sigma(\theta,\boldsymbol{o},t) \mathcal{C}^{\dagger}(\theta,\boldsymbol{o})\mathcal{R}^{-1}(\boldsymbol{o})\big(\mathrm{d}y(t) - \mathcal{C}(\theta,\boldsymbol{o})\hat{{u}}(\theta,\boldsymbol{o},t)\mathrm{d}t\big), \nonumber %\\[1mm]
%\hat{u}(\theta,\boldsymbol{o},0) &= \hat{u}_0(\theta,\boldsymbol{o}). \label{inf_dim_mean_init}
\end{align}
and $\Sigma(\theta,\boldsymbol{o},t)$ is a weak solution of the operator Ricatti equation \cite{Bensoussan2007,Curtain1978a} % \cite[Theorem 6.10]{Curtain1978a},
%\footnote{For a definition of `weak solution', see, for example, \cite{Bensoussan2007}.}
 \iffalse
 \footnote{In particular, for all $\varphi_1,\varphi_2\in\mathcal{D}(\mathcal{A})$, the function $\langle\Sigma(\theta,\boldsymbol{o},\cdot)\varphi_1,\varphi_2)\rangle$ is differentiable, and verifies the equation (e.g., \cite{Bensoussan2007})
 \begin{align}
\big\langle\dot{\Sigma}(\theta,\boldsymbol{o},t)\varphi_1,\varphi_2\big\rangle&= \big\langle\big(\mathcal{A}(\theta) \Sigma(\theta,\boldsymbol{o},t) + \Sigma(\theta,\boldsymbol{o},t)\mathcal{A}^{\dagger}(\theta)+\mathcal{B}\mathcal{Q}(\theta)\mathcal{B}^{\dagger} \label{inf_dim_ricatti_weak} \\
&\hspace{5mm}-\Sigma(\theta,\boldsymbol{o},t) \mathcal{C}^{\dagger}(\theta,\boldsymbol{o})\mathcal{R}^{-1}(\boldsymbol{o})\mathcal{C}(\theta,\boldsymbol{o})\Sigma(\theta,\boldsymbol{o},t)\big)\varphi_1,\varphi_2\big\rangle,  \nonumber \\[2mm]
\Sigma(\theta,\boldsymbol{o},0) &= \Sigma_0(\theta,\boldsymbol{o})
\end{align}}
\fi
 \begin{align}
\dot{\Sigma}(\theta,\boldsymbol{o},t)&=\mathcal{A}(\theta) \Sigma(\theta,\boldsymbol{o},t) + \Sigma(\theta,\boldsymbol{o},t)\mathcal{A}^{\dagger}(\theta) \label{inf_dim_ricatti} \\
&\hspace{5mm}+\mathcal{B}\mathcal{Q}(\theta)\mathcal{B}^{\dagger}-\Sigma(\theta,\boldsymbol{o},t) \mathcal{C}^{\dagger}(\theta,\boldsymbol{o})\mathcal{R}^{-1}(\boldsymbol{o})\mathcal{C}(\theta,\boldsymbol{o})\Sigma(\theta,\boldsymbol{o},t). \nonumber  %\\[2mm]
%\Sigma(\theta,\boldsymbol{o},0) &= \Sigma_0(\theta,\boldsymbol{o}),
 \end{align}

\subsection{Online Parameter Estimation}
We can now turn our attention to the problem of online parameter estimation. We will suppose that the model generates observations according to a true, but unknown, static parameter \color{black} $\smash{\theta^{*}}$. \color{black} Our objective is to obtain an estimator of the true parameter which is both $\mathcal{F}_t^Y$-measurable and recursively computable. That is, an estimator which can be computed online using the continuous stream of observations, without revisiting the past. %In this subsection, we will assume that the sensor locations $\boldsymbol{o}\in\Omega^{n_y}$ are fixed. 

For this task, \color{black} we will follow an approach based on the maximum likelihood principle. Naturally, \color{black}  we thus require an expression for the log-likelihood of the observations, or incomplete data log-likelihood, for a partially observed infinite-dimensional linear diffusion process. Under appropriate conditions, this can be obtained as (e.g., \cite{Aihara1988,Bagchi1984,Balakrishnan1973,Balakrishnan1975a,Liptser2001})
\begin{align}
\mathcal{L}_t(\theta,\boldsymbol{o}) %= \log\frac{\mathrm{d}\mathbb{P}_{y}(\theta)}{\mathrm{d}\mathbb{P}_{w}}
= \int_0^t \left\langle \mathcal{R}^{-1}(\boldsymbol{o}) \mathcal{C}(\theta,\boldsymbol{o})\hat{u}(\theta,\boldsymbol{o},s),\mathrm{d}y(s)\right\rangle -\frac{1}{2}\int_0^t ||\mathcal{R}^{-\frac{1}{2}}(\boldsymbol{o})\mathcal{C}(\theta,\boldsymbol{o})\hat{u}(\theta,\boldsymbol{o},s)||^2\mathrm{d}s, \label{log_lik} %\hspace{-4mm}
\end{align}
where \color{black} $\smash{\langle\cdot,\cdot\rangle}$ denotes the standard inner product on $\mathbb{R}^{n_y}$, and $\smash{||\varphi||^2 = \langle \varphi,\varphi\rangle}$. \color{black} %$\hat{u}(\theta,\boldsymbol{o})$ is the mild solution of the stochastic evolution equation \eqref{inf_dim_mean}. 
In the online setting, a standard approach is to recursively seek the value of $\theta$ which maximises the asymptotic log-likelihood, \color{black} namely, 
$\smash{\tilde{\mathcal{L}}(\theta,\boldsymbol{o}) = \lim_{t\rightarrow\infty}\frac{1}{t}\mathcal{L}_t(\theta,\boldsymbol{o})}$. \color{black}
\iffalse
\begin{equation}
\tilde{\mathcal{L}}(\theta,\boldsymbol{o}) = \lim_{t\rightarrow\infty}\frac{1}{t}\mathcal{L}_t(\theta,\boldsymbol{o}).%=\lim_{t\rightarrow\infty}\left[\frac{1}{t}\int_0^t R^{-1}(\boldsymbol{o})\hat{C}(\theta,\boldsymbol{o},s)\cdot\mathrm{d}y(s)-\frac{1}{2t}\int_0^tR^{-1}(\boldsymbol{o})||\hat{C}(\theta,\boldsymbol{o},s)||^2\mathrm{d}s\right]. \label{asymptotic_ll} \hspace{-27mm}
\end{equation}
\fi
\iffalse
In the finite-dimensional setting, we recall the following results:
\begin{itemize}
\item Existence of asymptotic log-likelihood. \cite[Proposition 1(i)]{Surace2019}.
\item Existence of maximum likelihood estimate. [Any references?]
\item Uniform continuity and differentiability of asymptotic log-likelihood. \cite[Proposition 1(ii)]{Surace2019}.
\item Boundedness of the asymptotic log-likelihood and its derivatives. \cite[Proposition 1(iv)]{Surace2019}.
\item Convergence of the approximate asymptotic log-likelihood. [Any references?]
\item Consistency of the finite-dimensional (projection) of the RMLE. [Any references?]
\item Very few results here. This remains an open problem (i.e., sufficient conditions for existence, regularity, differentiability of asymptotic log-likelihood for partially observed stochastic partial differential equations, and convergence of maxima of the approximate asymptotic log-likelihood to the maxima of the true asymptotic log-likelihood).
\end{itemize}
\fi
This optimisation problem can be tackled using a continuous-time stochastic gradient ascent algorithm, whereby the parameters follow a noisy ascent direction given by the integrand of the gradient of the log-likelihood, evaluated with the current parameter estimate (see \cite{Gerencser2009,Surace2019} in the finite-dimensional case). %We extend this approach to the infinite-dimensional setting. 
In particular, initialised at ${\theta}_0\in\Theta$, the parameter estimates can be generated according to \\[-4mm]
\begin{equation}
\mathrm{d}{\theta}(t) = \left\{ \begin{array}{lll} \gamma(t) \big[\mathcal{C}(\theta,\boldsymbol{o})\hat{u}^{\theta}(\theta,\boldsymbol{o},t)\big]^T\mathcal{R}^{-1}(\boldsymbol{o})\big[\mathrm{d}y(t)-\mathcal{C}(\theta,\boldsymbol{o})\hat{u}(\theta,\boldsymbol{o},t)\mathrm{d}t\big] \big|_{\theta=\theta(t)} & \hspace{-3mm}, & \hspace{-2.5mm} {\theta}(t)\in\Theta \\ 0 & \hspace{-3mm}, & \hspace{-2.5mm} {\theta}(t)\not\in\Theta  \end{array}\right. \hspace{-3mm} \label{eq_sde}
\end{equation}
\color{black} where $\gamma:\mathbb{R}_{+}\rightarrow\mathbb{R}_{+}$ is a non-negative, non-increasing, real function, known as the learning rate, and we have written $\hat{u}^{\theta}(\theta,\boldsymbol{o},t)$ for the `filter derivative' of the conditional mean. By this, we mean that this process is the solution, interpreted in the appropriate sense, of the equation obtained upon formal differentiation of \eqref{inf_dim_mean} with respect to $\theta$. \color{black} This algorithm is commonly referred to as recursive maximum likelihood (RML). %We remark that the inclusion of a projection device, which ensures that the parameter estimates remain in $\Theta$ with probability one, is standard in algorithms of this type.

\subsection{Optimal Sensor Placement} \label{sec:opt_sens}
We now review the optimal sensor placement problem. Here, the objective is to obtain an estimator of the set of $n_y$ sensor locations \color{black} ${\boldsymbol{o}}^{*} = \{{\boldsymbol{o}}^{*}_i\}_{i=1}^{n_{y}}$ \color{black} which are optimal with respect to some pre-determined criteria, possibly subject to constraints. Once more, we require our estimator to be $\mathcal{F}_t^Y$-measurable and recursively computable. %In this subsection, we will assume that the model parameters $\theta\in\Theta$ are fixed.
A standard approach to this problem is to first define a suitable objective function, say $\mathcal{J}(\theta,\cdot):\Omega^{n_{y}}\rightarrow\mathbb{R}$, and then to define the optimal sensor placement as ${ \color{black}  \boldsymbol{o}}^{*} \color{black} = \argmin_{\boldsymbol{o}\in\Omega^{n_{y}}} \mathcal{J}(\theta,\boldsymbol{o}).$
%\begin{equation}
%\hat{\boldsymbol{o}} = \argmin_{\boldsymbol{o}\in\Omega^{n_{y}}} \mathcal{J}(\theta,\boldsymbol{o}).
%\end{equation}

In this paper, we are interested in determining the sensor placement which minimises the uncertainty in our optimal state estimate. In this case, an appropriate objective function is (e.g., \cite{Burns2015,Chen1975,Herring1974,Korbicz1994}),
\begin{equation}
\mathcal{J}_t(\theta,\boldsymbol{o}) = \int_0^{t} \mathrm{Tr}\left[\mathcal{M}(s){\Sigma}(\theta,\boldsymbol{o},s)\right]\mathrm{d}s \label{obj_func}
\end{equation}
where ${\Sigma}(\theta,\boldsymbol{o},s)$ is the weak solution of the operator Ricatti equation (\ref{inf_dim_ricatti}), and $\mathcal{M}(s):\mathcal{H}\rightarrow\mathcal{H}$ is a user-chosen, possibly time-dependent, bounded linear operator, which allows one to weight significant parts of the state estimate.
\iffalse
Regarding this objective function, we recall the following results:
\begin{itemize}
\item Existence of optimal sensor placement \cite[Theorem 5.3]{Burns2015}. 
\item Differentiability of objective function \cite[Theorem 5.5]{Burns2015}. 
\item Convergence of finite-dimensional approximations of the objective function \cite[Corollary 6.3]{Burns2015}.
\item Convergence of approximate optimal sensor placements (does this follow immediately from the previous result?).
\end{itemize}
\fi
In the spirit of the previous section, in the online setting, to obtain the optimal sensor placement we can recursively seek the value of $\boldsymbol{o}$ which minimises the asymptotic objective function, namely \color{black} $\tilde{\mathcal{J}}(\theta,\boldsymbol{o})= \lim_{t\rightarrow\infty} \frac{1}{t}{\mathcal{J}}_t(\theta,\boldsymbol{o})$. 
\color{black}

Similarly to online parameter estimation, this optimisation problem can be tackled recursively using continuous-time stochastic gradient descent, whereby the sensor locations follow a noisy descent direction given by the integrand of the gradient of the objective function, evaluated with the current estimates of the sensor placements. In particular, initialised at ${\boldsymbol{o}}_0\in\Omega^{n_{y}}$,  estimates of the optimal sensor locations can be recursively generated according to
\begin{equation}
\mathrm{d}{\boldsymbol{o}}(t)= \left\{ \begin{array}{lll}  -\gamma(t)\mathrm{Tr}^{\boldsymbol{o}}\big[\mathcal{M}(t)\Sigma(\theta,\boldsymbol{o},t)\big]^T\mathrm{d}t \big|_{\boldsymbol{o}=\boldsymbol{o}(t)} & , & {\boldsymbol{o}}(t)\in\Omega^{n_{y}} \\ 0 & , & {\boldsymbol{o}}(t)\not\in\Omega^{n_{y}}  \end{array}\right.  \label{sensor_alg}
\end{equation}
where $\gamma:\mathbb{R}_{+}\rightarrow\mathbb{R}_{+}$ is the learning rate, and $\smash{\mathrm{Tr}^{\boldsymbol{o}}\left[\mathcal{M}(t)\Sigma(\theta,\boldsymbol{o},t)\right]=\nabla_{\boldsymbol{o}}\mathrm{Tr}\left[\mathcal{M}(t)\Sigma(\theta,\boldsymbol{o},t)\right]}$
denotes the `filter derivative' of the weighted trace of the conditional covariance with respect to the sensor placements.
%\footnote{This quantity can be computed directly from the filter derivative of the conditional covariance operator, namely, ${\Sigma}^{\boldsymbol{o}}(\theta,\boldsymbol{o},t)=\nabla_{\boldsymbol{o}}{\Sigma}(\theta,\boldsymbol{o},t)$. To see this, observe that the components of $\smash{\mathrm{Tr}^{\boldsymbol{o}}\left[\mathcal{M}(t)\Sigma(\theta,\boldsymbol{o},t)\right]=\nabla_{\boldsymbol{o}}\mathrm{Tr}\left[\mathcal{M}(t)\Sigma(\theta,\boldsymbol{o},t)\right]}$ can be computed as
%\begin{align}
%\left[\mathrm{Tr}^{\boldsymbol{o}}\left[\mathcal{M}(t)\Sigma(\theta,\boldsymbol{o},t)\right]\right]_i%:&=\bigg[\nabla_{\boldsymbol{o}}\mathrm{Tr}\left[\mathcal{M}(t)\Sigma(\theta,\boldsymbol{o},t)\right]\bigg]_{i} 
%= \partial_{\boldsymbol{o}_i}\mathrm{Tr}\left[\mathcal{M}(t)\Sigma(\theta,\boldsymbol{o},t)\right] = \mathrm{Tr}\left[\mathcal{M}(t)\partial_{\boldsymbol{o}_i}\Sigma(\theta,\boldsymbol{o},t)\right]
%%%= \mathrm{Tr}\left[\mathcal{M}(t)\left[\nabla_{\boldsymbol{o}}\Sigma(\theta,\boldsymbol{o},t)\right]_i\right] 
%=\mathrm{Tr}\left[\mathcal{M}(t)\left[\Sigma^{\boldsymbol{o}}(\theta,\boldsymbol{o},t)\right]_i\right]. \nonumber 
%\end{align}} 

%We will refer to this algorithm as recursive optimal sensor placement.

%Similar to the online parameter estimation algorithm, this recursion includes a projection device to ensure that the sensor placements $\{\boldsymbol{o}(t)\}_{t\geq 0}$ remain in the domain $\Omega^{n_{y}}$ with probability one.
We should remark that, if the true model parameters are known,  \color{black} then one can compute the asymptotic sensor placement objective function (and its gradient) prior to receiving any observations by solving the so-called {algebraic Ricatti equation} (see Appendix \ref{app:ARE}).
%It follows that, 
%\begin{equation}
%\tilde{\mathcal{J}}(\theta,\boldsymbol{o}) = \mathrm{Tr}\left[\mathcal{M}_{\infty} \Sigma_{\infty}(\theta,\boldsymbol{o})\right]
%\end{equation}
 %if the true model parameters %$\theta^{*}$ 
 %are {known}, then one can
 %then the asymptotic covariance $\smash{\Sigma_{\infty}(\theta^{*},\boldsymbol{o})}$ can be computed prior to receiving any observations by solving the algebraic Ricatti equation (see Appendix []). It is thus possible to 
%compute the asymptotic objective function 
%$\smash{\tilde{\mathcal{J}}(\theta^{*},\boldsymbol{o})}$ 
%and its gradient 
%$\smash{\nabla_{\boldsymbol{o}}\tilde{\mathcal{J}}(\theta^{*},\boldsymbol{o})}$ 
%directly, 
%=\mathrm{Tr}[\mathcal{M}_{\infty}\Sigma_{\infty}(\theta,\boldsymbol{o})]}$,
 %, and to obtain the optimal sensor placement ${{\boldsymbol{o}}^{*} =\argmin_{\boldsymbol{o}\in\Omega^{n_y}}\tilde{\mathcal{J}}(\theta^{*},\boldsymbol{o})}$, 
 %prior to receiving any observations.
In this case, it would thus arguably be preferable to use a (non-stochastic) gradient descent algorithm on the asymptotic objective function directly in order to obtain the optimal sensor placement (e.g., \cite{Aidarous1978,Amouroux1978}). If, however, the true model parameters are {unknown}, then one can no longer compute the true asymptotic sensor placement objective function (or its gradient) prior to receiving any observations, since the true solution of the algebraic Ricatti equation (and thus the optimal sensor placement) depends on knowing the value of these parameters (see Figure \ref{fig_section2}).
%This is particularly significant in cases where the solution of the algebraic Ricatti equation, and thus the optimal sensor placement, is highly dependent on the model parameters (see Figure \ref{fig_section2}). 
%In such cases, if one tries to compute the optimal sensor placement prior to receiving any observations, but using an incorrect specification of the true model parameters, then there is no guarantee of obtaining the true optimal sensor placement, even if it is possible to compute the global optimum of the objective function for those parameter values (see Figure \ref{fig_section2}). 
This observation highlights the need to tackle the parameter estimation and optimal sensor placement problems together.
\color{black}

\begin{figure}[!h]
\captionsetup[subfloat]{captionskip=-.7pt}
\centering
  \subfloat[]{\label{fig_section2_a}\includegraphics[width=0.3\textwidth]{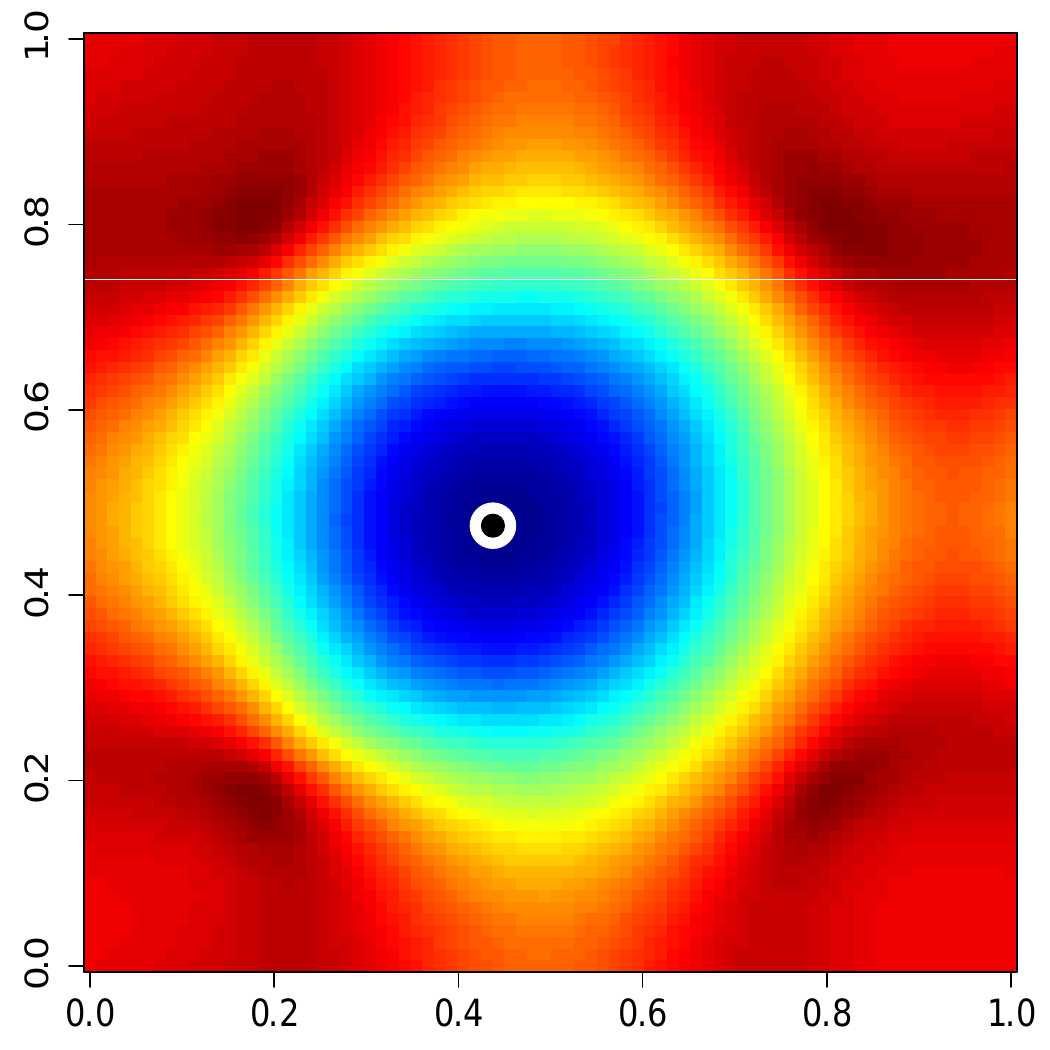}}
  \hspace{10mm}
  \subfloat[]{\label{fig_section2_b}\includegraphics[width=0.3\textwidth]{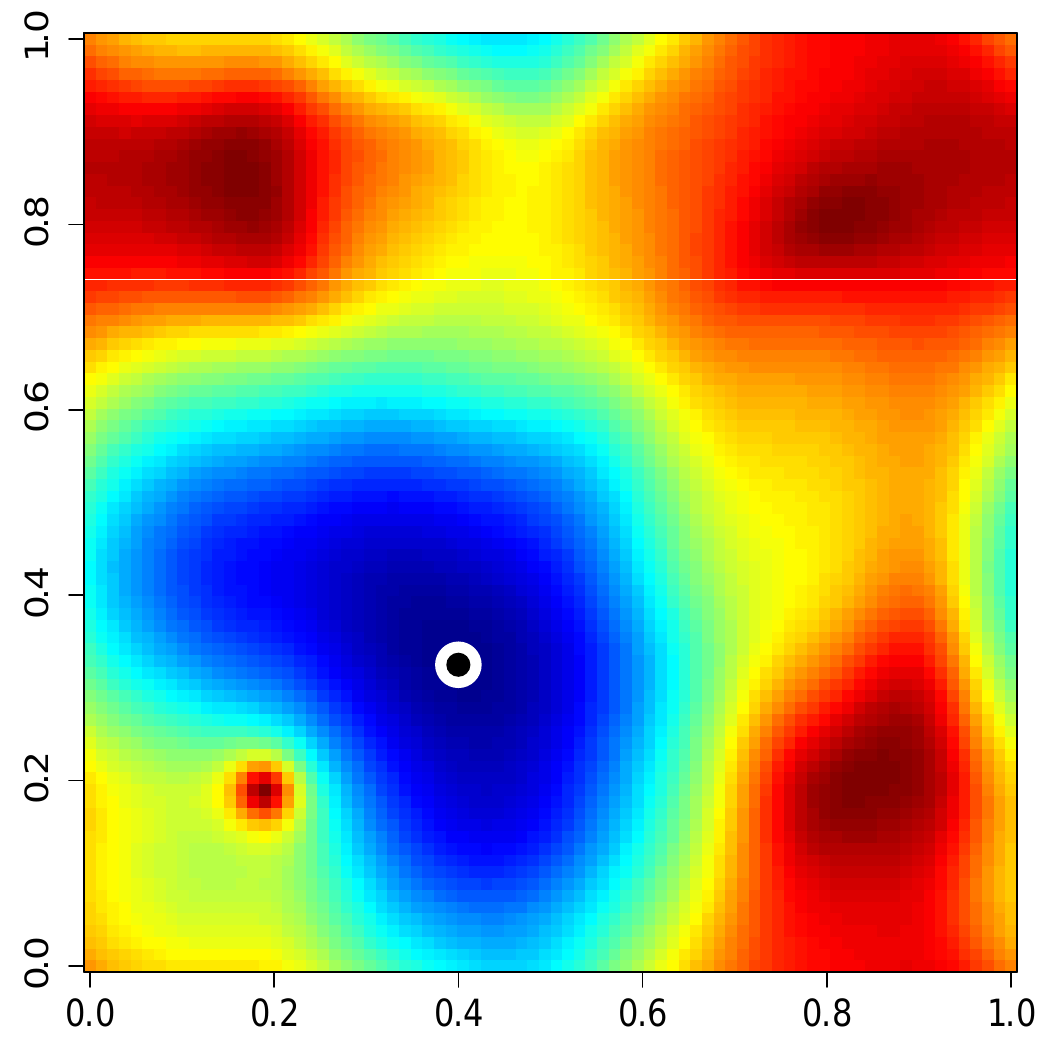}}
  \vspace{-1mm}
\caption{The `optimisation landscape'. Plots of the asymptotic sensor placement objective function, $\smash{\tilde{\mathcal{J}}}(\theta,\boldsymbol{o})$%= \mathrm{Tr}[\mathcal{M}_{\infty}\Sigma_{\infty}(\theta,\boldsymbol{o})]$
, and the corresponding optimal sensor placement, ${\boldsymbol{o}}^{*}=\argmin_{\boldsymbol{o}\in\Omega}\tilde{\mathcal{J}}(\theta,\boldsymbol{o})$, for two possible specifications of the model parameters $\theta$.}
\label{fig_section2}
\end{figure}

\subsection{Joint Online Parameter Estimation and Optimal Sensor Placement}
\label{sec_joint_RML_OSP} 

We finally now turn our attention to the problem of joint online parameter estimation and optimal sensor placement. In the spirit of the previous two sections, we will formulate this as a bilevel optimisation problem, in which the objective is to obtain ${\theta}\in\Theta$, ${\boldsymbol{o}}\in\Omega^{n_{y}}$ which simultaneously maximise the asymptotic log-likelihood $\tilde{\mathcal{L}}(\theta,\boldsymbol{o})$, and minimise the asymptotic objective function $\tilde{\mathcal{J}}(\theta,\boldsymbol{o})$. 

There are two possible approaches to this task. The first is to alternate between online parameter estimation and optimal sensor placement, periodically updating the locations of the measurement sensors on the basis of the current parameter estimates. The second is to jointly perform online parameter estimation and optimal sensor placement, simultaneously and recursively updating the parameter estimates and the locations of the measurement sensors. We strongly advocate the second approach, which is not only more numerically convenient, but can be implemented in a truly online fashion. Moreover, in the case of mobile sensors, this approach can provide real-time motion guidance.

On this basis, we propose the use of a continuous-time, two-timescale stochastic gradient descent algorithm (e.g., \cite{Borkar2008,Borkar1997,Sharrock2020a}), which combines the schemes in \eqref{eq_sde} and \eqref{sensor_alg}. %This algorithm can be viewed as an extension of the algorithm introduced in \cite{Sharrock2020a} to the infinite-dimensional setting. 
In particular, suppose some initialisation at ${\theta}_0\in\Theta$, ${\boldsymbol{o}}_0\in\Omega^{n_{y}}$. Then, simultaneously, the parameter estimates and the sensor placements are generated according to
\begin{subequations}
\begin{alignat}{2}
\hspace{-1mm} \mathrm{d}{\theta}(t) \hspace{-.5mm}&=\hspace{-.5mm} \left\{ \hspace{-.5mm}\begin{array}{l} \hspace{-1mm}-\gamma_{\theta}(t) \big[\mathcal{C}(\theta,\boldsymbol{o})\hat{u}^{\theta}(\theta,\boldsymbol{o},t)\big]^T\mathcal{R}^{-1}(\boldsymbol{o})\big[\mathcal{C}(\theta,\boldsymbol{o})\hat{u}(\theta,\boldsymbol{o},t)\mathrm{d}t-\mathrm{d}y(t)\big] \big|_{\substack{\theta=\theta(t) \\ \boldsymbol{o}=\boldsymbol{o}(t)}} \\ \hspace{-1mm} 0 \end{array}\right. &&\begin{array}{ll}  \hspace{-4mm}, & \hspace{-2mm} {\theta}(t)\in\Theta, \\[3mm] \hspace{-4mm}, & \hspace{-2mm} {\theta}(t)\not\in\Theta,  \end{array} \hspace{-12mm} \label{RML_ROSP_1} \\
\hspace{-1mm} \mathrm{d}{\boldsymbol{o}}(t)\hspace{-.5mm}&=\hspace{-.5mm} \left\{\hspace{-.5mm} \begin{array}{l}  \hspace{-1mm}-\gamma_{\boldsymbol{o}}(t)\mathrm{Tr}^{\boldsymbol{o}}\big[\mathcal{M}(t)\Sigma(\theta,\boldsymbol{o},t)\big]^T\mathrm{d}t \big|_{\substack{\theta=\theta(t) \\ \boldsymbol{o}=\boldsymbol{o}(t)}}  \\ \hspace{-1mm} 0 \end{array}\right. &&\begin{array}{ll} \hspace{-4mm}, & \hspace{-3mm}  {\boldsymbol{o}}(t)\in\Omega^{n_{y}}, \\[3mm]  \hspace{-4mm}, & \hspace{-3mm}  {\boldsymbol{o}}(t)\not\in\Omega^{n_{y}}.  \end{array} \hspace{-12mm} \label{RML_ROSP_2}
 % \customlabel{algorithm3.1}{3.1}
 \vspace{-8mm}
\end{alignat}
\end{subequations}
\iffalse
\begin{subequations}
\begin{alignat}{2}
\mathrm{d}{\theta}(t) \hspace{-.5mm} &=
\hspace{-1mm} \left\{ \begin{array}{l} -\gamma_{1}(t) \big[\mathcal{C}(\boldsymbol{o}(t))\hat{u}^{\theta}({\theta}(t),\boldsymbol{o}(t),t)\big]^T\mathcal{R}^{-1}(\boldsymbol{o}(t)) \big[C(\boldsymbol{o}(t))\hat{u}(\theta(t),\boldsymbol{o}(t),t)\mathrm{d}t-\mathrm{d}y(t)\big]^T \hspace{-3mm} \\ 0 \end{array} \right.
 &&\begin{array}{ll} \hspace{-1mm} , \hspace{-1mm} &  \hspace{-.5mm}  {\theta}(t)\in\Theta, \\ \hspace{-1mm} , \hspace{-1mm} & \hspace{-.5mm}{\theta}(t)\not\in\Theta,  \end{array} \label{RML_ROSP_1} \hspace{-5mm} \\
\mathrm{d}{\boldsymbol{o}}(t)  \hspace{-.5mm} &= \hspace{-1mm} \left\{ \begin{array}{l}  -\gamma_{2}(t)\big[\hat{j}^{\boldsymbol{o}}(\theta(t),\boldsymbol{o}(t),t)\big]^T\mathrm{d}t \hspace{-3mm} \\ 0 \end{array} \right.
 &&\begin{array}{ll} \hspace{-1mm} , \hspace{-1mm} &  \hspace{-.5mm} {\boldsymbol{o}}(t)\in\Omega^{n_y}, \\ \hspace{-1mm} , \hspace{-1mm} & \hspace{-.5mm} {\boldsymbol{o}}(t)\not\in\Omega^{n_y},  \end{array} \hspace{-5mm} \label{RML_ROSP_2}
 \customlabel{algorithm3.1}{3.1}
\end{alignat}
\end{subequations}
\fi
where \color{black} $\gamma_{\theta},\gamma_{\boldsymbol{o}}:\mathbb{R}_{+}\rightarrow\mathbb{R}_{+}$ are learning rates \color{black} which satisfy either $\smash{\lim_{t\rightarrow\infty}{\gamma_{\theta}(t)}/{\gamma_{\boldsymbol{o}}(t)}=0}$ or $\smash{\lim_{t\rightarrow\infty}{\gamma_{\boldsymbol{o}}(t)}/{\gamma_{\theta}(t)}=0}$. 
The choice between these two conditions determines which of the algorithm iterates moves on a slower timescale, \color{black} which in turn determines whether parameter estimation or optimal sensor placement is the primary objective.  
For example, \color{black} the first implies that the parameter estimates move on a slower timescale than the sensor placements, and is generally preferred if parameter estimation is the primary objective. 
%The second implies that the sensor placements move on a slower timescale than the parameter estimates, and is generally preferred if optimal sensor placement is the primary objective. 

%This condition on the learning rates implies that the parameter estimates $\{\theta(t)\}_{t\geq 0}$ move on a slower timescale than the optimal sensor placements $\{\boldsymbol{o}(t)\}_{t\geq 0}$. %As previously, these recursions include projections to ensure that the parameter estimates $\{\theta(t)\}_{t\geq 0}$ and the sensor placement $\{\boldsymbol{o}(t)\}_{t\geq 0}$ remain in the open sets $\Theta$ and $\Omega^m$, respectively, with probability one.

\iffalse
\begin{theorem_no_number} \label{theorem1}
Assume that Conditions \ref{assumption1} - \ref{assumption11} hold (see Appendix \ref{appendixA}). Then, with probability one, 
\begin{align}
\lim_{t\rightarrow\infty}\nabla_{\theta}\tilde{\mathcal{L}}(\theta(t),\boldsymbol{o}(t)) = \lim_{t\rightarrow\infty}\nabla_{\boldsymbol{o}}\tilde{\mathcal{J}}(\theta(t),\boldsymbol{o}(t)) = 0,
\end{align}
or
\begin{align}
\lim_{t\rightarrow\infty}(\theta(t),\boldsymbol{o}(t)) \in \{(\theta,\boldsymbol{o}):\theta\in \partial\Theta\cup \boldsymbol{o}\in \partial\Omega^{n_y}\}.
\end{align}
\end{theorem_no_number}
\fi

In practice, we cannot implement this algorithm directly, as it depends on the infinite-dimensional solutions of the Kalman-Bucy filtering equations. We are thus required to use a Galerkin discretisation, and project onto a finite-dimensional Hilbert space (e.g., \cite{Germani1988,Rosen1991}). 
Under certain, verifiable conditions, the approximate, finite-dimensional  solutions of the Kalman-Bucy filtering equations, and the algebraic Ricatti equation, converge to the true, infinite-dimensional solutions as the order of the projection %$\Pi_n$ 
is increased (e.g., \cite{Banks1984,Burns2015,DeSantis1993,Germani1988,Wu2016,Zhang2018}). 
It is thus reasonable to expect that, under similar conditions, the finite-dimensional approximations of the asymptotic log-likelihood and the asymptotic objective function, $\smash{\tilde{\mathcal{L}}_n(\theta,\boldsymbol{o})}$ and $\smash{\tilde{\mathcal{J}}_n(\theta,\boldsymbol{o})}$, will converge to their true, infinite-dimensional counterparts, as will the corresponding approximations of the MLE and the optimal sensor placement.
%, $\smash{\hat{\theta}_n := \argmax_{\theta\in\Theta}\tilde{\mathcal{L}}_n(\theta,\boldsymbol{o})}$ and $\smash{\hat{\boldsymbol{o}}_n := \argmin_{\boldsymbol{o}\in\Omega^{n_y}}\tilde{\mathcal{J}}_{n}(\theta,\boldsymbol{o})}.$
\iffalse
\begin{equation}
\smash{\hat{\theta}_n := \argmax_{\theta\in\Theta}\tilde{\mathcal{L}}_n(\theta,\boldsymbol{o})}~~\text{ and } ~~\smash{\hat{\boldsymbol{o}}_n := \argmin_{\boldsymbol{o}\in\Omega^{n_y}}\tilde{\mathcal{J}}_{n}(\theta,\boldsymbol{o})}.
\end{equation}
\fi

In fact, rigorous convergence results of this type have already been obtained for the sensor placement objective function and the optimal sensor placement. In particular, under precisely the conditions required for convergence of the finite-dimensional filter, the finite-dimensional approximation of the sensor placement objective function and the optimal sensor placement do indeed converge to their true values  (e.g., \cite{Burns2015,Wu2016,Zhang2018}). 
\color{black} While similar results do not currently exist for the log-likelihood and the MLE (see \cite{Cialenco2018} for some relevant results in the fully observed case), it is expected that similar arguments could be applied in this setting.\color{black}

In this context, we have strong justification for implementing a finite-dimensional version of Algorithm \eqref{RML_ROSP_1} - \eqref{RML_ROSP_2}, in which the filter and filter derivatives are replaced by their finite-dimensional approximations. The resulting algorithm is a particular case of the joint online parameter estimation and optimal sensor placement algorithm analysed in \cite[Proposition 3.1]{Sharrock2020a}. Thus, under suitable conditions on the latent state, the optimal filter, and the filter derivatives, the parameter estimates and the optimal sensor placements generated by this algorithm are guaranteed to converge to the stationary points of the (finite-dimensional approximations of the) asymptotic log-likelihood and the asymptotic sensor placement objective function, respectively. That is, 
\begin{equation}
\lim_{t\rightarrow\infty}\nabla_{\theta}\tilde{\mathcal{L}}_n(\theta(t),\boldsymbol{o}(t)) = \lim_{t\rightarrow\infty}\nabla_{\boldsymbol{o}}\tilde{\mathcal{J}}_n(\theta(t),\boldsymbol{o}(t)) = 0.
\end{equation}
We state these conditions in full in Appendix \ref{appendixA}, and provide sufficient conditions in the linear Gaussian case. 

\newpage
\section{Numerical Results}
\label{sec:numerics}

In this section, we provide numerical examples illustrating the performance of the joint online parameter estimation and optimal sensor placement algorithm. \color{black} The \texttt{R} code is available at https://github.com/louissharrock/RML-ROSP. All simulations are performed on a 2015 MacBook Pro with 2.7 GHz Intel Core i5 processor and 8GB RAM. \color{black}

\subsection{Numerical Considerations}

%\subsubsection{Finite Dimensional Approximation}
%[\textbf{To do}]%: verify that conditions for finite-dimensional approximations are satisfied (e.g., can just refer to \cite{Burns2015}).]
%\subsubsection{Finite Dimensional Approximation}

For numerical purposes, \color{black} we will \color{black} project the infinite-dimensional solution of the signal equation onto a finite dimensional Hilbert space. \color{black} In particular, we will consider the finite dimensional subspace $\mathcal{H}_n\subset\mathcal{H}$ spanned by the truncated set of Fourier basis functions $\{\phi_{\boldsymbol{k}}\}_{\boldsymbol{k}\in\Lambda_{n}}$, where $\Lambda_{n}\subset \mathbb{Z}^2$ is the set of wave-numbers
\begin{equation}
\Lambda_{n}=\left\{\boldsymbol{k}\in\mathbb{Z}^2: -\left(\frac{n}{2}-1\right)\leq k_1,k_2\leq \frac{n}{2}\right\}~,~~~n\in2\mathbb{N}~,~~~|\Lambda_n| = n^2. \label{index_set}
\end{equation}
\color{black} We should emphasise that this choice of basis is not unique. Indeed, in principle, one could consider any finite dimensional basis (e.g., Chebyshev polynomials, finite-elements, etc.), provided that the resulting projection converged in an appropriate sense as its dimension increased (see, e.g., Theorem 4.2 in \cite{Zhang2018}). Indeed, other choices of the finite dimensional basis, such as the piecewise linear basis functions in \cite{Lindgren2011}, may be more appropriate in the case of non-periodic boundary conditions or other more complex geometries.  
%We highlight, in particular, the Gaussian Markov Random Field approach introduced in \cite{Lindgren2011}, which makes use of piecewise linear basis functions on a triangulation of the domain. 

For $n\geq 1$, let $\Pi_n:\mathcal{H}\rightarrow \mathcal{H}_n$ denote the orthogonal projection onto this space, defined in the usual fashion. The Galerkin projection of $u(t)$ is then given by $u_n(t) \color{black} = u_n(\cdot,t)=\{u_n(\boldsymbol{x},t):\boldsymbol{x}\in\Pi\}\in\mathcal{H}_n$, where
\begin{equation}
\color{black} u_n(\boldsymbol{x},t) \color{black} = \Pi_n \color{black} u(\boldsymbol{x},t) \color{black} = \sum_{\textbf{k}\in \Lambda_n} \alpha_{\boldsymbol{k}}(t) \color{black} \phi_{\boldsymbol{k}}(\boldsymbol{x}) \color{black} ~,~~~\alpha_{\boldsymbol{k}}(t) = \langle u(t),\phi_{\boldsymbol{k}} \rangle = \int_{\Pi} \color{black} u(\boldsymbol{x},t) \color{black} \overline{\phi_{\boldsymbol{k}}}(\boldsymbol{x})\mathrm{d}\boldsymbol{x}, \vspace{-1mm}
\end{equation}
and where the vector of Fourier coefficients $\{\alpha_{\boldsymbol{k}}(t)\}_{\boldsymbol{k}\in\Lambda_n}$ now obey
%$\smash{\{\alpha_{\boldsymbol{k}_i}(t)\}_{i=1}^n}$ 
the finite dimensional SDE %, obtained upon taking the inner product of both sides of the final dimensional projection of the signal equation with $\phi_{\boldsymbol{k}}$, 
\begin{equation}
\mathrm{d} \alpha_{\boldsymbol{k}}(t) = 
\sum_{\boldsymbol{j}\in\Lambda_n}  \lambda_{\boldsymbol{j},\boldsymbol{k}}(\theta) \alpha_{\boldsymbol{j}}(t) \mathrm{d}t+\sum_{\boldsymbol{j}\in\Lambda_n}  \xi_{\boldsymbol{j},\boldsymbol{k}}\eta_{\boldsymbol{j}}(\theta)\mathrm{d}z_{\boldsymbol{j}}(t)~,~~~\boldsymbol{k}\in\Lambda_n, \vspace{-1mm} \label{finite_dim_fourier}
%\mathrm{d} \alpha_{\boldsymbol{k}_i}(t) = 
%\sum_{i,j=1}^n  \lambda_{\boldsymbol{k}_i,\boldsymbol{k}_j}(\theta) \alpha_{\boldsymbol{k}_j}(t) \mathrm{d}t+\sum_{i,j=1}^n  \xi_{\boldsymbol{k}_i,\boldsymbol{k}_j}(\theta)\eta_{\boldsymbol{k}_j}(\theta)\mathrm{d}z_{\boldsymbol{k}_j}(t)  \label{finite_dim_fourier}
\end{equation}
\color{black}
\color{black} with $\lambda_{\boldsymbol{j},\boldsymbol{k}}(\theta)$ and $\xi_{\boldsymbol{j},\boldsymbol{k}}(\theta)$ defined as previously. \color{black} This high dimensional SDE will provide an approximation for the original, infinite dimensional SPDE. \color{black} Indeed, as $n\rightarrow\infty$, one can show that the finite-dimensional approximation $u_n(t)$ does indeed converge in law to the true solution $u(t)$ (see, e.g., \cite{Sigrist2015}).
% One can think of the finite-dimensional projection $u_n(t)$ as an interpolating function, which provides the exact value of the infinite-dimensional signal $u(t)$ at a set of $n^2$ points on $\Pi = [0,1]^2$. One typically assumes that these points lie on a uniformly spaced grid, in which case it is possible to move efficiently between the physical domain and the spectral domain using the discrete Fourier transform (DFT) or the fast Fourier transform (FFT). 
It is convenient to rewrite this equation, as well as the corresponding observation equation, in vector form. 
%In particular, suppose we assign some ordering to the wave-numbers in $\Lambda_n= \{\boldsymbol{k}_i\}_{i=1}^{n^2}$. 
Let $\smash{{\alpha}_{n}(t)= \{\alpha_{\boldsymbol{k}}\}_{\boldsymbol{k}\in\Lambda_n}}$ 
%\alpha_{\boldsymbol{k}_1}(t),\dots, \alpha_{\boldsymbol{k}_{n^2}}(t))^T\in\mathbb{R}^{n^2}}$ 
denote the $n^2$-dimensional vector of Fourier coefficients. We then have
\color{black}
\begin{alignat}{2}
\mathrm{d}{\alpha}_n(t) &= A_n(\theta) {\alpha}_n(t)\mathrm{d}t + B_n\mathrm{d}{{v}}_{n,\theta}(t)~,~~~&&\alpha_n(0)=\alpha_{n,0}, \label{finite_dim_signal} \\
\mathrm{d}{y}_n(t) &= C_n(\theta,\boldsymbol{o}){\alpha}_n(t)\mathrm{d}t+\mathrm{d}{w}_{\boldsymbol{o}}(t)~,~~~&&~y_n(0)=0, \label{finite_dim_obs}
\end{alignat}
where $A_n(\theta)\in\mathbb{R}^{n^2\times n^2}$, $B_n\in\mathbb{R}^{n^2\times n^2}$, and $C_n(\theta,\boldsymbol{o}) = \mathcal{C}(\theta,\boldsymbol{o})|_{\mathbb{R}^{n^2}}\in\mathbb{R}^{n_y\times n^2}$ are the matrices 
\begin{equation}
[A_n(\theta)]_{j,k} = \lambda_{\boldsymbol{j},\boldsymbol{k}}(\theta)~~,~~[B_n]_{j,k} = \xi_{\boldsymbol{j},\boldsymbol{k}}~~,~~[C_n(\theta,\boldsymbol{o})]_{j,k} = \mathcal{C}_j(\theta,\boldsymbol{o})\phi_{\boldsymbol{k}}, \label{matrices}
\end{equation}
and where $v_{n,\theta}(t)$ is the $\mathbb{R}^{n^2}$-valued Wiener process with incremental covariance matrix \linebreak $\smash{Q_n(\theta) = \mathrm{diag}[ \eta_{\boldsymbol{k}_1}^2(\theta),\dots,\eta_{\boldsymbol{k}_{n^2}}^2(\theta)]}$.
\color{black} In our numerical simulations, will typically set $n=50$, so that the simulated observations correspond to noisy realisations of the projection of the true infinite dimensional signal onto an $n^2=2500$ dimensional subspace.

Meanwhile, will apply the finite-dimensional Kalman-Bucy filter and tangent filter, and implement the joint online parameter estimation and optimal sensor placement algorithm, using a reduced Fourier basis of $K\ll n^2$ basis functions. This is typical in similar applications (e.g., \cite{Cressie2011a,Sigrist2015}). In particular, following \cite{Sigrist2015}, we used a reduced Fourier basis given by $\{\phi_{\boldsymbol{k}}\}_{\boldsymbol{k}\in\Gamma_{m,n}}$, where $\Gamma_{m,n}\subseteq \Lambda_n \subset \mathbb{Z}^2$ is the following set of wave-numbers 
\begin{equation}
\Gamma_{m,n}=  \{\boldsymbol{k}\in\mathbb{Z}^2: k_1^2+k_2^2\leq m\} \cap \Lambda_n ~,~~~m\in\mathbb{N}_{0}~,~~~K = |\Gamma_{m,n}|. 
\end{equation}
%We remark that our approximation scheme satisfies all of the conditions required for convergence of the finite-dimensional filter to the true, infinite-dimensional filter (e.g., \cite{Burns2015}).  
\iffalse
If required, we can reconstruct the finite-dimensional approximation of the optimal state estimate as
\begin{equation}
\hat{u}_n(\theta,\boldsymbol{o},t) = \sum_{\boldsymbol{k}\in\mathbb{L}_n} \hat{\alpha}_{\boldsymbol{k}}(\theta,\boldsymbol{o},t)\phi_{\boldsymbol{k}}.
\end{equation}
\fi
 In our simulations, we will generally set $m=5$, which yields a Fourier truncation with $K=21$ basis functions. 
\iffalse
\begin{align}
\Gamma_{5,50} = \left\{\begin{pmatrix} 0 \\ 0 \end{pmatrix}, \begin{pmatrix} 1 \\ 0 \end{pmatrix},\begin{pmatrix} 0 \\ 1 \end{pmatrix},\begin{pmatrix} -1 \\ 0 \end{pmatrix}, \begin{pmatrix} 0 \\ -1 \end{pmatrix} , \begin{pmatrix} 1 \\ 1 \end{pmatrix},\begin{pmatrix} -1 \\ 1 \end{pmatrix},\begin{pmatrix} -1 \\ -1 \end{pmatrix}, \begin{pmatrix} 1 \\ -1 \end{pmatrix} \right. \\
\left. \begin{pmatrix} 2 \\ 0 \end{pmatrix}, \begin{pmatrix} 0 \\ 2 \end{pmatrix},\begin{pmatrix} -2 \\ 0 \end{pmatrix},\begin{pmatrix} 0 \\ -2 \end{pmatrix}, \begin{pmatrix} 2 \\ 1 \end{pmatrix} \begin{pmatrix} 1 \\ 2 \end{pmatrix}, \begin{pmatrix} -1 \\ 2 \end{pmatrix},\begin{pmatrix} -2 \\ 1 \end{pmatrix},\begin{pmatrix} -2 \\ -1 \end{pmatrix}, \begin{pmatrix} -1 \\ -2 \end{pmatrix}
\begin{pmatrix} -1 \\ 2 \end{pmatrix},\begin{pmatrix} -2 \\ 1 \end{pmatrix} \right\}
\end{align}
\fi
Numerical results indicate that, for our purposes, this choice represents a reasonable trade-off between accuracy - both of the optimal sensor placement (see Table \ref{tab:K_comparison}) and the optimal state estimate (see Figure \ref{fig_13a})) - and computational cost. This choice is also comparable with other related works (e.g., \cite{Burns2015,Sigrist2015,Zhang2018}). 
\color{black} 
% Finally, if required, we can reconstruct the finite-dimensional approximation of the optimal state estimate $\hat{u}_n(\theta,\boldsymbol{o},t)$ via
%\begin{equation}
%\hat{u}_n(\theta,\boldsymbol{o},t) = \sum_{i=1}^n \hat{\alpha}_{\boldsymbol{k}_i}(\theta,\boldsymbol{o},t)\phi_{\boldsymbol{k}_i}.
%s\end{equation}
%\subsubsection{Temporal Discretisation}
\color{black} 
Regarding the time discretisation, we use %In practice, it is clearly only possible to measure the observation process at a discrete set of times $\{t_i\}_{i\geq 0}$. %, with $t_0=0$ and $\Delta t_i = t_{i+1}-t_i>0$ for all $i\geq 0$. %We will assume that this partition is equidistant, so that $t_{i+1}-t_i= \delta t$ for all $i\geq 0$. %\footnote{Equivalently, $t_i=i\Delta$ for all $i\geq 0$.} 
%It is thus necessary to discretise the joint online parameter estimation and optimal sensor placement algorithm. We do so using an `indirect approach', which first applies 
an exponential Euler scheme for the finite-dimensional approximation of the partially observed diffusion process \eqref{finite_dim_signal} - \eqref{finite_dim_obs}, and implement the discrete-time analogue of the stochastic gradient descent algorithm \eqref{RML_ROSP_1} - \eqref{RML_ROSP_2}.   %This implies the use of the discrete-time Kalman filter, and the discrete-time log-likelihood (e.g., \cite{Cappe2005}). %We provide further details in Appendix \ref{appendixB}.

\begin{table}[!h]
\color{black}
\caption{\color{black} The mean squared error (MSE) of the approximate optimal sensor placement, and the CPU time per iteration, for different values of the number of basis functions $K$, in two cases of interest. In {Case I}, there are 25 sensors, all of which are movable. In {Case II}, there are 40 sensors, 5 of which are movable. In both cases, the sensors are initially uniformly placed at random over $\Pi = [0,1]^2$, and the algorithm is run for $T=1\times 10^4$ iterations. While increasing the number of basis functions beyond $K=21$ can clearly decrease the error, this choice is adequate to obtain a relatively accurate approximation at a relatively low computational cost, particularly when there are relatively few movable sensors (Case II). \color{black}}
\centering
{\footnotesize
\begin{tabular}{|l|llllllll|}
\hline
                             & $\textbf{K}$                           & 5     & 13   & 21   & 37    & 57    & 81    & 101   \\
\hline
\multirow{2}{*}{\textbf{Case I}}  & MSE ($\times 10^{-2}$) & 5.75  & 4.28 & 2.23 & 2.67  & 1.43  & 0.28  & 0.13  \\
                             & CPU Time per Iteration (s)    
                             & 0.03 %3.77  
                             & 0.05 %& 4.52 
                             & 0.06 %5.92 
                             & 0.11 %11.4 
                             & 0.23 %23.2 
                             & 0.54 %53.5 
                             & 0.87 \\ %86.9 \\
                             \hline
\multirow{2}{*}{\textbf{Case II}} & MSE ($\times 10^{-2}$)                 & 12.49 & 9.07 & 0.32 & 0.16  & 0.05  & 0.02  & 0.01  \\
                             & CPU Time per Iteration (s)    
                             & 0.04 %4.11  
                             & 0.05 %5.04 
                             & 0.07 %6.77 
                             & 0.12 %12.1 
                             & 0.24 %23.7 
                             & 0.52 %51.9 
                             & 0.83 \\%83.1\\
 \hline
\end{tabular}
\label{tab:K_comparison}
}
\end{table}

\begin{figure}[!h]
\color{black}
\vspace{-5mm}
\centering
  \subfloat[Number of basis functions.]{\label{fig_13a}\includegraphics[width=0.46\textwidth]{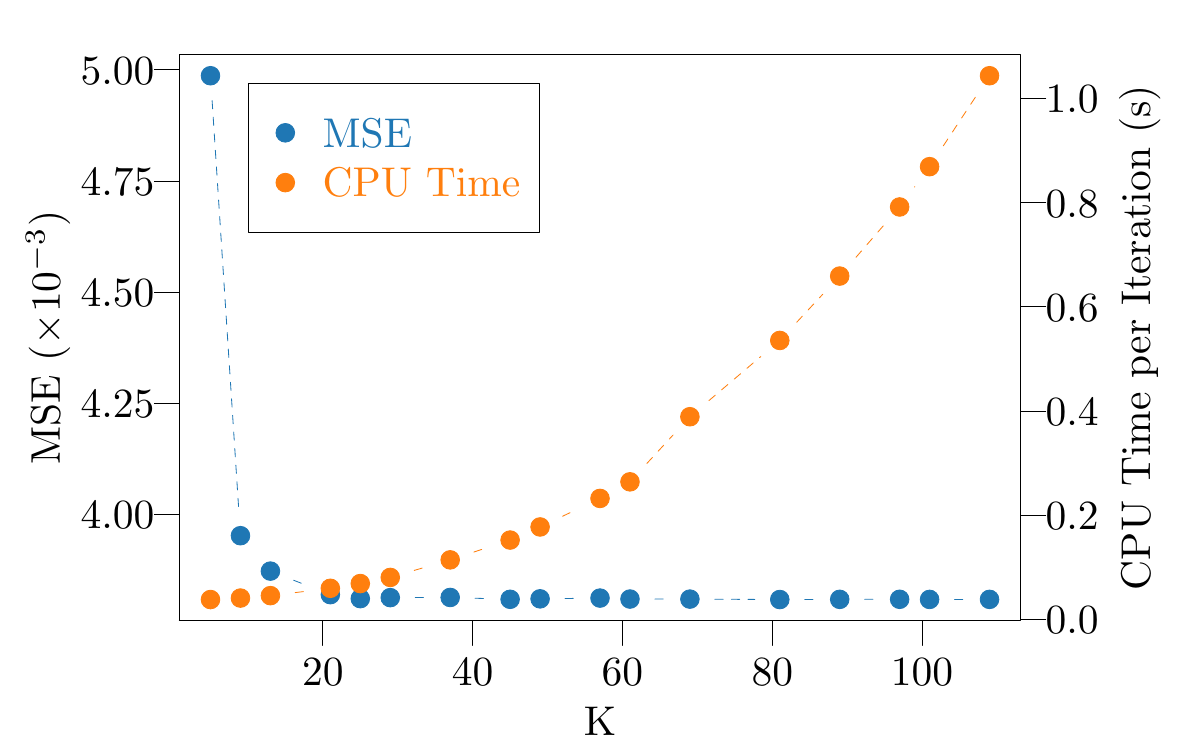}}
  \hspace{1mm}
  \subfloat[Number of movable sensors.]{\label{fig_13b}\includegraphics[width=0.46\textwidth]{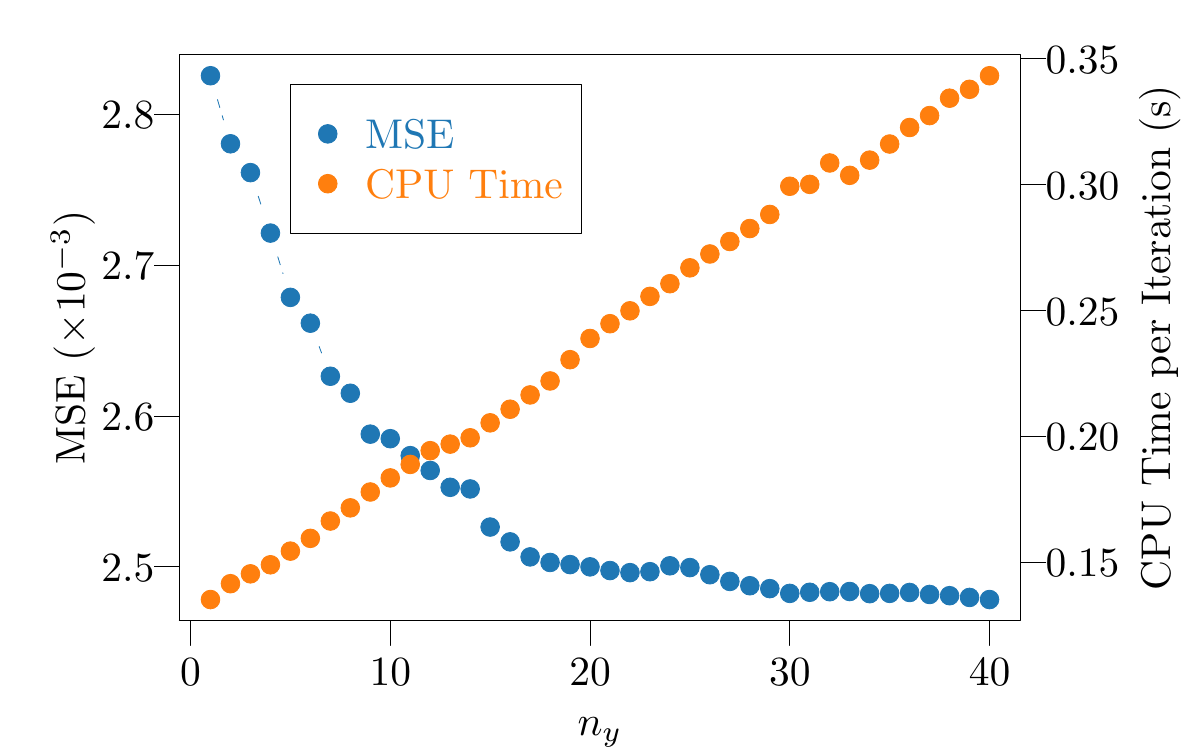}}
\caption{The MSE of the optimal state estimate, and the CPU time (per iteration), for different choices of (a) the number of basis functions $K$ and (b) the number of movable sensors. In (a) we use $25$ sensors, all of which are movable. In (b) we use 40 sensors, varying the number of movable sensors, and use $K=45$ basis functions. In both cases, the sensors are initially uniformly placed at random over $\Pi = [0,1]^2$, and the algorithm is run for $T=1\times 10^4$ iterations.}
\label{fig13}
\vspace{-15mm}
\end{figure}
\color{black}

\color{black}

%\item Discussion of computational advantages under assumptions in Sigrist.
\subsection{Numerical Experiments}

\subsubsection{Simulation I}
We first investigate the convergence of the parameter estimates and the optimal sensor placements under conditions which guarantee convergence to the stationary points of the asymptotic log-likelihood and the asymptotic sensor placement objective function, respectively (see Appendix \ref{appendixA}). We assume that the true model parameters and the initial parameter estimates are given respectively by
\begin{align}
\theta^{*} &= (\rho_0 = 0.50,\sigma^2 = 0.2,\zeta=0.5,\rho_1=0.1,\gamma=2.0,\alpha=\tfrac{\pi}{4},\mu_x=0.30,\mu_y=-0.3,\tau^2 = 0.01), \label{true_model_param} \nonumber \\[1mm]
\theta_0 &= ({\rho}_{0} = 0.25,{\sigma}^2 = 0.8,{\zeta}=0.1,{\rho}_{1}=0.2,{\gamma}=1.2,{\alpha}=\tfrac{\pi}{3},{\mu}_{x}=0.1,{\mu}_{y}=-0.15,{\tau}^2 = 0.10).  \nonumber
\end{align}
We also assume that we have $n_{y}=8$ sensors in $\Pi=[0,1]^2$. We suppose that the sensors are independent, have zero bias, and generate noisy measurements with variance $\tau^2$. Thus, in the observation equation, we have $\beta = (0,\dots,0)^T$ and $\mathcal{R} = \diag(\tau^2)$.
\iffalse
\begin{equation}
\beta = (0,\dots,0)^T~~~,~~~\mathcal{R} = \diag(\tau^2)% \begin{pmatrix} 0 \\ \vdots \\ 0\end{pmatrix}~~~,~~~\mathcal{R} = \begin{pmatrix} \tau^2 & \dots & 0 \\ \vdots & \ddots & \vdots \\ 0 & \cdots & \tau^2\end{pmatrix}
\end{equation}
\fi
In this test simulation, in order to verify the convergence of our algorithm, we suppose that our objective is to obtain the optimal sensor placement with respect to the state estimate at a set of \color{black} known `{target}' \color{black} locations. \color{black} This is achieved by choosing an operator $\mathcal{M}$ in the sensor placement objective function $\tilde{\mathcal{J}}(\theta,\boldsymbol{o})$, c.f. \eqref{obj_func}, which places greater emphasis on minimising the uncertainty in the state estimate at the target locations. We provide an explicit definition of this operator in Appendix \ref{app:weighting}. \color{black} In particular, we assume that the target sensor locations and the initial sensor locations are given, respectively, by 
\begin{align}
{{\boldsymbol{o}}}^{*} &= \left\{\begin{pmatrix} 0.00 \\ 0.59\end{pmatrix},\begin{pmatrix} 0.50 \\ 0.66 \end{pmatrix},\begin{pmatrix} 0.33 \\ 0.33 \end{pmatrix}, \begin{pmatrix} 0.75 \\ 0.50 \end{pmatrix}, \begin{pmatrix} 0.08 \\ 0.08 \end{pmatrix}, \begin{pmatrix} 0.58\\ 0.83 \end{pmatrix}, \begin{pmatrix} 0.83 \\ 0.92 \end{pmatrix}, \begin{pmatrix} 0.25 \\ 0.83 \end{pmatrix}\right\}, \label{targets} \\[1mm]
{\boldsymbol{o}}_0 &= \left\{\begin{pmatrix} 0.84 \\ 0.65 \end{pmatrix},\begin{pmatrix} 0.34 \\ 0.50 \end{pmatrix},\begin{pmatrix} 0.43 \\ 0.31 \end{pmatrix}, \begin{pmatrix} 0.60 \\ 0.34 \end{pmatrix}, \begin{pmatrix} 0.27 \\ 0.26 \end{pmatrix}, \begin{pmatrix} 0.51 \\ 0.18 \end{pmatrix}, \begin{pmatrix} 0.08 \\ 0.23 \end{pmatrix}, \begin{pmatrix} 0.25 \\ 0.08 \end{pmatrix}\right\}. 
\end{align}
It remains to specify the learning rates $\{\gamma^{i}_{\theta}(t)\}_{t\geq 0}^{i=1,\dots,9}$ and $\{\gamma^{j}_{\boldsymbol{o}}(t)\}^{j=1,\dots,8}_{t\geq 0}$, where the indices $i,j$ now make explicit the fact that the step sizes are permitted to vary between parameters, and between sensors. In this simulation, we assume that our primary objective is to estimate the true model parameters, and our secondary objective is to optimally place the measurement sensors. We thus set $\smash{\gamma_{\theta}^{i}(t) = \gamma^{i}_{\theta,0}t^{-\varepsilon^i_{\theta}}}$ and $\smash{\gamma_{\boldsymbol{o}}^{j}(t) = \gamma^{j}_{\theta,0}t^{-\varepsilon^j_{\boldsymbol{o}}}}$, where $\gamma^{i}_{\theta,0},\gamma^{j}_{\boldsymbol{o},0}>0$ and $\smash{0.5<\varepsilon_{\boldsymbol{o}}^j<\varepsilon_{\theta}^{i}\leq1}$ for all $i=1,\dots,9$ and $j=1,\dots,8$, with the values of $\smash{\gamma^{i}_{\theta,0}}$, $\smash{\varepsilon_{\theta}^i}$, $\smash{\gamma^{j}_{\boldsymbol{o},0}}$, and $\smash{\varepsilon_{\boldsymbol{o}}^{j}}$ tuned individually. \color{black} In our numerics, the specific values of the learning rates are chosen on the basis of initial experiments.  In principle, however, one can use any one of a number of adaptive learning rate methods to automate this choice, including backtracking line search, Adagrad \cite{Duchi2011}, Adadelta \cite{Zeiler2012}, Adam \cite{Kingma2015}, AMSgrad \cite{Reddi2018}, and others.  \color{black} This choice of learning rate satisfies all of the conditions of \cite[Proposition 3.1]{Sharrock2020a}. In particular, it guarantees that
\begin{equation}
\lim_{t\rightarrow\infty} \frac{\gamma_{\theta}^i(t)}{\gamma_{\boldsymbol{o}}^j(t)} = 0~~~\forall i=1,\dots,9~,~j=1,\dots,8. \label{learning_rate_1}
\end{equation}
This implies that the parameter estimates $\{\theta(t)\}_{t\geq 0}$ move on a slower timescale than the sensor placements $\{\boldsymbol{o}(t)\}_{t\geq 0}$. Thus, the sensor placements see the parameter estimates as quasi-static, while the parameter estimates see the sensor placements as almost equilibrated. In practice, this means that $\boldsymbol{o}(t)$ will asymptotically track 
%$\boldsymbol{o}(\theta(t))$, 
the sensor placements which are optimal with respect to the \emph{current} parameter estimates. This is particularly advantageous when the optimal sensor placement depends significantly on the parameters (see Section \ref{sim2}). 

The performance of the two-timescale stochastic gradient descent algorithm is visualised in Figures \ref{fig1}, in which we plot the sequence of online parameter estimates and optimal sensor placements, Figure \ref{fig1_KF}, in which we plot a single component of the hidden state used to generate the observations, and the optimal state estimate, and in Figure \ref{fig1_MSE}, in which we plot the time evolution of the mean squared error (MSE) for the corresponding filter. As expected, all of the parameter estimates converge to within a small neighbourhood of their true values (Figure \ref{fig_1a}), and all of the sensors converge to one of the target locations (Figure \ref{fig_1b}). As a result, the performance of the filter is improved to near-optimal after approximately $T = 2 \times 10^{4}$ iterations (Figure \ref{fig1_MSE}). This number is largely determined by the initial magnitudes of the learning rates $\smash{\{\gamma^{i}_{\theta}(t)\}_{t\geq 0}^{i=1,\dots,9}}$ and $\smash{\{\gamma^{j}_{\boldsymbol{o}}(t)\}^{j=1,\dots,8}_{t\geq 0}}$. In particular, increasing one or more of these values will often decrease the time taken for the algorithm iterates to converge. %We should remark, however, that as with any (stochastic) gradient method, the convergence of our algorithm is somewhat sensitive to the choice of learning rates.
%[\textbf{To add}]%: other plots (e.g., gradient of log-likelihood and objective function as a function of time, objective function as a function of number of sensors and the sensor noise, state estimate as a function of time).]

\begin{figure}[!h]
\captionsetup[subfloat]{captionskip=2pt}
\centering
  \subfloat[Online parameter estimates.]{\label{fig_1a}\includegraphics[width=0.49\textwidth]{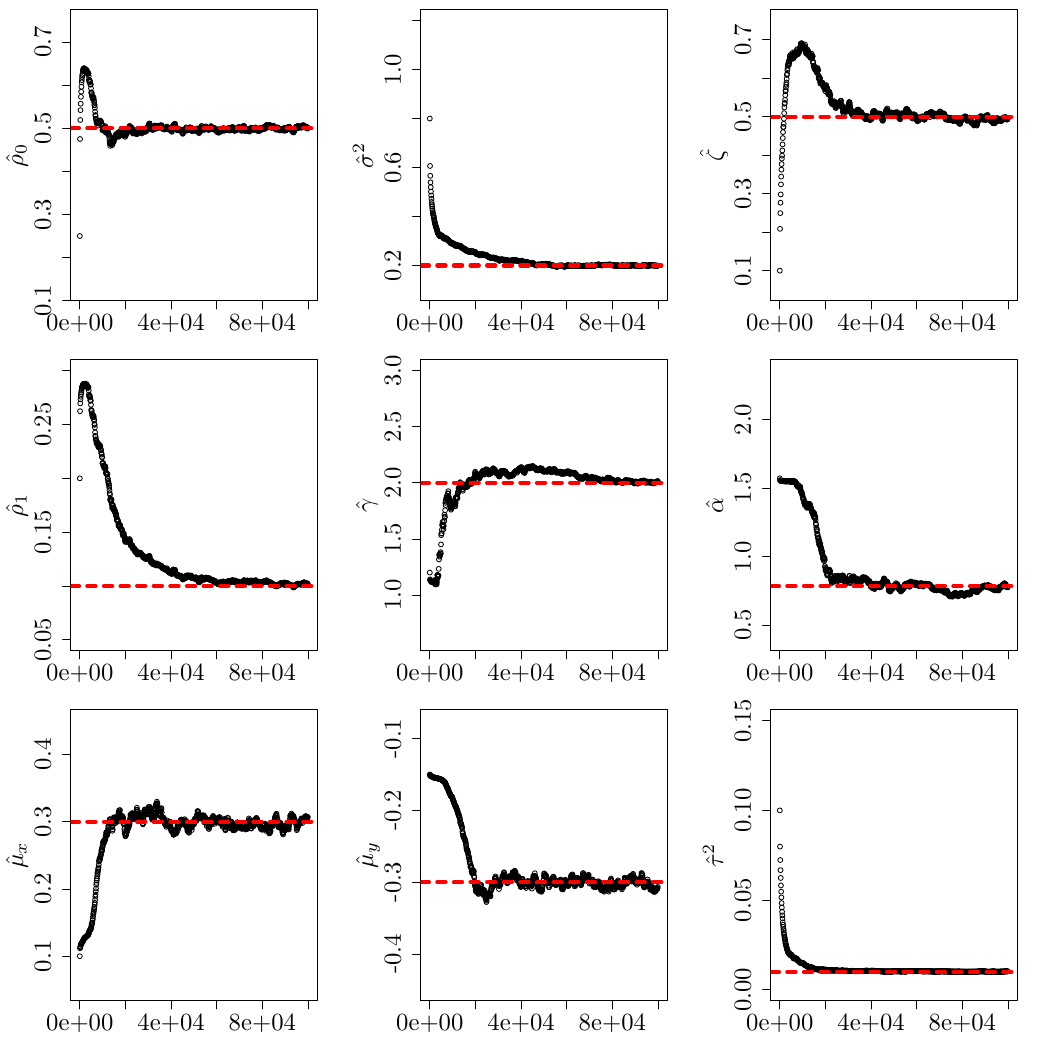}}
  \hspace{1mm}
  \subfloat[Optimal sensor placements.]{\label{fig_1b}\includegraphics[width=0.49\textwidth]{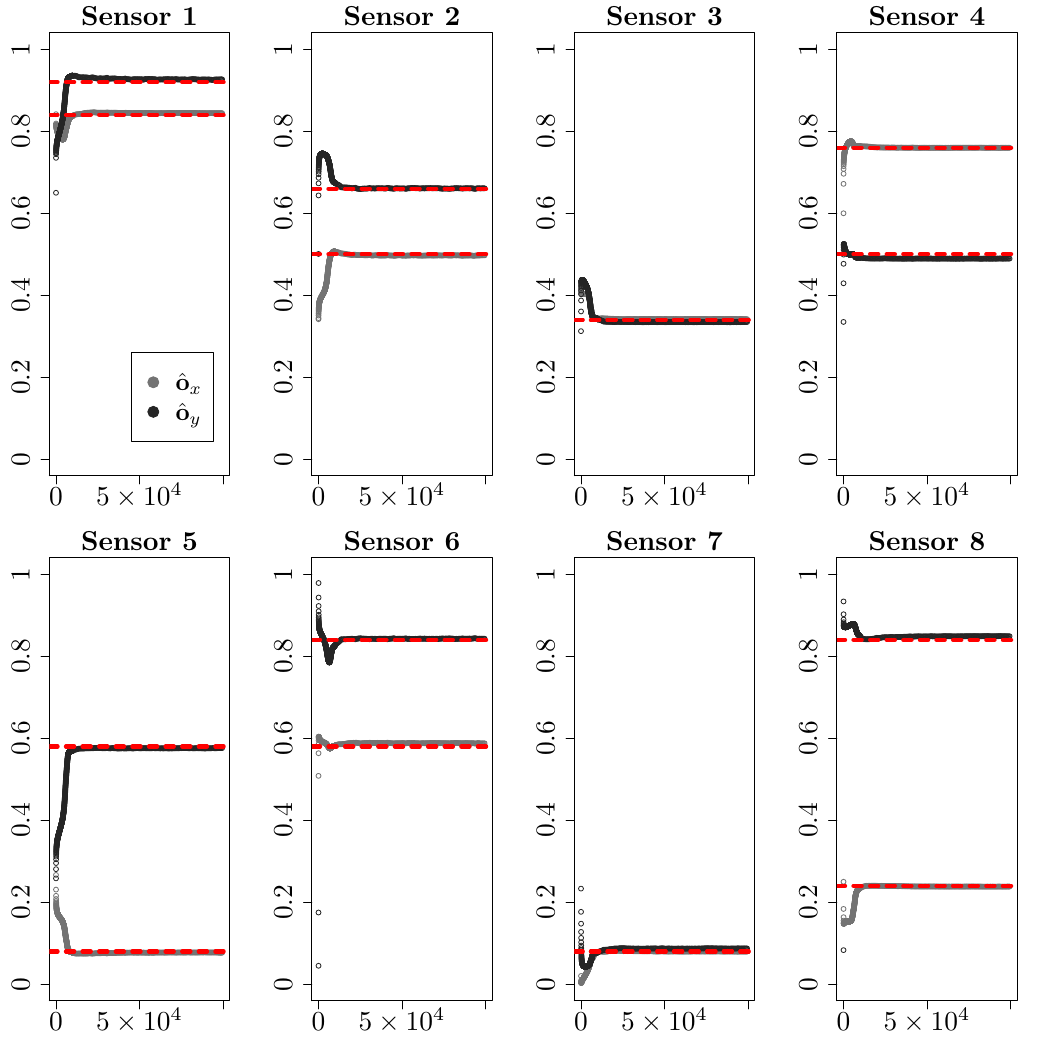}}
  \vspace{-2mm}
\caption{Simulation Ia. The online parameter estimates \& the optimal sensor placements (black); and the true parameters \& optimal sensor placements (red, dashed). In this simulation, the parameter estimates move on the slower timescale, and the sensor placements move on the faster timescale. \color{black}The total CPU time required for this simulation is 2368 seconds (0.02368 seconds per iteration).\color{black}}
\label{fig1}
\end{figure}

\begin{figure}[!h]
\captionsetup[subfloat]{captionskip=2pt}
\centering
  \subfloat[The entire learning period.]{\label{fig1_KF_a}\includegraphics[width=0.7\textwidth]{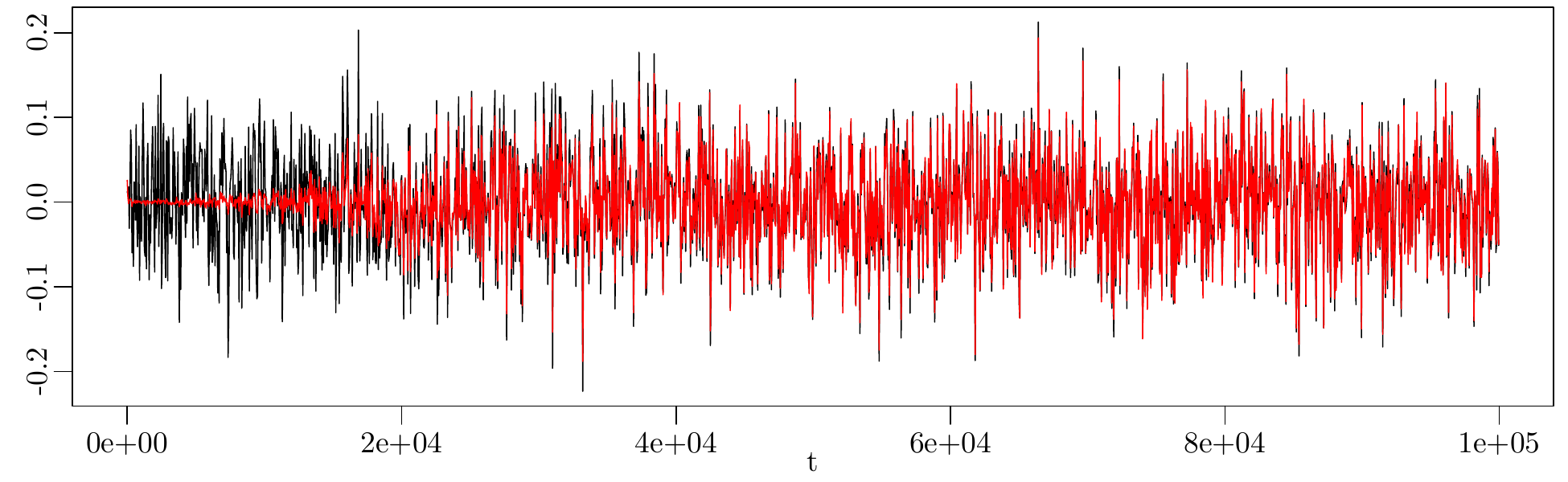}} \\[-3mm]
  \hspace{3mm}\subfloat[The first 200 iterations.]{\label{fig1_KF_b}\includegraphics[width=0.34\textwidth]{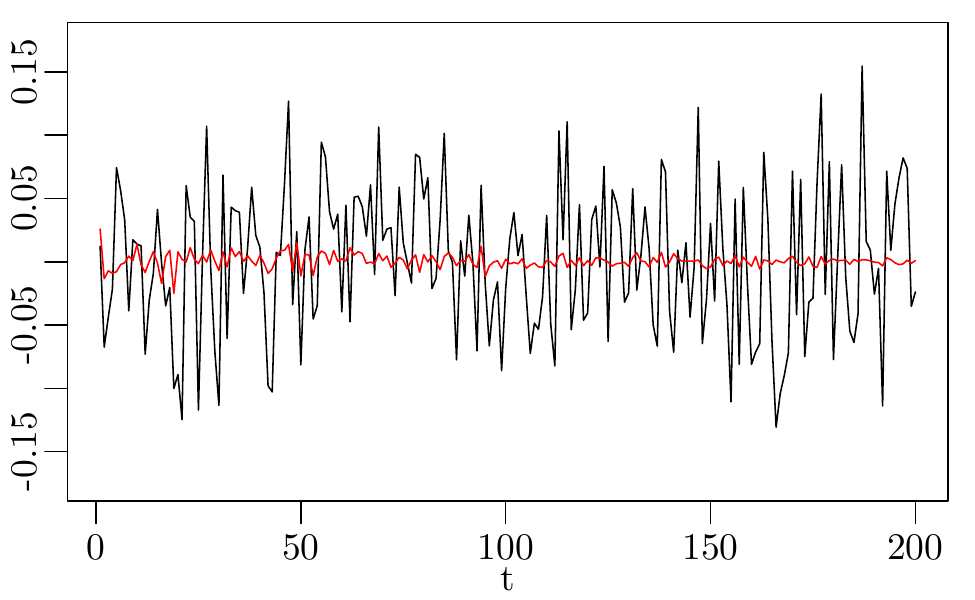}}
  \hspace{1mm}
  \subfloat[The final 200 iterations.]{\label{fig1_KF_c}\includegraphics[width=0.34\textwidth]{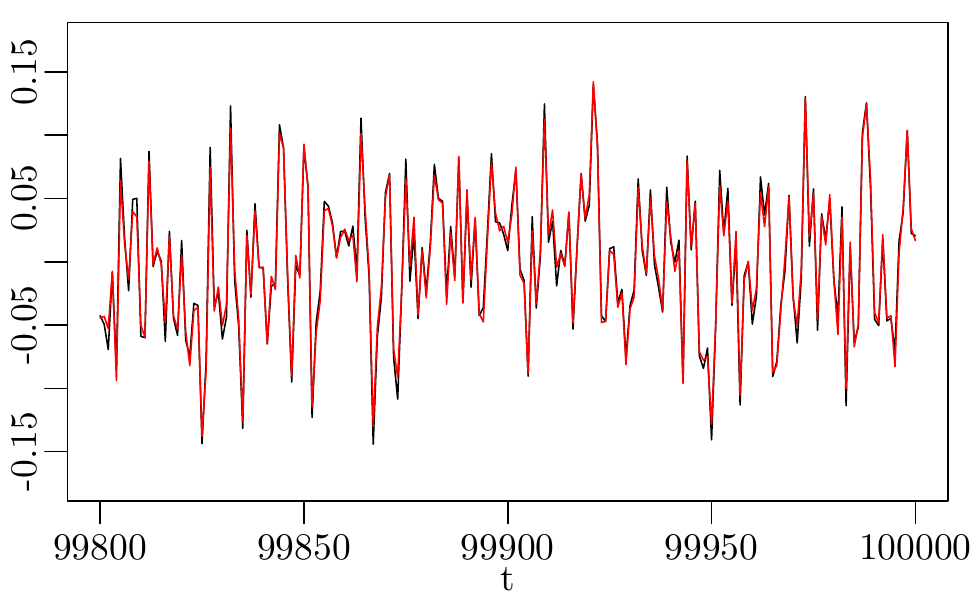}}
\caption{Simulation Ia. A component of the hidden state $\alpha_n(t)$ (black) and the state estimate $\hat{\alpha}_n(t)$ (red).}
\label{fig1_KF}
\end{figure}

\begin{figure}[!h]
\centering
  \includegraphics[width=.7\textwidth]{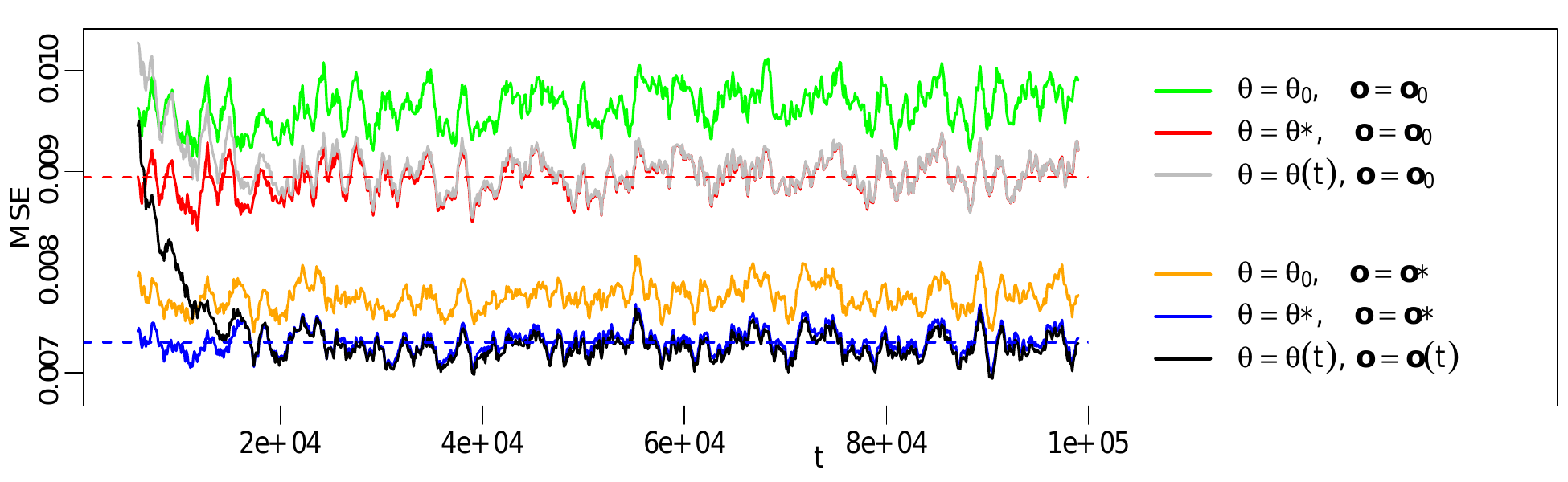}
  \vspace{-4mm}
\caption{Simulation Ia. The moving average of the MSE of the optimal state estimate under different learning scenarios (various colours). We also plot the average of the MSE for the true parameters and the initial sensor placement (red, dashed), and the average of the MSE for the true parameters and the optimal sensor placement (blue, dashed). The MSE is calculated at $n^2 = 50^2$ uniformly spaced grid points on $\Pi = [0,1]^2$.}
\label{fig1_MSE}
\end{figure}

It is also possible to apply our algorithm when the primary objective is to obtain the optimal sensor placement, and the secondary objective is to estimate the true model parameters. That is, the order of the two optimisation problems is reversed. In particular, this is achieved 
%This number can also depend, however, on the relative asymptotics of the learning rates for the parameter estimates and the optimal sensor placements. In particular, it may be possible to decrease the time to convergence 
by choosing learning rates which no longer satisfy \eqref{learning_rate_1}, but instead satisfy
\begin{equation}
\lim_{t\rightarrow\infty} \frac{\gamma_{\boldsymbol{o}}^i(t)}{\gamma_{\theta}^j(t)} = 0~~~\forall i=1,\dots,9~,~j=1,\dots,8. \label{learning_rate_2}
\end{equation}
This implies, of course, that the sensor placements $\{\boldsymbol{o}(t)\}_{t\geq 0}$ now move on a slower timescale than the parameter estimates $\{\theta(t)\}_{t\geq 0}$. The performance of the two-timescale stochastic gradient descent algorithm in this scenario, with all other assumptions unchanged from the first simulation, is illustrated in Figure \ref{fig1_reverse}. Once more, we observe that all of the parameter estimates converge to within a small neighbourhood of their true values, and all of the sensors converge to one of the target locations. Unsurprisingly, given the alternative specification of the learning rates, the convergence of the parameter estimates is somewhat faster than before, while the convergence of the optimal sensor placements is somewhat slower. 

\begin{figure}[!h]
\centering
  \subfloat[Online parameter estimates.]{\label{fig_1a_reverse}\includegraphics[width=0.49\textwidth]{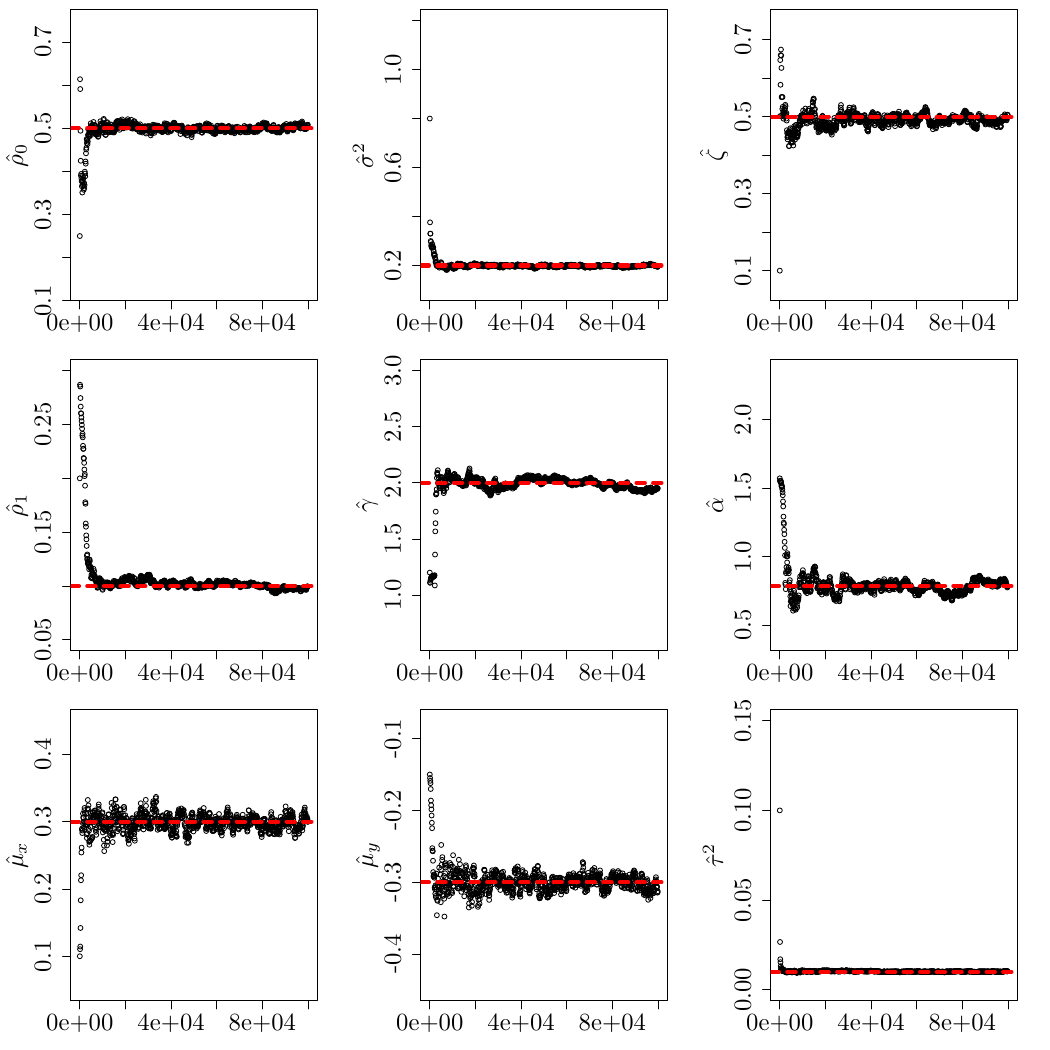}}
  \hspace{1mm}
  \subfloat[Optimal sensor placements.]{\label{fig_1b_reverse}\includegraphics[width=0.49\textwidth]{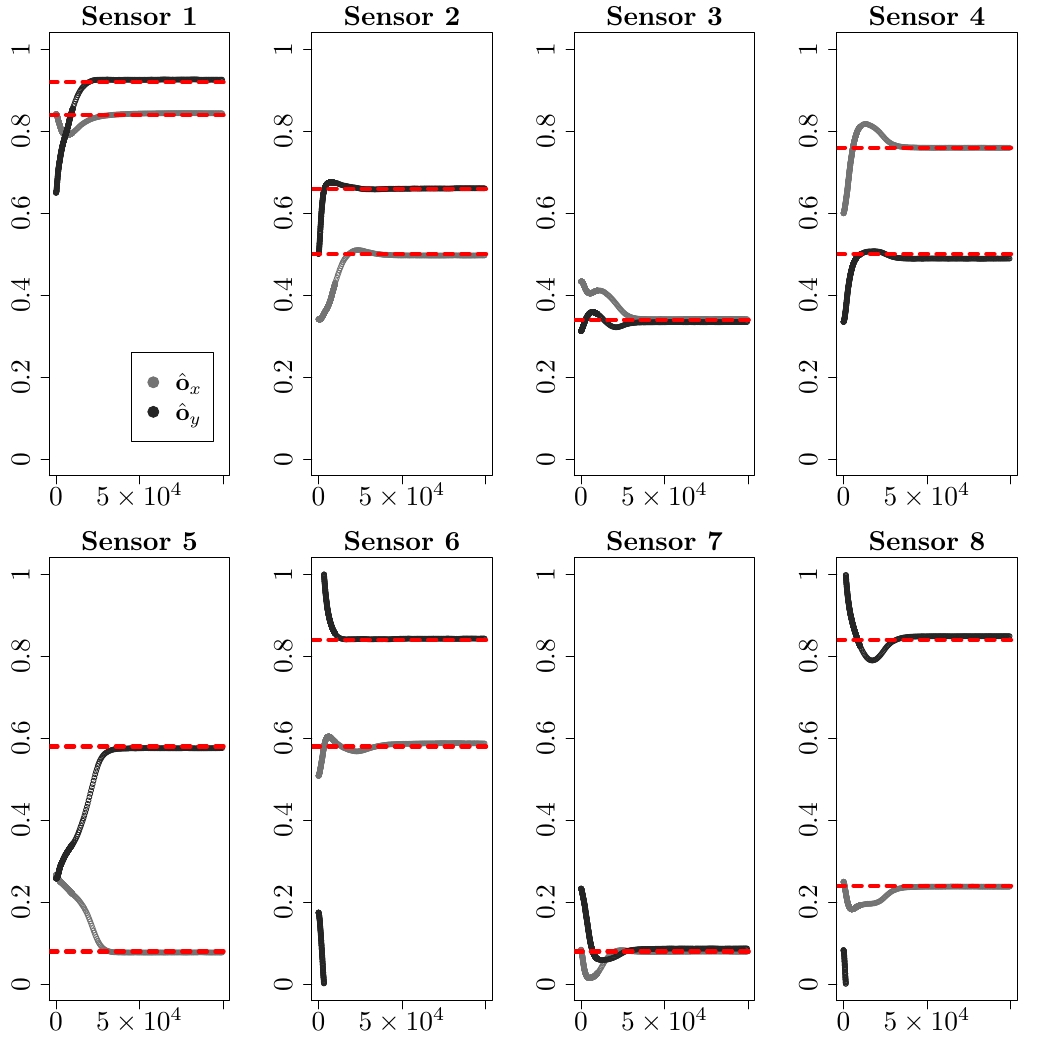}}
  \vspace{-2mm}
\caption{Simulation Ib. The online parameter estimates \& the optimal sensor placements (black); and the true parameters \& optimal sensor placements (red, dashed). In this simulation, the parameter estimates move on the faster timescale, and the sensor placements move on the slower timescale.}
\label{fig1_reverse}
\end{figure}

It is worth re-emphasising, at this stage, that the convergence of our algorithm does \emph{not} depend on whether the learning rates satisfy \eqref{learning_rate_1} or \eqref{learning_rate_2}. That is to say, \emph{a priori}, the algorithm has no preference over which of the parameter estimates or the optimal sensor placements moves on the faster time scale, and which moves on the slower time scale. This is a clear advantage of the two-timescale approach.

\subsubsection{Simulation II} \label{sim2}
In our second numerical experiment, we investigate the performance of our algorithm in a scenario where the optimal sensor placement depends to a significant extent on the value of one of the model parameters. In this simulation, we will assume that the values of $\theta_{2:9} = (\sigma^2,\zeta,\rho_1,\gamma,\alpha,\mu_x,\mu_y,\tau^2)$ are known, and fixed equal to their true values, while the value of $\theta_1 = \rho_0$ is unknown. The true value and the initial value of the unknown parameter are given by ${\rho}_0^{*} = 0.3$ and $\rho_0 = 0.01$, respectively. 

We now assume that we have $n_y = 5$ sensors in $\Pi = [0,1]^2$, all of which are independent, have zero bias, and the same variance. The locations of the first 4 sensors are fixed, while the location of the final sensor is to be optimised. In contrast to the previous simulation, we now suppose our objective is to obtain the optimal sensor placement with respect to the state estimate over the entire spatial domain \color{black} (i.e., not weighted towards a set of target locations). This is achieved by setting the spatial weighting operator $\mathcal{M} = \mathds{1}$ (i.e., the identity operator) in the sensor placement objective function. \color{black} The locations of the fixed sensors, and the initial location of the sensor whose location is to be optimised, are shown in Figure \ref{fig_sim2_plot2}b.

It remains to specify the learning rates $\smash{\{\gamma_{\rho_0}(t)\}_{t\geq 0}}$ and $\smash{\{\gamma_{\boldsymbol{o}_5}(t)\}_{t\geq 0}}$. In this case, we set $\smash{\gamma_{\rho_0}(t)= 0.1t^{-0.55}}$ and $\smash{\gamma_{\boldsymbol{o}_5}(t)= 0.1t^{-0.51}}$, implying that the sensor placements move on a faster timescale than the parameter estimates. % Thus, broadly speaking, the sensor placements see the parameter estimates as quasi-static, while the parameter estimates see the sensor placements as almost equilibrated. 
Thus, as outlined previously, $\boldsymbol{o}(t)$ should asymptotically track $\smash{\boldsymbol{o}^{*}(\rho_0(t))}$, the sensor placement which is optimal with respect to the current parameter estimate. This is clearly advantageous in the current scenario, in which the optimal sensor placement is known to depend on the unknown model parameter. This is clearly visualised in Figure \ref{fig_sim2}, which contains plots of the asymptotic sensor placement objective function, and the corresponding optimal sensor placement, for several different values of the unknown model parameter. For this configuration of fixed sensors, we observe that the optimal location of the additional sensor is to the south-east (or north-west) of centre for small $\rho_0$ (Figure \ref{fig_sim2a}), and converges to the centre as $\rho_0$ increases (Figures \ref{fig_sim2_b} - \ref{fig_sim2_d}).

The performance of the two-timescale gradient descent algorithm is illustrated in Figure \ref{fig_sim2_plot2}, in which we have plotted the sequence of online parameter estimates $\{\rho_0(t)\}_{t\geq 0}$ and optimal sensor placements $\{\boldsymbol{o}_5(t)\}_{t\geq 0}$. Unsurprisingly, the online parameter estimate, on the slow-timescale, is seen to converge to the true value of $\rho_0^{*} = 0.3$ over the course of the entire learning period.  Meanwhile, the optimal sensor placement, on the fast-timescale, begins by moving rapidly from its initial position to a location to the south-east of centre.
 %This corresponds precisely to a local optima of the sensor placement objective function for small $\rho_0$ (Figure \ref{fig_sim2a}). 
It then moves slowly towards the centre of the domain as the online parameter estimate of $\rho_0$ increases towards its true value. 
%Thus, in particular, it does indeed track the local optima of the sensor placement objective function (Figures \ref{fig_sim2_b} - \ref{fig_sim2_d}).
Thus, the optimal sensor placement does indeed track the local optima of the sensor placement objective function, while the online parameter estimate converges to its true value. % This is a clear advantage of the two-timescale approach. 

\begin{figure}[!h]

\captionsetup[subfloat]{captionskip=6pt}
\centering
\subfloat[$\rho_0 = 0.03$]{\label{fig_sim2a}\includegraphics[width=.18\linewidth]{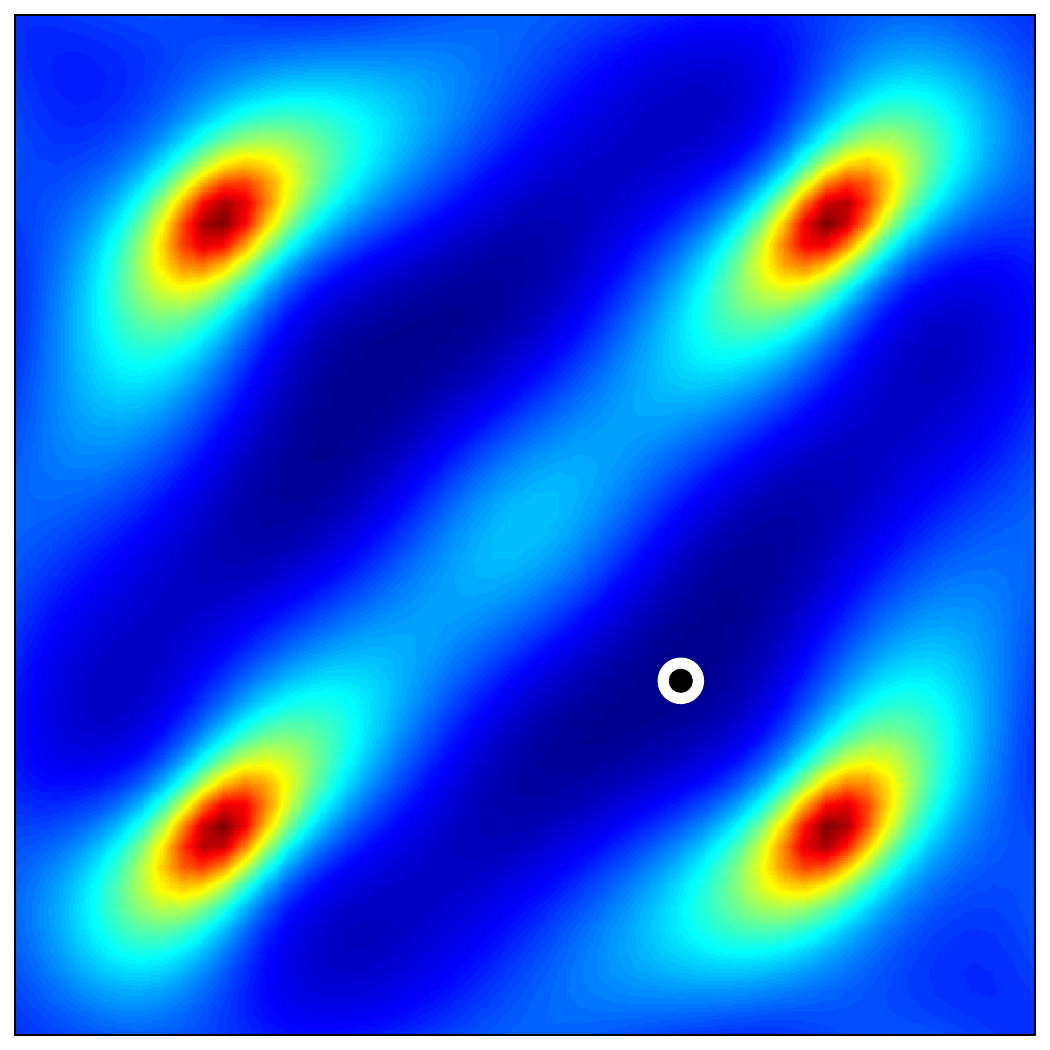}}
  \hspace{1.5mm}
 \subfloat[$\rho_0 = 0.10$]{\label{fig_sim2_b}\includegraphics[width=.18\linewidth]{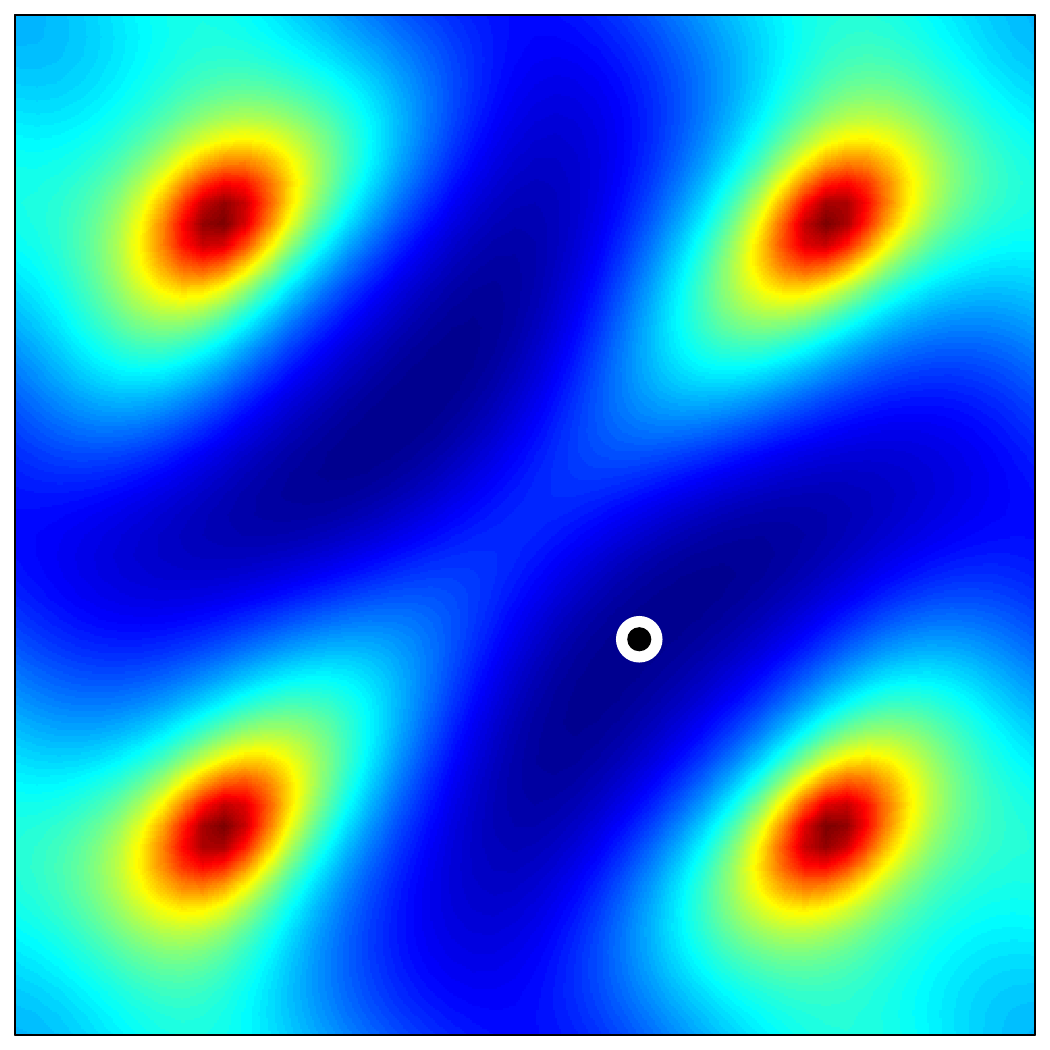}}
 \hspace{1.5mm}
  \subfloat[$\rho_0 = 0.15$ ]{\label{fig_sim2_c}\includegraphics[width=.18\linewidth]{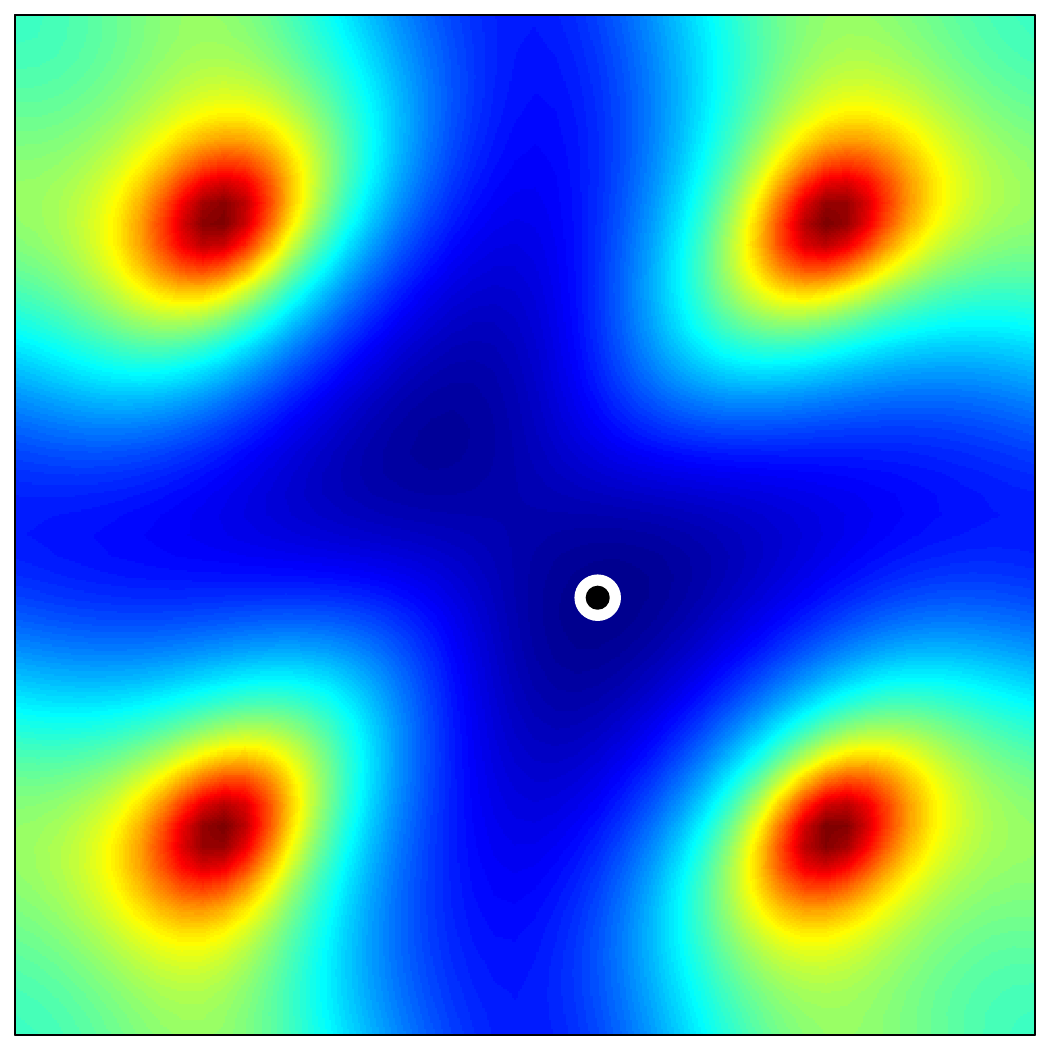}}
 \hspace{1.5mm}
  \subfloat[$\rho_0 = 0.20$]{\label{fig_sim2_d}\includegraphics[width=.18\linewidth]{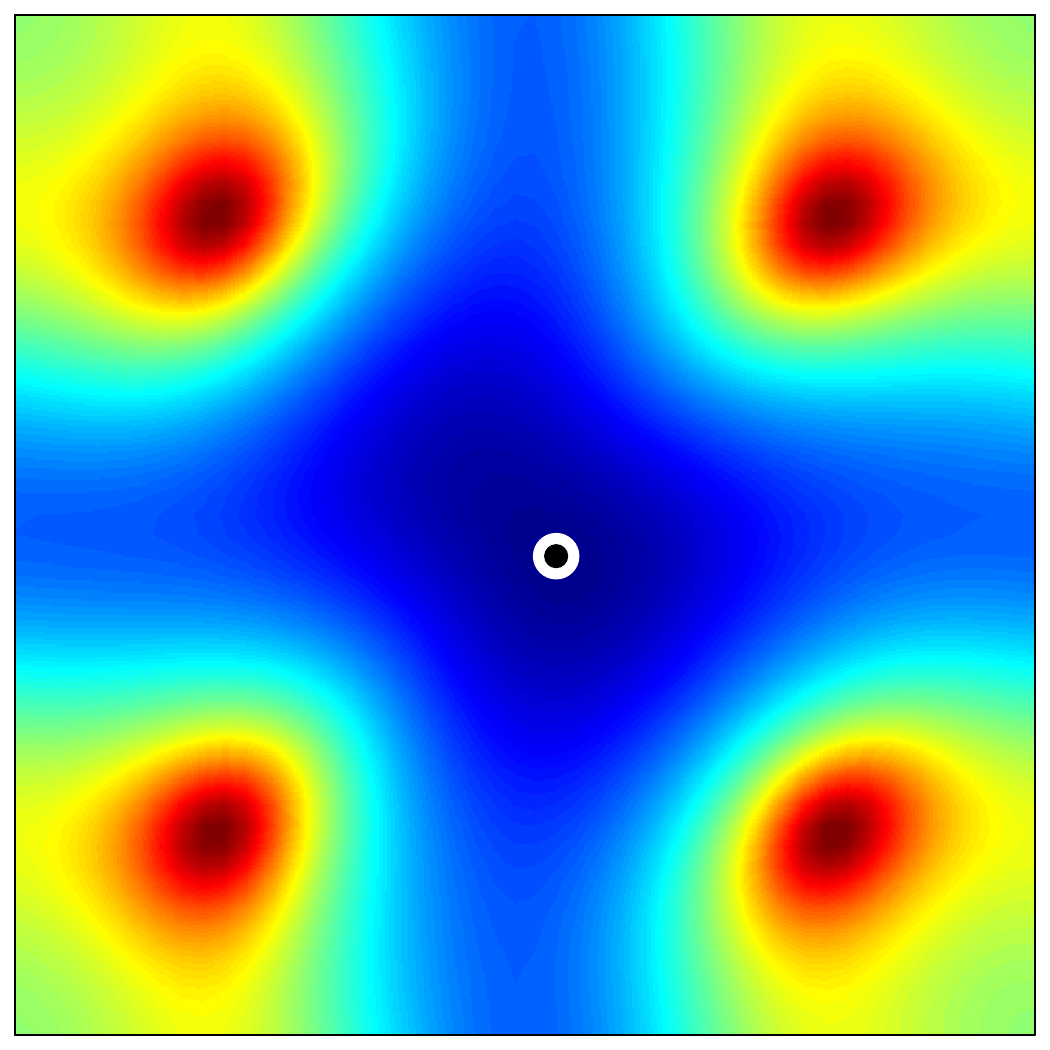}}
  \vspace{-2mm}
\caption{Simulation II. Heat maps of the sensor placement objective function, and the optimal sensor placement, for different values of $\rho_0$.}
\label{fig_sim2}
\vspace{-2mm}
\end{figure}

\begin{figure}[!h]
\captionsetup[subfloat]{captionskip=2pt}
\centering
  \subfloat[Online parameter estimates.\hspace{12mm} (b) Optimal sensor placements. ] {\label{fig_sim2_plot2_a}
  \hspace{-5mm}
  \includegraphics[width=0.86\textwidth]{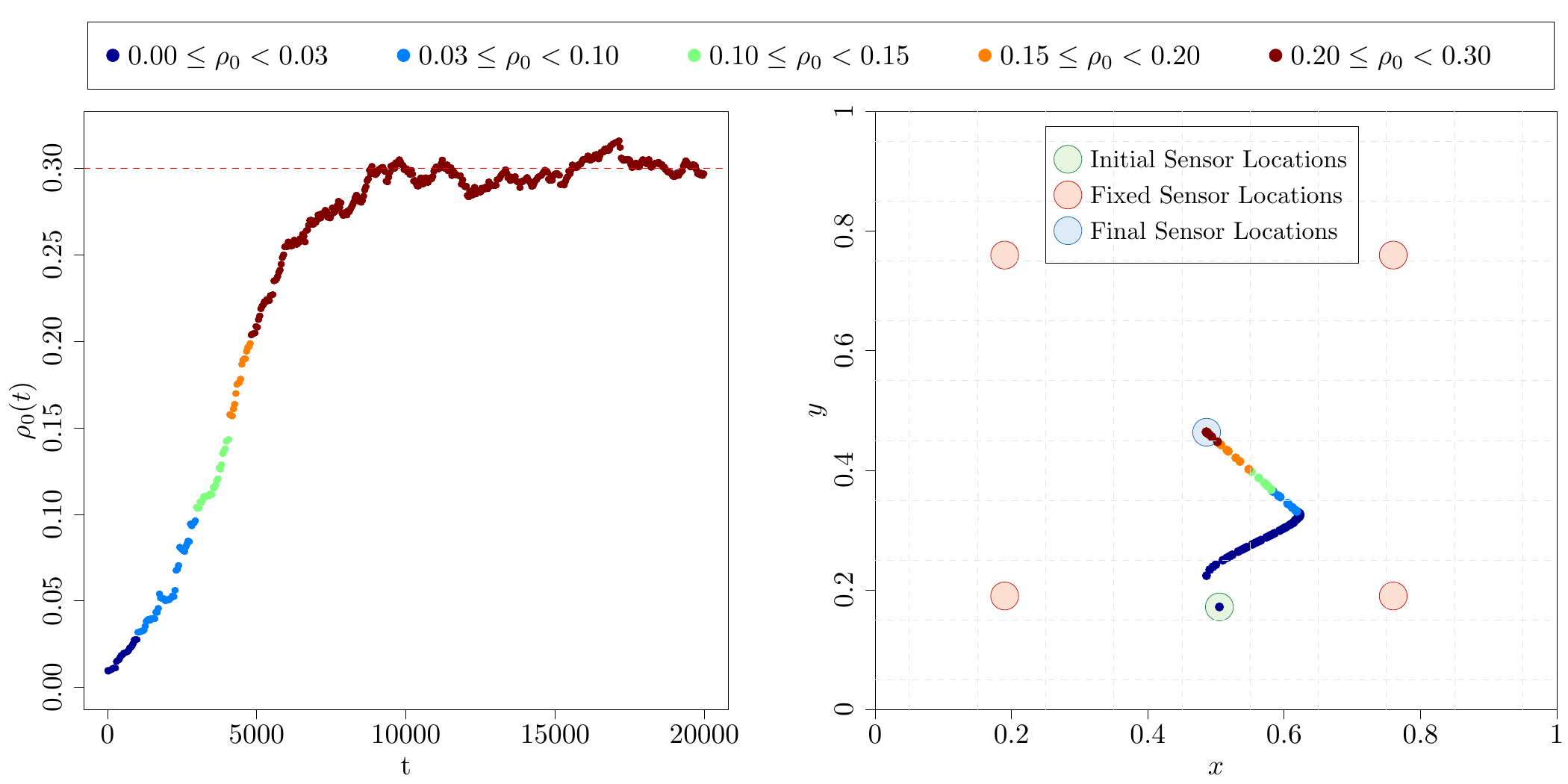}} 
  \hspace{1mm}
  \vspace{-2mm}
\caption{Simulation II. The online parameter estimates \& the optimal sensor placements (various colours); and the true parameter (red, dashed). 
\color{black} The total CPU time required for this simulation is 185s (0.00925s per iteration). \color{black}}
\label{fig_sim2_plot2}
\end{figure}

\color{black} We should note, at this point, that the asymptotic sensor placement objective function (and the asymptotic log-likelihood function) can admit multiple local optima (see Figure \ref{fig_sim2}). There is thus no guarantee that the sensor placements (or the parameter estimates) generated by our algorithm will always converge to the global optimum (i.e., the true optimal sensor placements or the true parameter values). On this point, let us make several remarks. Firstly, this is a necessary feature of any gradient based method; such methods are only guaranteed to converge to a global optimum under the rather restrictive assumption of global convexity (see \cite{Sharrock2020a}). 
%For the problem at hand, this should be seen as a reasonable trade-off for the combinatorial complexity of a brute-force approach. 
Secondly, we use a {stochastic} gradient descent method, updating the sensor placements and the parameter estimates in the directions of noisy estimates of the gradients of the asymptotic sensor placement objective and the asymptotic log-likelihood. In comparison to a (non-stochastic) gradient descent scheme, this approach is significantly more likely to avoid local minima and saddle points (e.g., \cite{Bottou2010,Ge2015}). One can also use momentum \cite{Qian1999} or additional random noise to help to escape local minima. Finally, if required, one can run the algorithm multiple times using random restarts.  

\color{black}

\subsubsection{Simulation III} 
In our third numerical experiment, we investigate the performance of our algorithm under the assumption that the true model parameters $\theta^{*} = \theta^{*}(t)$ are no longer static, and contain change-points at certain points in time. The values of these parameters are shown in Figure \ref{fig2}. Meanwhile, the initial parameter estimates are now given by
\begin{align}
\smash{\theta_0 = ({\rho}_{0} = 0.25,{\sigma}^2 = 0.5,{\zeta}=0.3,{\rho}_{1}=0.2,{\gamma}=1.5,{\alpha}=\tfrac{\pi}{3},{\mu}_{x}=0.1,{\mu}_{y}=-0.15,{\tau}^2 = 0.1). } \nonumber
\end{align}
\color{black}
We also now suppose that the optimal sensor locations $\boldsymbol{o}^{*} = \boldsymbol{o}^{*}(t)$ vary in time. In particular, in this simulation we consider a scenario in which our objective is to obtain the sensor placement which minimises the uncertainty in the state estimate over the entire spatial domain, but now weighted slightly towards a set of four time-varying spatial locations. Once again, this is achieved by a suitable choice of spatial weighting operator in the sensor placement objective function, whose explicit definition we provide in Appendix \ref{app:weighting}. We assume that we have $n_y = 25$ sensors in $\Pi = [0,1]^2$, each with zero bias and equal variance. The first 16 sensors are distributed evenly towards the boundary of the spatial domain, with their locations fixed. Meanwhile, the locations of the final 9 sensors are to be optimised. We show the locations of the fixed sensors (red), the initial sensor locations of the nine sensors to be optimised (green) and the weighted spatial locations at four different time points (purple) in Figure \ref{fig2g}. 
\color{black}
\iffalse
, respectively, by 
\begin{align}
\hat{{\boldsymbol{o}}} &= \frac{1}{6}\left\{\begin{pmatrix} 1.00 \\ 1.00\end{pmatrix},\begin{pmatrix} 2.00 \\ 5.00\end{pmatrix},\begin{pmatrix} 5.00 \\ 3.00\end{pmatrix},\begin{pmatrix} 4.00 \\ 1.00\end{pmatrix}\right\},  \\[2mm]
{\boldsymbol{o}}_0 &= \frac{1}{6}\left\{\begin{pmatrix} 3.40 \\ 3.40\end{pmatrix},\begin{pmatrix} 3.40 \\ 4.10\end{pmatrix},\begin{pmatrix} 4.10 \\ 3.40\end{pmatrix},\begin{pmatrix} 4.10 \\ 4.10\end{pmatrix}\right\}. 
\end{align}
\fi

It remains, once more, to specify the learning rates $\{\gamma^{i}_{\theta}(t)\}_{t\geq 0}^{i=1,\dots,9}$, $\{\gamma^{j}_{\boldsymbol{o}}(t)\}^{j=17,\dots,20}_{t\geq 0}$. \color{black} In this simulation, we set the learning rates for the parameter estimates and the sensors placements as constant. That is, $\gamma^{i}_{\theta}(t)=\gamma^{i}_0$ and $\gamma^{j}_{\boldsymbol{o}}(t) = \gamma_{0}^{j}$, with the specific values of $\gamma^{i}_0$, $\gamma_{0}^{j}$ tuned individually. This is a standard choice when the true parameters are no longer static (e.g., \cite{Soderstrom1983}). The choice of constant learning rates violates one of conditions required for convergence of the parameter estimates and the optimal sensor placements, namely, that $\int_0^{\infty}\gamma_{\theta}^2(t)\mathrm{d}t<\infty$ and $\int_0^{\infty}\gamma_{\boldsymbol{o}}^2(t)\mathrm{d}t<\infty$ (see Appendix \ref{appendixA}). There is thus no longer any guarantee that the algorithm iterates will converge to the stationary points of the two objective functions. 
%parameter estimates generated by our algorithm will converge to the stationary points of the asymptotic log-likelihood, or indeed that the optimal sensor placements will converge to the stationary points of the asymptotic objective function. 
They are, however, expected to oscillate around the optimal points. The advantage of constant learning rates is that the algorithm iterates can now adapt rapidly to changes in the true model parameters and the optimal sensor placements. %
\color{black}

In practice, the two-timescale stochastic gradient algorithm still performs remarkably well in this scenario (Figure \ref{fig2}). The online parameter estimates generated by the algorithm are able to track the changes in the dynamic model parameters in real time (Figure \ref{fig2a}), \color{black} while the sensor placements update in response to changes in the time-varying weighted spatial locations (Figure \ref{fig2g}). 

Let us make some brief remarks regarding the optimal sensor placements. In general, we see that, at any given time instant, the sensors tend to be positioned closer to the current locations of the four weighted points than they would be in a completely uniform configuration. For example, at $t=30000$ (top right hand panel in Figure \ref{fig2g}), all of the sensors, to a greater or lesser extent, have moved towards the south-west of the domain. At the same time, the sensors also maintain a relatively even distribution across the entire centre of the domain. This should not come as a surprise, and does indeed represent the optimal placement of the available measurement sensors. In particular, this configuration represents a trade-off between attempting to minimise the uncertainty of the state estimate over the entire spatial domain (which favours a uniform placement of sensors), while also placing a slightly greater emphasis on the accuracy of the state estimate at the four time-varying weighted locations (which favours a placement of sensors close to these locations).
%It is worth noting that, while this simulation considers the case in which the model parameters and the optimal sensor locations change discontinuously in time, the algorithm is also able to track continuous changes (results omitted).
\color{black}

\vspace{-6mm}
\begin{figure}[!h]
\color{black}
\centering
  \subfloat[Online parameter estimates.]{\label{fig2a}\includegraphics[width=0.48\textwidth]{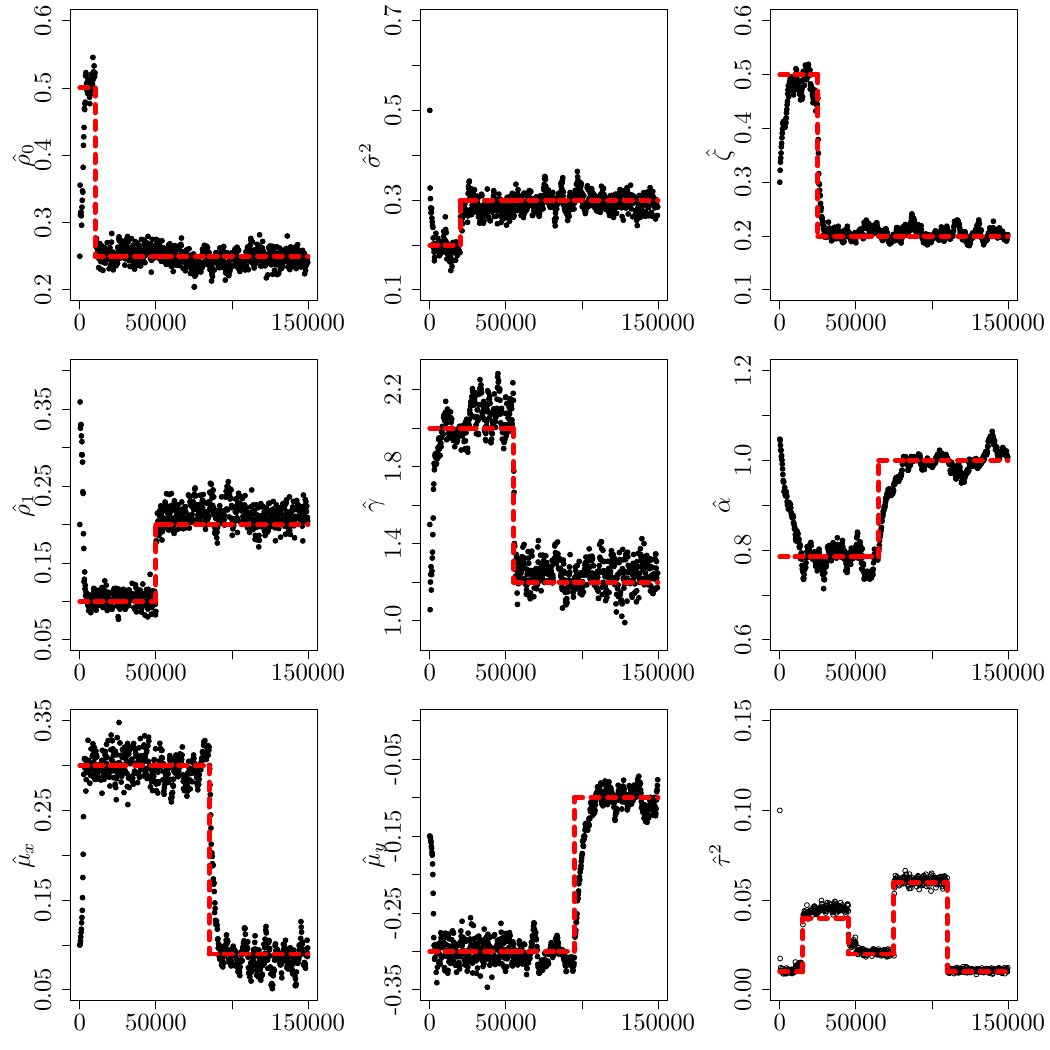}}
  \hspace{1mm}
  \subfloat[Optimal sensor placements]{\label{fig2g}\includegraphics[width=0.5\textwidth]{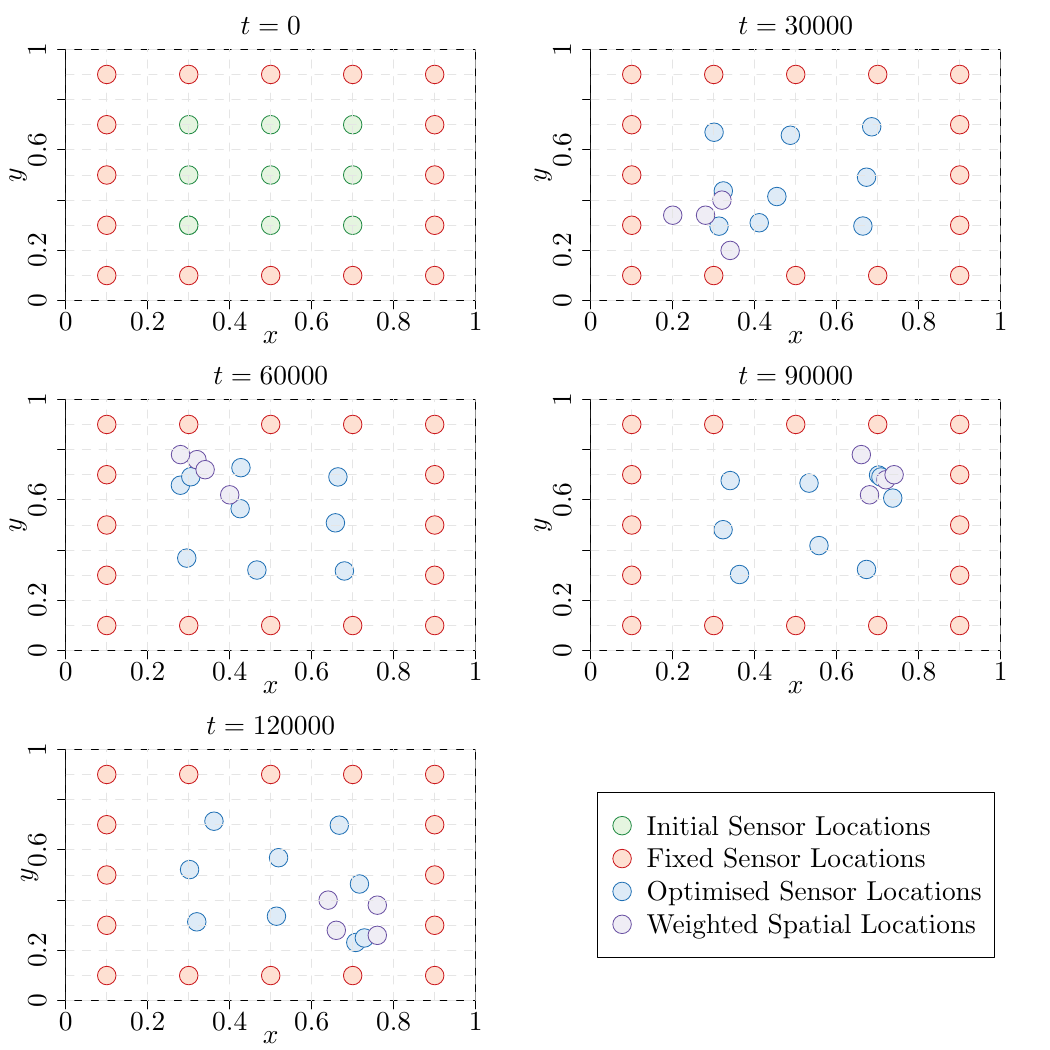}}
\caption{Simulation III. The online parameter estimates \& the optimal sensor placements; and the true parameters (red, dashed). 
The total CPU time required for this simulation is 9345s (0.0623s per iteration).}
\label{fig2}
\color{black}
\end{figure}
\vspace{-3mm}

\subsubsection{Simulation IV} In our fourth numerical experiment, we investigate the ability of our algorithm to estimate multiple unknown bias and variance parameters. We thus relax our previous assumption that the sensors all have zero bias, and the same variance. This scenario is of significant practical interest: in real-data applications, it is often necessary to calibrate the bias and variance of many measurement sensors simultaneously, and in real time. 

\color{black}
In this simulation, we assume that we have $n_y = 11$ sensors in $\Pi = [0,1]^2$, six of which have unknown bias and variance. The true values, and initial estimates, of these parameters, are given respectively by
\begin{alignat}{3}
 {\boldsymbol{\tau}^{*}}^2&=({\tau_1^{*}}^2=0.01,{\tau_2^{*}}^2=0.01,{\tau_3^{*}}^2=0.05,{\tau_4^{*}}^2=0.05,{\tau_5^{*}}^2=0.10,{\tau_6^{*}}^2=0.10) \\
 {\boldsymbol{\tau}}^2_0&=({\tau}^2_{1,0}=0.05,{\tau}^2_{2,0}=0.03,{\tau}^2_{3,0}=0.15,{\tau}^2_{3,0}=0.20,{\tau}^2_{5,0}=0.02,{\tau}^2_{6,0}=0.25)
 \end{alignat}
 and 
 \begin{alignat}{3}
{\boldsymbol{\beta}}^{*} &= (\beta^{*}_{1} = 1.0, \beta^{*}_{2} = 1.0,\beta^{*}_{3} = 1.0,\beta^{*}_{4} = 2.0,\beta^{*}_{5} = 2.0,\beta^{*}_{6} = 2.0)  \\
{\boldsymbol{\beta}}_0 &=({\beta}_{1,0}=0.1,{\beta}_{2,0}=3.0,{\beta}_{3,0}=1.5,{\beta}_{4,0}=2.5,{\beta}_{5,0}=0.0,{\beta}_{6,0}=0.5).
\end{alignat}
%We remark that, while the sensors could be categorised naturally into $p_1 = 3$ variance classes and $p_1 = 2$ bias classes (i.e., into pairs of sensors which share the same variance, and triples of sensors which share the same bias), each of which we could identify we do not assume that this is the case. Thus, rather than estimating the variance and bias of each sensor class, the bias and variance of each sensor is estimated independently. While this increases the computational cost of the algorithm, this approach allows for greater flexibility, and is rather more realistic in practical cases. 
We estimate the bias and variance of each of these sensors independently. In terms of the parameters in the signal equation, we now assume that the values of $\theta_{3:6} = (\zeta,\rho_1,\gamma,\alpha)$ are known, and fixed equal to their true values, while the values of $\theta_{1,2,7:8} = (\rho_0,\sigma^2,\mu_x,\mu_y)$ are unknown, and to be estimated. The true values and initial values of these parameters are shown in Figure \ref{fig3a}.

Regarding the sensor placement, we assume that the locations of the six sensors whose biases and variances are unknown are to be optimised, while the locations of the remaining five sensors are fixed. The objective is to minimise the uncertainty in the state estimate over the entire spatial domain, as in Simulation II. The locations of the fixed sensors are distributed non-uniformly close to the boundary of the domain, while the initial locations of the sensors whose locations are to be optimised are distributed non-uniformly close to the centre of the domain. Finally, the step-sizes are of the same form as those in the Simulation Ib. 

The performance of the two-timescale algorithm is shown in Figure \ref{fig3}, in which we have plotted the sequence of online parameter estimates for the unknown parameters. 
%, and Figure [], in which we have plotted the final sensor locations (Figure []), as well as the time evolution of the MSE of the optimal state estimate. 
As previously, the parameter estimates are all seen to converge towards a small neighbourhood of their true values. Thus, the algorithm correctly identifies the different biases and variances corresponding to each of the measurement sensors. Meanwhile, the final locations of the movable measurement sensors are distributed more evenly throughout the domain (plot omitted), leading to a 27\% reduction in the error in the optimal state estimate (0.026 to 0.019). 
\color{black}

\vspace{-4mm}
\begin{figure}[!h]
\color{black}
\captionsetup[subfloat]{captionskip=0pt}
\begin{minipage}{.49\linewidth}
\hspace{5mm}
\centering
\subfloat[The signal parameters.]{\label{fig3a}\includegraphics[width=.82\linewidth]{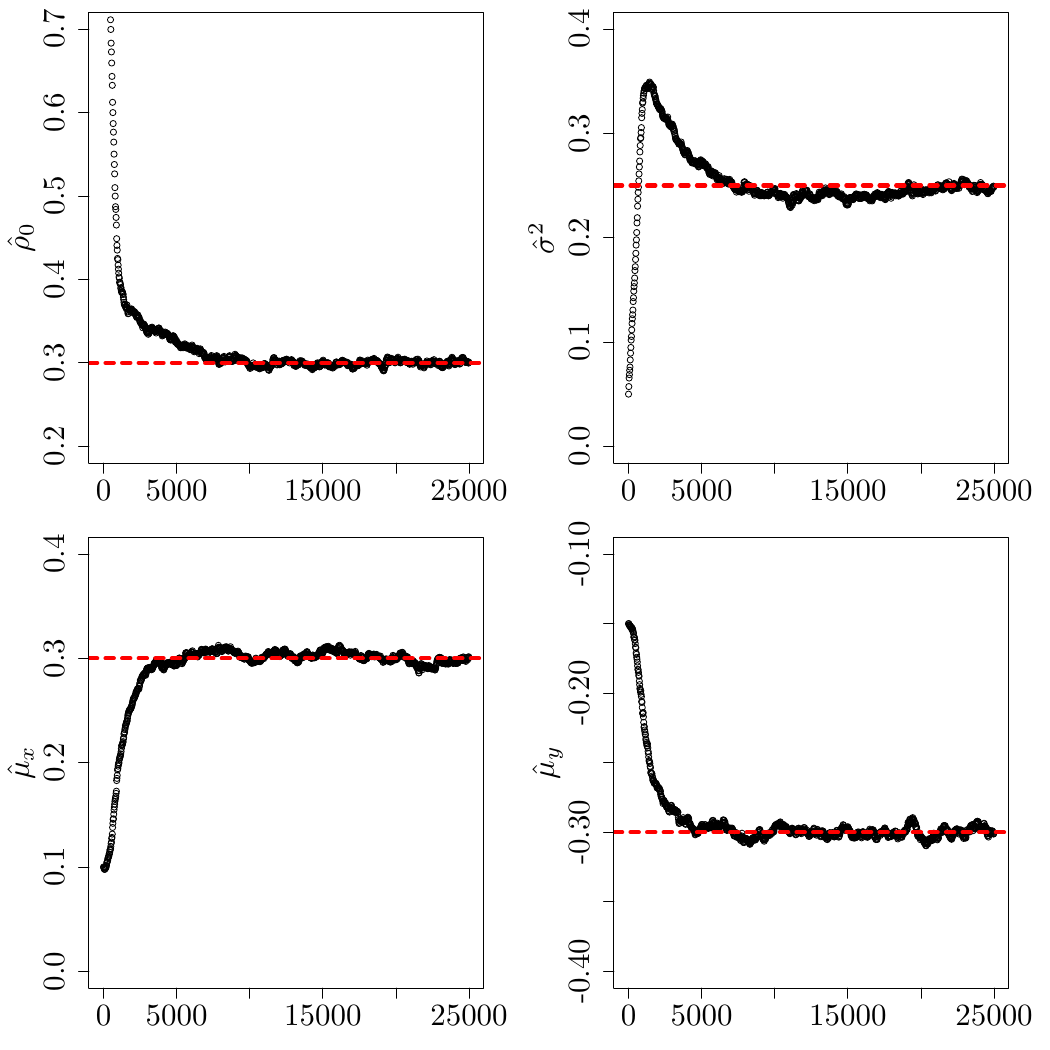}}
 \end{minipage}
 \hspace{-4mm}
 \begin{minipage}{.49\linewidth}
 \centering
 \hspace{-5mm}
 \captionsetup[subfloat]{captionskip=-4pt}
 \subfloat[The variance parameters.]{\label{fig3b}\includegraphics[width=.65\linewidth]{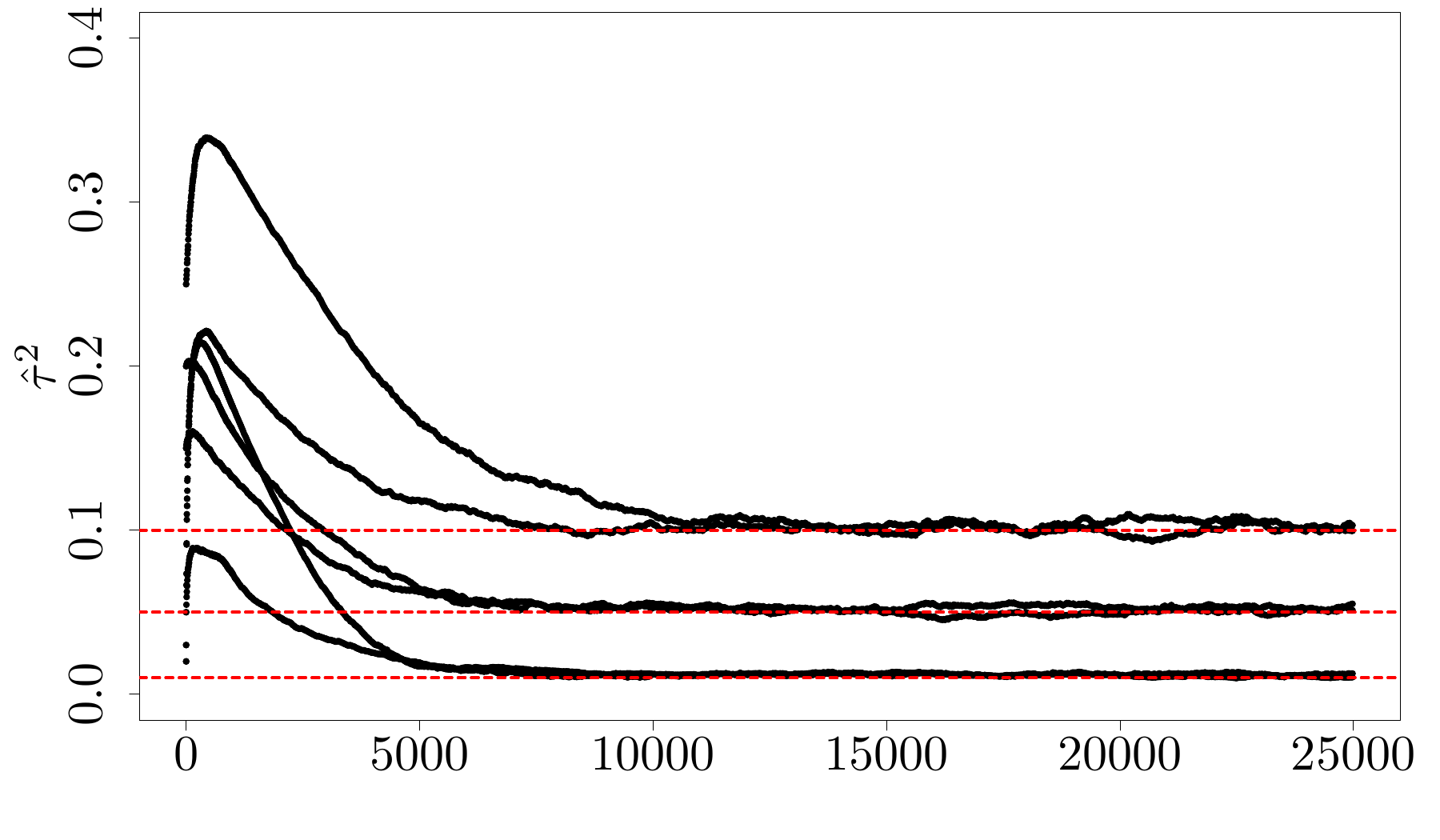}}
 \hspace{-20mm} \\
  \subfloat[The bias parameters.]{\label{fig3c}\includegraphics[width=.65\linewidth]{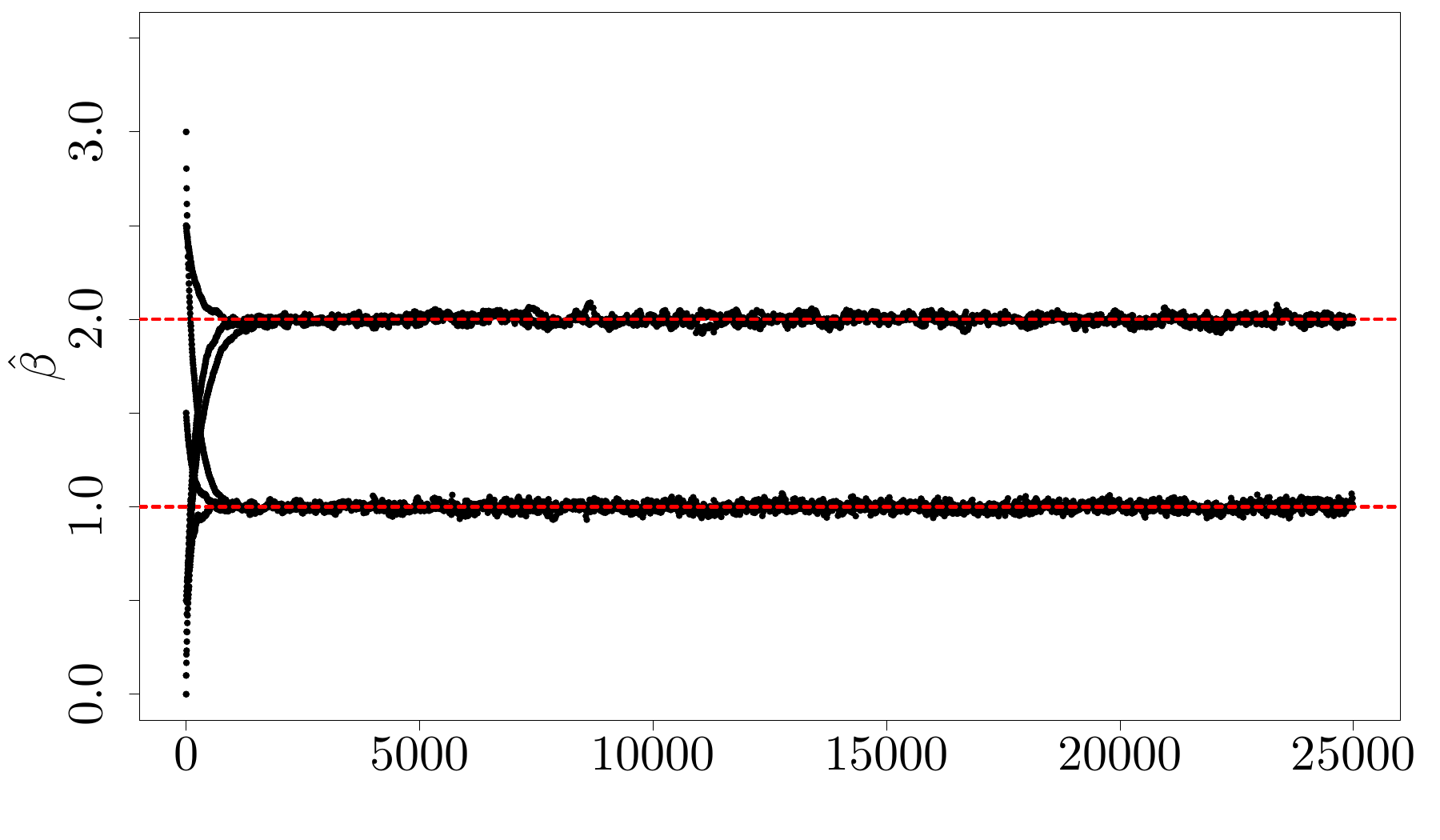}}
\end{minipage}
\caption{Simulation IV. The online parameter estimates (black); and the true parameters (red, dashed). The total CPU time required for this simulation is 888s (0.0355s per iteration). }
\label{fig3}
\vspace{-14mm}
\color{black}
\end{figure}

\subsubsection{Simulation V}  In our final numerical simulation, we investigate the performance of the two-timescale stochastic gradient descent algorithm in the presence of a spatially weighted disturbance in the signal noise. We will assume, in this simulation, that $\theta_{1:6} = (\rho_0,\sigma^2,\zeta,\rho_1,\gamma,\alpha)$ are known, while $\theta_{7:9} = (\mu_x,\mu_y,\tau^2)$ are to be estimated. The true values and initial estimates of these parameters are given respectively by 
\begin{align}
\theta^{*}&= (\mu_x = 0.10,\mu_y =-0.10, \tau^2 =0.01 ), \\
\theta_{0}& = (\mu_x = 0.39,\mu_y =-0.41,\tau^2 =0.50 ).
\end{align}

We now assume that we have $n_y = 10$ sensors $\Pi = [0,1]^2$, each with zero bias and equal variance. The locations of nine of these sensors are fixed, while the location of the final sensor is to be optimised. The locations of the fixed sensors, and the initial location of the sensor whose location is to be optimised, are shown in Figure \ref{fig_4b}. 
\iffalse
given respectively by 
\begin{align}
{\boldsymbol{o}}_{\text{fixed}} &= \frac{1}{12}\left\{\begin{pmatrix} 2 \\ 2\end{pmatrix},\begin{pmatrix} 2\\ 6\end{pmatrix},\begin{pmatrix} 2\\ 10 \end{pmatrix}, \begin{pmatrix} 6 \\ 2 \end{pmatrix}, \begin{pmatrix} 6 \\ 6 \end{pmatrix}, \begin{pmatrix} 6 \\ 10 \end{pmatrix}, \begin{pmatrix} 10 \\ 2 \end{pmatrix}, \begin{pmatrix} 10 \\ 6 \end{pmatrix}, \begin{pmatrix} 10 \\ 10 \end{pmatrix}\right\}, \\
{\boldsymbol{o}}_0 &= \frac{1}{12}\left\{\begin{pmatrix} 4 \\ 4\end{pmatrix}\right\}.
\end{align}
\fi
As in the second numerical experiment, we will suppose that the objective is to obtain the optimal state estimate over the entire spatial domain (i.e., not only at a set of target locations). We also now suppose that there is a localised disturbance in the signal noise around the point $(\tfrac{5}{12},\tfrac{5}{12})$. Thus, in the signal equation, we now specify the spatial weighting function
\begin{equation}
b(\boldsymbol{x}) = b(x,y) =  \mathrm{sech}\bigg[\bigg(\frac{(x-\tfrac{5}{12})^2}{0.2^2}+\frac{(y-\tfrac{5}{12})^2}{0.2^2}\bigg)^{\frac{1}{2}}\bigg].
\end{equation}
%Finally, we use 
%It remains to specify the learning rates $\smash{\{\gamma^{i}_{\theta}(t)\}_{t\geq 0}^{i=7,\dots,9}}$ and $\smash{\{\gamma^{j}_{\boldsymbol{o}}(t)\}^{j=10}_{t\geq 0}}$. To avoid further repetition, we remark only that the
 %learning rates of the same form as those in Simulation Ia. 
 %Once more, the specific values of $\gamma^{i}_{\theta,0}$, $\varepsilon_{\theta}^i$, $\gamma^{j}_{\boldsymbol{o},0}$, and $\varepsilon_{\boldsymbol{o}}^{j}$ are tuned individually. 

The performance of the two-timescale algorithm is illustrated in Figures \ref{fig4} and \ref{fig5}, in which we have plotted trial averaged sequences of the optimal sensor placements, and the online parameter estimates, respectively. As previously, the online parameter estimates all converge to a neighbourhood around their true values. Meanwhile, the optimal sensor placement is seen to converge to a location close to, but not directly at, the centre of the local disturbance. The slight offset to the south-west of the centre of this disturbance is explained by the presence of the fixed sensor at $(0.5,0.5)$, which is just to the north-east of the centre of the disturbance. These numerical results corroborate those also obtained in, e.g., \cite{Burns2015,Zhang2018}.

\begin{figure}[!h]
\vspace{-5mm}
\captionsetup[subfloat]{captionskip=-2pt}
\centering
\subfloat[The spatial weighting $b(\boldsymbol{x})$.]{\label{fig_4a}\includegraphics[width=.33\linewidth]{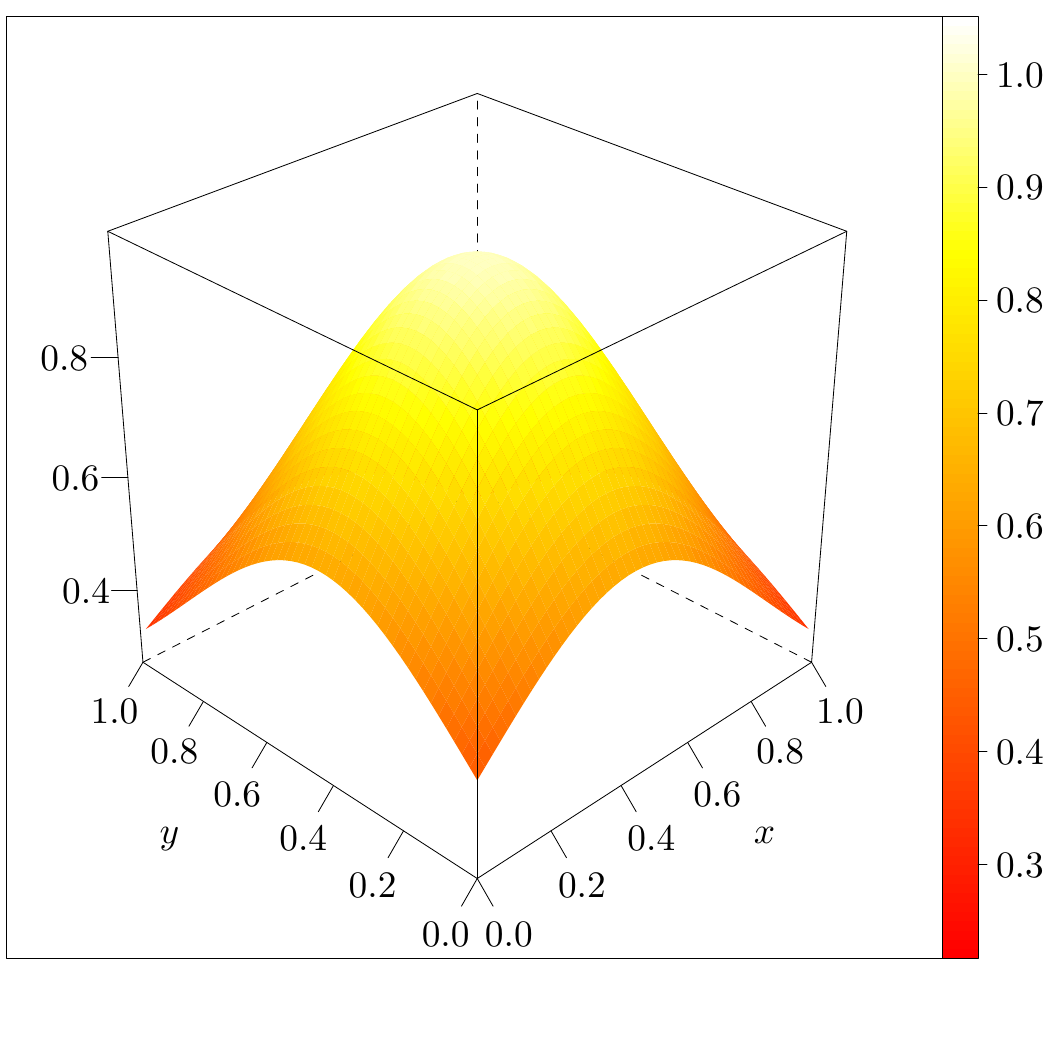}}
  \hspace{8mm}
 \subfloat[Sensor placements on $\Pi = {[0,1]}^2$.]{\label{fig_4b}\includegraphics[width=.33\linewidth]{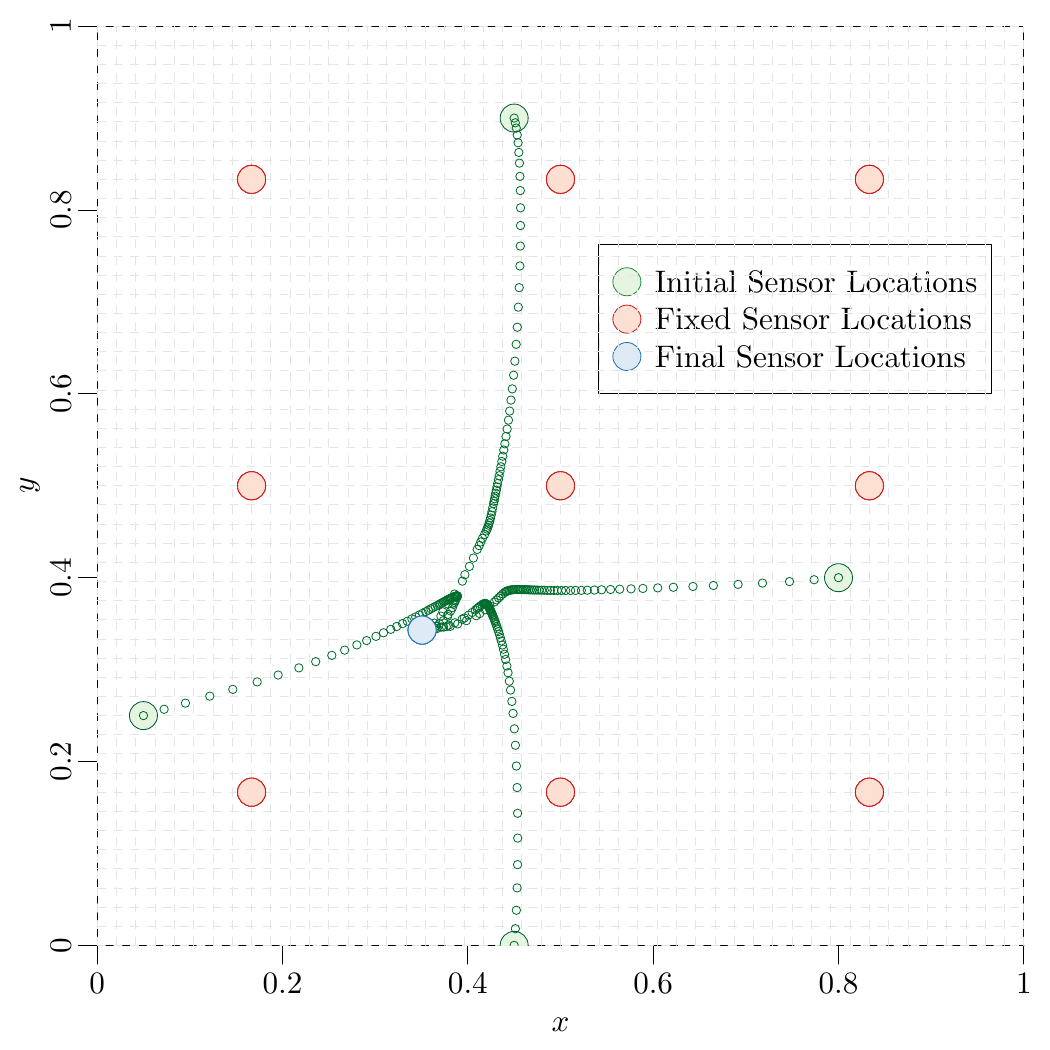}}
\caption{Simulation V. The sequence of optimal sensor placements for four different initial conditions. The average CPU time required for this simulation is 41s (0.016s per iteration). }
\label{fig4}
\end{figure}

\vspace{-4mm}
\begin{figure}[!h]
\color{black}
\vspace{-4mm}
\captionsetup[subfloat]{captionskip=0pt}
\centering
\hspace{-2mm}
\subfloat[Online parameter estimates.]{\label{fig_5a}\includegraphics[width=.44\linewidth]{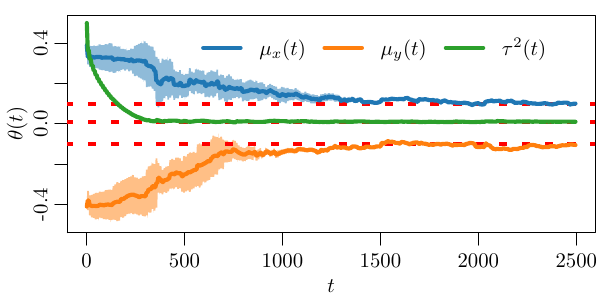}}
  \hspace{1mm}
 \subfloat[Optimal sensor placements.]{\label{fig_5b}\includegraphics[width=.44\linewidth]{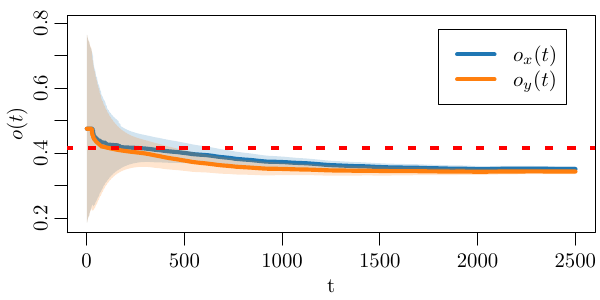}}
\caption{Simulation V. The online parameter estimates \& the optimal sensor placements (various colours, solid); and the true parameters \& centre of the signal noise disturbance (red, dashed). The plots are averaged over $400$ trials which use different initialisations of the movable sensor.}
\label{fig5}
\end{figure}
\vspace{-2mm}

\section{Conclusions}
\label{sec:conclusions}

\iffalse
In this article, we have presented a two-timescale stochastic gradient descent algorithm for the problem of joint online parameter estimation and optimal sensor placement in a partially observed, infinite-dimensional linear diffusion process. The proposed algorithm simultaneously and recursively seeks to maximise the log-likelihood of the observations with respect to the model parameters, and minimise the uncertainty in the estimate of the latent signal with respect to the sensor locations. We have demonstrated, via a series of detailed numerical simulations, how this methodology can be successfully applied to a partially observed stochastic advection-diffusion equation, which depends in a highly non-linear fashion on a set of nine or more unknown model parameters. 
\fi

In this article, we have considered the problem of joint online parameter estimation and optimal sensor placement for a partially observed, infinite-dimensional linear diffusion process. We have presented a solution to this problem in the form of a two-timescale stochastic gradient descent algorithm,
%, which seeks to maximise the log-likelihood of the observations with respect to the model parameters, and to minimise the uncertainty in the estimate of the latent signal with respect to the sensor locations. 
and shown in detail how this algorithm can be successfully applied to a partially observed stochastic advection-diffusion equation, which depends in a highly non-linear fashion on a set of nine or more unknown model parameters. Our numerical results have illustrated the effectiveness of the proposed approach in a number of scenarios of practical interest. Moreover, they have highlighted the advantages of tackling the problems of online parameter estimation and optimal sensor placement together. 

There are several important extensions to the work presented in this paper. From a theoretical perspective, the main open problem is to obtain rigorous convergence results for the parameter estimates and optimal sensor placements generated by the finite-dimensional approximation of Algorithm \eqref{RML_ROSP_1} - \eqref{RML_ROSP_2} to the stationary points of the true, infinite-dimensional asymptotic log-likelihood and asymptotic sensor placement objective function. This essentially represents an extension of \cite[Proposition 3.1]{Sharrock2020a} to include an infinite-dimensional latent signal process, and the effect of using a finite-dimensional approximation. There are, in fact, several existing results along these lines for the optimal sensor placement (e.g., \cite{Burns2015,Zhang2018}). The main obstacle, therefore, is to obtain corresponding results for the MLE.

\iffalse
From a theoretical perspective, there are two major open problems. The first is to obtain rigorous convergence results for the parameter estimates and optimal sensor placements generated by Algorithm \eqref{RML_ROSP_1} - \eqref{RML_ROSP_2} to the stationary points of the infinite-dimensional asymptotic log-likelihood and sensor placement objective function, respectively. This essentially requires a careful extension of \cite[Proposition 3.1]{Sharrock2020a} to include the case in which the latent signal process is infinite-dimensional. 

The second is to obtain conditions for the convergence of the parameter estimates and optimal sensor placements generated by a finite-dimensional approximation of Algorithm \eqref{RML_ROSP_1} - \eqref{RML_ROSP_2}, as the order of this approximation increases. This is essentially equivalent to obtaining conditions under which finite-dimensional approximations of the MLE and optimal sensor placement %, say $\smash{\theta^{*}_n = \argmin_{\theta\in\Theta}\tilde{\mathcal{L}}_n(\theta,\boldsymbol{o})}$ and $\smash{\hat{\boldsymbol{o}}_n = \argmin_{\theta\in\Theta}\tilde{\mathcal{J}}_n(\theta,\boldsymbol{o})}$, 
converge to their true, infinite-dimensional counterparts. There are, in fact, several existing results along these lines for the optimal sensor placement (e.g., \cite{Burns2015,Zhang2018}). The main obstacle, then, is to obtain corresponding results for the MLE.
\fi

From a computational perspective, a natural extension of our numerical results is to consider the case in which the advection-diffusion operator is no longer spatially or temporally invariant, with the drift, diffusion, and damping parameters allowed to vary in space (see, e.g., \cite{Liu2019}) or in time. It is also of interest to consider alternative, more complex spatial domains and boundary conditions, which are typical of environmental monitoring applications in urban settings. We are also interested in considering other sensor configurations, perhaps allowing explicitly for the possibility of mobile sensors whose motion is governed by some controlled ODEs, or for differing levels of communication between sensors (e.g., \cite{Burns2015,Demetriou2009}). Finally, we would like to extend the results in this paper to partially observed diffusion processes governed by non-linear dissipative SPDEs, such as the stochastic Navier-Stokes or Kuramoto-Sivashinski equations \cite{Breit2018,Jardak2010}. In such models, of course, it will no longer be possible to compute the filter or tangent filter analytically, and it will be necessary to replace these quantities by suitable approximations (e.g., \cite{Beskos2020,Kantas2015}). 
\newpage

\iffalse
\section*{Acknowledgments}
\label{sec:acknowledgments}
We are very grateful to D. Crisan and A. Forbes for many helpful discussions and suggestions. The first author was funded by the EPRSC CDT in the Mathematics of Planet Earth (grant number EP/L016613/1) and the National Physical Laboratory. The second author was partially funded under a J.P. Morgan A.I. Research Award (2019).
\fi
%\vfill

%%%%%%%%%%%%
%%% APPENDIX %%%
%%%%%%%%%%%%

\appendix{
\section{Convergence of the Joint Online Parameter Estimation and Optimal Sensor Placement Algorithm} \label{appendixA} In this Appendix, we provide an almost sure convergence result for the joint parameter estimation and optimal sensor placement algorithm proposed in the main article. This corresponds to \cite[Proposition 3.1]{Sharrock2020a}, adapted appropriately to the linear Gaussian case. This result guarantees the almost sure convergence of the online parameter estimates and optimal sensor placements generated by the finite-dimensional approximation of Algorithm \eqref{RML_ROSP_1} - \eqref{RML_ROSP_2} to the stationary points of the finite-dimensional approximations of the asymptotic log-likelihood and the asymptotic sensor placement objective function, respectively. We also state and verify several sufficient conditions in the linear Gaussian case. 
\iffalse
, viz
\begin{align}
\lim_{t\rightarrow\infty}\nabla_{\theta}\tilde{\mathcal{L}}(\theta(t),\boldsymbol{o}(t)) = \lim_{t\rightarrow\infty}\nabla_{\boldsymbol{o}}\tilde{\mathcal{J}}(\theta(t),\boldsymbol{o}(t)) = 0.
\end{align}
\fi
%These conditions are adapted from \cite[Proposition 3.1]{Sharrock2020a}. 
%We also verify several of these conditions.

\subsection{Additional Notation}
We will require the following additional notation. We will write $||\cdot||$ to denote both the standard Euclidean norm and the standard matrix norm; and $||\cdot||_{\mathrm{HS}}$ to denote the Hilbert-Schmidt norm. We will say that a function $H:\Theta\times\Pi^{n_y}\times\mathbb{R}^{d}\rightarrow \mathbb{R}$ satisfies the polynomial growth property (PGP) if there exist $q,K>0$ such that, for all $\theta\in\Theta$, $\boldsymbol{o}\in\Pi^{n_y}$,
\begin{equation}
|H(\theta,\boldsymbol{o},x)|\leq  K(1+||x||^q).
\end{equation}
We will then write $\smash{\mathbb{H}^{d}}$ to denote the space of functions $H:\Theta\times\Pi^{n_y}\times\mathbb{R}^{d}\rightarrow \mathbb{R}$ such that $H(\cdot,\boldsymbol{o},x)\in C^2(\Theta)$, $H(\theta,\cdot,x)\in C^2(\Pi^{n_y})$, $H(\theta,\boldsymbol{o},\cdot)\in C(\mathbb{R}^{d})$, and $\nabla^{i}_{\theta}H(\cdot,\cdot,x)$, $\nabla^{i}_{\boldsymbol{o}}H(\cdot,\cdot,x)$, $i=1,2$, are H\"older continuous with exponent $\delta$. We will also write $\mathbb{H}^{d}_{c}$ to denote the subspace consisting of $H\in\mathbb{H}^{d}$ that are centred, i.e., $
\int_{\mathbb{R}^{d}} H(\theta,\boldsymbol{o},x)\mu_{\theta,\boldsymbol{o}}(\mathrm{d}x) = 0$.
Finally, we will write $\bar{\mathbb{H}}^{d}$ to denote the subspace consisting of $H\in\mathbb{H}^{d}$ such that $H$ and all of its first and second derivatives with respect to $\theta$ and $\boldsymbol{o}$ satisfy the PGP.

\color{black}
\subsection{The Finite-Dimensional Kalman-Bucy Filter}
Let $\smash{\hat{\alpha}_n(\theta,\boldsymbol{o},t) = \mathbb{E}_{\theta,\boldsymbol{o}}[\alpha_n(t)|\mathcal{F}_t^{Y_n}]}$ and $\smash{\Sigma_n^{\alpha}(\theta,\boldsymbol{o},t) = \mathrm{Cov}_{\theta,\boldsymbol{o}}[\alpha_n(t)|\mathcal{F}_t^{Y_n}]}$ denote the conditional mean and conditional covariance of the vector of Fourier coefficients $\smash{{\alpha}_{n}(t)=\{\alpha_{\boldsymbol{k}}(t)\}_{\boldsymbol{k}\in\Lambda_n}}$ given the observation sigma algebra $\smash{\mathcal{F}_t^{Y_n} = \sigma\{y_n(s):s\in[0,t]\}}$. 
%Since the equations governing $\alpha_n(t)$ and $y_n(t)$, namely \eqref{finite_dim_signal} and \eqref{finite_dim_obs}, are finite dimensional, $\smash{\hat{\alpha}_n(\theta,\boldsymbol{o},t)}$ and $\smash{\Sigma_n^{\alpha}(\theta,\boldsymbol{o},t)}$ evolve according to the classical Kalman-Bucy filter (e.g., \cite{Kalman1961}). 
Since \eqref{finite_dim_signal} - \eqref{finite_dim_obs} are finite dimensional, these quantities evolve according to the finite dimensional Kalman-Bucy filter (e.g., \cite{Kalman1961}). In particular, we have that
\begin{align}
 \mathrm{d}\hat{\alpha}_n(\theta,\boldsymbol{o},t) &= {A}_n(\theta)\hat{{\alpha}}_n(\theta,\boldsymbol{o},t)\mathrm{d}t \label{fin_dim_mean}  \\
&\hspace{5mm}+\Sigma_n^{\alpha}(\theta,\boldsymbol{o},t) {C}^{*}_n(\theta,\boldsymbol{o})\mathcal{R}^{-1}(\boldsymbol{o})\big(\mathrm{d}y_n(t) - {C}_n(\theta,\boldsymbol{o})\hat{{\alpha}}_n(\theta,\boldsymbol{o},t)\mathrm{d}t\big), \nonumber  \\[2mm] 
\dot{\Sigma}^{\alpha}_{n}(\theta,\boldsymbol{o},t)&={A}_n(\theta) \Sigma^{\alpha}_{n}(\theta,\boldsymbol{o},t) + \Sigma^{\alpha}_n(\theta,\boldsymbol{o},t){A}^{*}_n(\theta) \label{fin_dim_ricatti} \\
&\hspace{5mm}+{B}_n{Q}_n(\theta){B}^{*}_n-\Sigma^{\alpha}_n(\theta,\boldsymbol{o},t) {C}^{*}_n(\theta,\boldsymbol{o})\mathcal{R}^{-1}(\boldsymbol{o}){C}_n(\theta,\boldsymbol{o})\Sigma^{\alpha}_n(\theta,\boldsymbol{o},t). \hspace{-2mm}  \nonumber
\end{align} 

\subsection{The Finite-Dimensional Objective Functions}
Using these quantities, we can obtain approximations of the asymptotic log-likelihood function and the asymptotic sensor placement objective function. In particular, we define 
\begin{align}
\tilde{\mathcal{L}}_n(\theta,\boldsymbol{o}) &= \lim_{t\rightarrow\infty}\frac{1}{t} \left[ \int_0^t \langle R^{-1}(\boldsymbol{o}) C_n(\theta,\boldsymbol{o}) \hat{\alpha}_n(\theta,\boldsymbol{o},s),\mathrm{d}y_n(s)\rangle \right. \\
&\left.\hspace{20mm}- \frac{1}{2}\int_0^t ||\mathcal{R}^{-1}(\boldsymbol{o}) C_n(\theta,\boldsymbol{o}) \hat{\alpha}_n(\theta,\boldsymbol{o},s)||^2\mathrm{d}s\right] \nonumber \\
\tilde{\mathcal{J}}_n(\theta,\boldsymbol{o}) &= \lim_{t\rightarrow\infty} \frac{1}{t}\int_0^t \mathrm{Tr}\left[M_n(s) \Sigma_n^{\alpha}(\theta,\boldsymbol{o},s)\right]\mathrm{d}s
\end{align}
where $M_n(s) = \Pi_n \mathcal{M}(s) \Pi_n$ denotes the finite-dimensional approximation of the bounded linear operator $\mathcal{M}(s)$ used to weight significant parts of the state estimate.

\subsection{The Finite-Dimensional Algorithm}
We can also now define a finite-dimensional approximation of the joint online parameter estimation and optimal sensor placement algorithm, namely, 
\begin{subequations}
\begin{alignat}{2}
\hspace{-1mm} \mathrm{d}{\theta}_n(t) \hspace{-.5mm}&=\hspace{-.5mm} \left\{ \hspace{-.5mm}\begin{array}{l} \hspace{-1mm}-\gamma_{\theta}(t) \big[{C}_n(\theta,\boldsymbol{o})\hat{\alpha}_n^{\theta}(\theta,\boldsymbol{o},t)\big]^T{R}^{-1}(\boldsymbol{o})\big[{C}_n(\theta,\boldsymbol{o})\hat{\alpha}_n(\theta,\boldsymbol{o},t)\mathrm{d}t-\mathrm{d}y_n(t)\big] \big|_{\substack{\theta=\theta_n(t) \\ \boldsymbol{o}=\boldsymbol{o}_n(t)}} \\ \hspace{-1mm} 0 \end{array}\right. &&\begin{array}{ll}  \hspace{-4mm}, & \hspace{-2mm} {\theta}_n(t)\in\Theta, \\[3mm] \hspace{-4mm}, & \hspace{-2mm} {\theta}_n(t)\not\in\Theta,  \end{array} \hspace{-12mm} \label{RML_ROSP_1)} \\
\hspace{-1mm} \mathrm{d}{\boldsymbol{o}}_n(t)\hspace{-.5mm}&=\hspace{-.5mm} \left\{\hspace{-.5mm} \begin{array}{l}  \hspace{-1mm}-\gamma_{\boldsymbol{o}}(t)\mathrm{Tr}^{\boldsymbol{o}}\big[{M}_n(t)\Sigma_n^{\alpha}(\theta,\boldsymbol{o},t)\big]^T\mathrm{d}t \big|_{\substack{\theta=\theta_n(t) \\ \boldsymbol{o}=\boldsymbol{o}_n(t)}}  \\ \hspace{-1mm} 0 \end{array}\right. &&\begin{array}{ll} \hspace{-4mm}, & \hspace{-3mm}  {\boldsymbol{o}}_n(t)\in\Omega^{n_{y}}, \\[3mm]  \hspace{-4mm}, & \hspace{-3mm}  {\boldsymbol{o}}_n(t)\not\in\Omega^{n_{y}}.  \end{array} \hspace{-12mm} \label{RML_ROSP_2_}
 % \customlabel{algorithm3.1}{3.1}
 \vspace{-8mm}
\end{alignat}
\end{subequations}
\color{black}

\subsection{Convergence Result}
Finally, we can now provide a statement of the relevant convergence result from \cite{Sharrock2020a}. In particular, this corresponds to \cite[Proposition 3.1]{Sharrock2020a}. 
\begin{proposition} \label{prop1}
Assume that Conditions \ref{assumption1}, \ref{assumption2} - \ref{assumption5}, \ref{assumption10}, and \ref{assumption11} hold (see Section \ref{app:assumptions}). Then, with probability one, 
\begin{align}
\lim_{t\rightarrow\infty}\nabla_{\theta}\tilde{\mathcal{L}}_n(\theta_n(t),\boldsymbol{o}_n(t)) = \lim_{t\rightarrow\infty}\nabla_{\boldsymbol{o}}\tilde{\mathcal{J}}_n(\theta_n(t),\boldsymbol{o}_n(t)) = 0,
\end{align}
or
\begin{align}
\lim_{t\rightarrow\infty}(\theta_n(t),\boldsymbol{o}_n(t)) \in \{(\theta,\boldsymbol{o}):\theta\in \partial\Theta\cup \boldsymbol{o}\in \partial\Pi^{n_y}\}.
\end{align}
\end{proposition}

\subsection{Conditions}  \label{app:assumptions}
The conditions for Proposition \ref{prop1} are stated below. For reference, these correspond to Assumptions 2.2.1, 2.2.2a, 2.2.2b, 2.2.2c, 2.2.2d.i'', 2.2.2d.ii'', 2.2.2e, and 2.1.5 - 2.1.6 in \cite{Sharrock2020a}, adapted appropriately to the linear Gaussian setting.  We remark that this set of Assumptions is sufficient for \cite[Proposition 3.1]{Sharrock2020a} by \cite[Propositions E.1, E.2]{Sharrock2020a}.
\begin{manualassumption}{A.1} \label{assumption1}
The step sizes $\{\gamma_{\theta}(t)\}_{t\geq 0}$, $\{\gamma_{\boldsymbol{o}}(t)\}_{t\geq 0}$, satisfy
\begin{subequations} 
\begin{alignat}{2}
&\lim_{t\rightarrow\infty}\gamma_\theta(t)= \lim_{t\rightarrow\infty}\gamma_{\boldsymbol{o}}(t) =\lim_{t\rightarrow\infty} \frac{\gamma_{\theta}(t)}{\gamma_{\boldsymbol{o}}(t)}=0, &&\int_0^{\infty}\gamma_{\theta}(t)\mathrm{d}t =\int_0^{\infty}\gamma_{\boldsymbol{o}}(t)\mathrm{d}t =\infty, \\
&\int_0^{\infty}\gamma^2_{\theta}(t)\mathrm{d}t,\int_0^{\infty}\gamma^2_{\boldsymbol{o}}(t)\mathrm{d}t<\infty,&&\hspace{-2mm}\int_0^{\infty}\left|\dot{\gamma}_{\theta}(t)\right|\mathrm{d}t,\int_0^{\infty}\left|\dot{\gamma}_{\boldsymbol{o}}(t)\right|\mathrm{d}t< \infty. \label{step_size_1d} 
\end{alignat}
\end{subequations}

In addition, there exist $r_{\theta},r_{\boldsymbol{o}}>0$, such that %$\lim_{t\rightarrow\infty} \gamma_i^2(t)t^{\frac{1}{2}+2r_i} = 0.$
\begin{equation}
\lim_{t\rightarrow\infty} \gamma_{\theta}^2(t)t^{\frac{1}{2}+2r_{\theta}} =\lim_{t\rightarrow\infty} \gamma_{\boldsymbol{o}}^2(t)t^{\frac{1}{2}+2r_{\boldsymbol{o}}} = 0. \label{step_size_2}
\end{equation}
\end{manualassumption}

\begin{manualassumption}{A.2a} \label{assumption2}
The process $\{\mathcal{X}_n(\theta,\boldsymbol{o},t)\}_{t\geq 0}$, which consists of the concatenation of the finite-dimensional approximations of the latent signal process $\{\alpha_n(t)\}_{t\geq 0}$, the optimal filter $\{\hat{\alpha}_n(\theta,\boldsymbol{o},t)\}_{t\geq 0}$, $\{\Sigma^{\alpha}_n(\theta,\boldsymbol{o},t))\}_{t\geq 0}$, and the filter derivatives with respect to $\theta$, $\boldsymbol{o}$, and takes values in $\mathbb{R}^{d_n}$, is ergodic for all fixed $\theta\in\Theta$ and $\boldsymbol{o}\in\Pi^{n_y}$, with unique invariant probability measure $\mu_{n,\theta,\boldsymbol{o}}$.
\end{manualassumption}

\begin{manualassumption}{A.2b} \label{assumption3}
For all $q>0$, $\theta\in\Theta$, $\boldsymbol{o}\in\Pi^{n_y}$, there exist constants $K_{q},K^{\theta}_{q},K_{q}^{\boldsymbol{o}}>0$, such that 
\begin{subequations}
\begin{align}
\int_{\mathbb{R}^{d}} (1+||x||^{q})\mu_{n,\theta,\boldsymbol{o}}(\mathrm{d}x) &\leq K_q, \\
\int_{\mathbb{R}^{d}} (1+||x||^{q})|\nu^{(\theta)}_{n,\theta,\boldsymbol{o},i}(\mathrm{d}x)|&\leq K^{\theta}_q~,~\int_{\mathbb{R}^{d}} (1+||x||^{q})|\nu^{(\boldsymbol{o})}_{n,\theta,\boldsymbol{o},i}(\mathrm{d}x)|\leq K^{\boldsymbol{o}}_q
\end{align}
\end{subequations}
where $\smash{|\nu^{(\theta)}_{n,\theta,\boldsymbol{o},i}(\mathrm{d}x)|}$ and $\smash{|\nu^{(\boldsymbol{o})}_{n,\theta,\boldsymbol{o},i}(\mathrm{d}x)|}$ denote the total variations of the finite signed measures
$\smash{\nu^{(\theta)}_{n,\theta,\boldsymbol{o},i}=\partial_{\theta_i}\mu_{n,\theta,\boldsymbol{o}}}$, $i=1,\dots,d_1$, and $\smash{\nu^{(\boldsymbol{o})}_{n,\theta,\boldsymbol{o},i}=\partial_{\boldsymbol{o}_i}\mu_{n,\theta,\boldsymbol{o}}}$, $i=1,\dots,d_2$.
\end{manualassumption}

\begin{manualassumption}{A.2c} \label{assumption8}
The diffusion coefficient in the stochastic differential equation for the process $\{\mathcal{X}_n(\theta,\boldsymbol{o},t)\}_{t\geq 0}$ has the PGP (component-wise).
\end{manualassumption}

\begin{manualassumption}{A.2di} \label{assumption4a}
For all $H\in\mathbb{H}^{d}_c$, the Poisson equation $\smash{\mathcal{A}_{\mathcal{X}_n}v(\theta,\boldsymbol{o},x)}$ \linebreak ${=H(\theta,\boldsymbol{o},x)}$, where $\mathcal{A}_{\mathcal{X}_n}$ denotes the infinitesimal generator of $\mathcal{X}_n$,
has a unique solution $v(\theta,\boldsymbol{o},x)$ that lies in $\mathbb{H}^{d}$, with $v_n(\theta,\boldsymbol{o},\cdot)\in C^2(\mathbb{R}^{d})$. Moreover, if $H\in\bar{\mathbb{H}}^{d}$, then $v_n\in\bar{\mathbb{H}}^{d}$, and its mixed first partial derivatives with respect to $(\theta,x)$ and $(\boldsymbol{o},x)$ have the PGP.
\end{manualassumption}

\begin{manualassumption}{A.2dii} \label{assumption6}
The functions $\varphi_n,\psi_n,\phi_n,\eta_n,\zeta_n,\xi_n$, which map $\Theta\times\Pi^{n_y}\times \mathbb{R}^{d_n}$ to $\mathbb{R}^{n_y}$, $\mathbb{R}$, $\mathbb{R}^{n_y\times n_{\theta}}$, $\mathbb{R}^{n_{\theta}}$, $\mathbb{R}$, and $\mathbb{R}^{n_yn_{\boldsymbol{o}}}$, respectively, and are defined by
\begin{align}
\varphi_n(\theta,\boldsymbol{o},\mathcal{X}_n)&={C}_n(\theta,\boldsymbol{o})\hat{\alpha}_n(\theta,\boldsymbol{o}) \\
\zeta_n(\theta,\boldsymbol{o},\mathcal{X}_n) &= [{C}_n(\theta,\boldsymbol{o})\hat{\alpha}_n(\theta,\boldsymbol{o})]^T [C_n(\boldsymbol{o})\alpha_n-\tfrac{1}{2}C_n(\boldsymbol{o})\hat{\alpha}_n(\theta,\boldsymbol{o})] \\
\eta_n(\theta,\boldsymbol{o},\mathcal{X}_n) &= {C}_n(\theta,\boldsymbol{o})\hat{\alpha}^{\theta}_n(\theta,\boldsymbol{o}) \\
\psi_n(\theta,\boldsymbol{o},\mathcal{X}_n)&= [{C}_n(\theta,\boldsymbol{o})\hat{u}^{\theta}_n(\theta,\boldsymbol{o})]^T [C_n(\boldsymbol{o})\alpha_n-C_n(\boldsymbol{o})\hat{\alpha}_n(\theta,\boldsymbol{o})] \\
\iota_n(\theta,\boldsymbol{o},\mathcal{X}_n)&= \mathrm{Tr}\left[{M}_n\Sigma^{\alpha}_n(\theta,\boldsymbol{o})\right] \\
\phi_n(\theta,\boldsymbol{o},\mathcal{X}_n)&= \left[\mathrm{Tr}^{\boldsymbol{o}}\left[{M}_n\Sigma^{\alpha}_n(\theta,\boldsymbol{o})\right]\right]^T
\end{align}
have the PGP (component-wise).
\end{manualassumption}

\begin{manualassumption}{A.2e}  \label{assumption5}
For all $q>0$, for all $t\geq 0$, $\mathbb{E}[||\mathcal{X}_n(\theta(t),\boldsymbol{o}(t),t)||^q]<\infty$. Furthermore, there exists $K>0$ such that for all $t$ sufficiently large, 
\begin{subequations}
\begin{align}
\mathbb{E}\left[\sup_{s\leq t}||\mathcal{X}_n(\theta,\boldsymbol{o},s)||^q\right]&\leq K\sqrt{t}~,~~~\forall \theta\in\Theta~,~\forall \boldsymbol{o}\in\Pi^{n_y}, \\
\mathbb{E}\left[\sup_{s\leq t}||\mathcal{X}_n(\theta_n(s),\boldsymbol{o}_n(s),s)||^q\right]&\leq K\sqrt{t}.
\end{align}
\end{subequations}
\end{manualassumption}

\begin{manualassumption}{A.3} \label{assumption10}
For all $\theta\in\mathbb{R}^{n_{\theta}}$, the ordinary differential equation
\begin{equation}
\dot{\boldsymbol{o}}_n(t)= -\nabla_{\boldsymbol{o}} \tilde{\mathcal{J}}_n(\theta,\boldsymbol{o}_n(t)))
\end{equation} 
has a discrete, countable set of equilibria $\{\boldsymbol{o}_i^{*}\}_{i\geq 1} = \{\lambda_i(\theta)\}_{i\geq 1}$, where $\lambda_i:\Theta\rightarrow\Pi^{n_y}$, $i\geq 1$, are Lipschitz-continuous maps. 
\end{manualassumption}

\begin{manualassumption}{A.4} \label{assumption11}
For all $i\geq 1$, the ordinary differential equation 
\begin{equation}
\dot{\theta}_n(t)=-\nabla_{\theta}\tilde{\mathcal{L}}_n(\theta_n(t),\lambda_i(\theta_n(t))
\end{equation}
has a discrete, countable set of equilibria $\{\theta_{ij}^{*}\}_{j\geq 1}$.
\iffalse
For all $i\geq 1$, the set $\tilde{\mathcal{L}}(E_i,\lambda(E_i))$ contains no open sets of $\mathbb{R}^{n_{\theta}}$ other than the empty set (i.e., has empty interior), where
\begin{equation}
E_i = \{\theta\in\mathbb{R}^{n_{\theta}}:\nabla_{\theta}\tilde{\mathcal{L}}(\theta,\lambda_i(\theta))=0\}.
\end{equation}
\fi
\end{manualassumption}

In the general case, Assumptions \ref{assumption10} - \ref{assumption11} are somewhat challenging to verify, \color{black} although they are satisfied for a rather broad class of functions on account of standard results from Morse theory (see, e.g., \cite[Chapter 2]{Matsumoto2002}). \color{black} Thus, for the remainder of this Appendix, we will focus our attention on Assumptions \ref{assumption2} - \ref{assumption5}.

\subsection{Sufficient Conditions}
In particular, we now provide sufficient conditions for Assumptions \ref{assumption2} - \ref{assumption5} in the linear Gaussian case. We also verify that these conditions are, indeed, sufficient. 
%In particular, Assumption \ref{assumption2} stipulates that the process $\{\mathcal{X}(\theta,\boldsymbol{o},t)\}_{t\geq 0}$ is ergodic for all fixed $\theta\in\Theta$ and $\boldsymbol{o}\in\Pi^{n_y}$, with unique invariant probability measure $\mu_{\theta,\boldsymbol{o}}$. Assumption \ref{assumption3} relates to the regularity of the invariant measure and its derivatives. Assumption \ref{assumption4a} relates to the existence, uniqueness, and properties of the solutions of the Poisson equation associated with the ergodic diffusion process. Finally, Assumption \ref{assumption5} requires that the moments of this process are bounded.

%We can establish this as follows. We consider a (strongly) elliptic diffusion process (see Section \ref{sec:equation}). Thus, under reasonable assumptions on the generator of this process (see, e.g., \cite{Pardoux2001}, \cite[Appendix C]{Sharrock2020a}), the latent state is ergodic. The ergodicity of the filter then follows directly from ergodicity of the latent state, and the non-degeneracy of the observations (e.g., \cite{Surace2019}). 
\begin{manualassumption}{A.2i} \label{manual_assumption2i}
For all $\theta\in\Theta, \boldsymbol{o}\in\Pi^{n_y}$, the pair $\smash{({A}_n(\theta),{B}_n{Q}^{\frac{1}{2}}_n(\theta))}$ is %exponentially 
stabilisable, and the pair $\smash{({A}_n(\theta),{C}_n(\theta,\boldsymbol{o}))}$ is %exponentially 
detectable.
\end{manualassumption}

This is a standard sufficient condition for stability of the Kalman-Bucy filter (e.g., \cite{Curtain1978a,Kwakernaak1972}).

\begin{manualassumption}{A.2ii} \label{manual_assumption2ii}
For all $\theta\in\Theta$, the matrix $A_n(\theta)$ is stable.
\end{manualassumption}
%We will now prove that Assumptions \ref{manual_assumption2i} - \ref{manual_assumption2ii} are indeed sufficient for Assumptions \ref{assumption2}, \ref{assumption3}, and \ref{assumption5}.

We remark that this assumption holds for the finite-dimensional matrix representation of the advection-diffusion operator implemented in Section \ref{sec:numerics}. In this case, the matrix $A_n(\theta)$ is diagonal (see equation \eqref{matrices}), and its diagonal entries (i.e., its eigenvalues) have strictly negative real parts.

\begin{manualassumption}{A.2iii} \label{manual_assumption2iii}
For all $\theta\in\Theta$, $\boldsymbol{o}\in\Pi^{n_y}$, the pair $(\Phi_n(\theta,\boldsymbol{o}),\Psi_n(\theta,\boldsymbol{o}))$ is controllable, where $\Phi_n(\theta,\boldsymbol{o})$ and $\Psi_n(\theta,\boldsymbol{o})$ are the drift and diffusion coefficients in the SDE for the process $\{\mathcal{X}_n(\theta,\boldsymbol{o},t)\}_{t\geq 0}$. %That is, equivalently, the $d\times d(n+n_y)$ controllability matrix $M_n(\theta,\boldsymbol{o}) = [\Psi_n(\theta,\boldsymbol{o}), \Phi_n(\theta,\boldsymbol{o})\Psi_n(\theta,\boldsymbol{o}),\cdots,\Phi_n^{d-1}(\theta,\boldsymbol{o})\Psi_n(\theta,\boldsymbol{o})]$ has rank $d$. 
\end{manualassumption}

This assumption ensures that the covariance matrix of $\{\mathcal{X}_n(\theta,\boldsymbol{o},t)\}_{t\geq 0}$ is non-degenerate.

On the basis of Assumptions \ref{manual_assumption2i} - \ref{manual_assumption2iii}, we can now prove the following auxiliary results.
\begin{lemma} \label{propA0}
Suppose that Assumption \ref{manual_assumption2i} holds. Then the process $\{\mathcal{X}_n(\theta,\boldsymbol{o},t)\}_{t\geq 0}$, which consists of the concatenation of the set of Fourier coefficients $\alpha_n(t)$, the filter mean $\hat{\alpha}_n(\theta,\boldsymbol{o},t)$, and the filter derivatives, $\hat{\alpha}_{n,\theta}(\theta,\boldsymbol{o},t)$, $\hat{\alpha}_{n,\boldsymbol{o}}(\theta,\boldsymbol{o},t)$, is a multivariate Ornstein-Uhlenbeck (OU) process. 
\end{lemma}

\begin{proof}
Under Assumption \ref{manual_assumption2i},
%\footnote{Alternatively, we can establish ergodicity of the optimal filter as follows. We consider a (strongly) elliptic diffusion process (see Section \ref{sec:equation}). Thus, under standard assumptions on the generator of this process (e.g., \cite{Pardoux2001}, \cite[Appendix C]{Sharrock2020a}), the latent state is ergodic. The ergodicity of the filter then follows directly from ergodicity of the latent state, and the non-degeneracy of the observations (e.g., \cite{Surace2019}).}
we have $\Sigma^{\alpha}_n(\theta,\boldsymbol{o},t)\rightarrow\Sigma^{\alpha}_{n,\infty}(\theta,\boldsymbol{o})$ as $t\rightarrow\infty$, where $\Sigma^{\alpha}_{n,\infty}(\theta,\boldsymbol{o})$ denotes the solution of the algebraic Ricatti equation (e.g., \cite{Kwakernaak1972})
\begin{align}
0&={A}_n(\theta) \Sigma^{\alpha}_{n,\infty}(\theta,\boldsymbol{o}) + \Sigma^{\alpha}_{n,\infty}(\theta,\boldsymbol{o}){A}^{*}_n(\theta) \label{fin_dim_alg_ricatti} \\
&\hspace{5mm}+{B}_n{Q}_n(\theta){B}^{*}_n-\Sigma^{\alpha}_{n,\infty}(\theta,\boldsymbol{o}) {C}_n^{*}(
\theta,\boldsymbol{o}){R}^{-1}(\boldsymbol{o}){C}_n(\theta,\boldsymbol{o})\Sigma^{\alpha}_{n,\infty}(\theta,\boldsymbol{o}). \nonumber 
 \end{align}
 By initialising the Kalman-Bucy filter with $\Sigma^{\alpha}_{n,\infty}(\theta,\boldsymbol{o})$, its representation can be made $n$-dimensional. In particular, it can now be represented solely in terms of the process $\hat{\alpha}_n(\theta,\boldsymbol{o},t)$, which evolves according to 
\begin{align}
\mathrm{d}\hat{\alpha}_n(\theta,\boldsymbol{o},t) &= {A}_n(\theta)\hat{{\alpha}}_n(\theta,\boldsymbol{o},t)\mathrm{d}t  \label{filter_asymptotic} \\
&\hspace{5mm}+\Sigma^{\alpha}_{n,\infty}(\theta,\boldsymbol{o}) {C}^{*}_n(\theta,\boldsymbol{o}){R}^{-1}(\boldsymbol{o})\big(\mathrm{d}y_n(t) - {C}_n(\theta,\boldsymbol{o})\hat{{\alpha}}_n(\theta,\boldsymbol{o},t)\mathrm{d}t\big). \nonumber %, \nonumber \\[1mm]
%\hat{u}(\theta,\boldsymbol{o},0) &= \hat{u}_0(\theta,\boldsymbol{o}).
\end{align}
Using the finite-dimensional observation equation \eqref{finite_dim_obs}, we can rewrite this as
\begin{align}
\mathrm{d}\hat{\alpha}_n(\theta,\boldsymbol{o},t) = {A}_n(\theta)\hat{{\alpha}}_n(\theta,\boldsymbol{o},t)\mathrm{d}t &+ \Sigma^{\alpha}_{n,\infty}(\theta,\boldsymbol{o}) {C}^{*}_n(\theta,\boldsymbol{o}){R}^{-1}(\boldsymbol{o}){C}_n(\theta,\boldsymbol{o})\alpha(t)\mathrm{d}t  \label{A42} \\
&-  \Sigma^{\alpha}_{n,\infty}(\theta,\boldsymbol{o}) {C}^{*}_n(\theta,\boldsymbol{o}){R}^{-1}(\boldsymbol{o}){C}_n(\theta,\boldsymbol{o})\hat{{\alpha}}_n(\theta,\boldsymbol{o},t)\mathrm{d}t \nonumber \\
&+\Sigma^{\alpha}_{n,\infty}(\theta,\boldsymbol{o}) {C}^{*}_n(\theta,\boldsymbol{o}){R}^{-1}(\boldsymbol{o})\mathrm{d}w_{\boldsymbol{o}}(t)  \nonumber %, \nonumber \\[1mm]
%\hat{u}(\theta,\boldsymbol{o},0) &= \hat{u}_0(\theta,\boldsymbol{o}).
\end{align}
Let $\{\mathcal{X}_n(\theta,\boldsymbol{o},t)\}_{t\geq 0}$ denote the process consisting of the concatenation of the set of Fourier coefficients $\alpha_n(t)$, the $n$-dimensional representation of the filter, $\hat{\alpha}_n(\theta,\boldsymbol{o},t)$, and the vectorisation of the filter derivatives, $\hat{\alpha}_{n,\theta}(\theta,\boldsymbol{o},t)$, $\hat{\alpha}_{n,\boldsymbol{o}}(\theta,\boldsymbol{o},t)$. This process takes values in $\mathbb{R}^{d}$, where $d = n^2+n^2+n^2n_{\theta}+n^2n_y$. Then it follows from equation \eqref{finite_dim_signal}, equation \eqref{A42}, and the derivatives of equation \eqref{A42}, that $\{\mathcal{X}_n(\theta,\boldsymbol{o},t)\}_{t\geq 0}$ evolves according to
\begin{equation}
\mathrm{d} \mathcal{X}_n(\theta,\boldsymbol{o},t)  = \Phi_n(\theta,\boldsymbol{o}) \mathcal{X}_n(\theta,\boldsymbol{o},t)\mathrm{d}t + \Psi_n(\theta,\boldsymbol{o})\mathrm{d}b_n(t) \label{OU_eq}
\end{equation}
where $\Phi_n(\theta,\boldsymbol{o}) \in\mathbb{R}^{d\times d}$, $\Psi_n(\theta,\boldsymbol{o},t)\in\mathbb{R}^{d\times (n+n_y)}$, and $b_n(t) = \left(v_{n,\theta}(t), w_{\boldsymbol{o}}(t)\right)^T$ is the $\mathbb{R}^{n+n_y}$-valued process consisting of the concatenation of the finite-dimensional approximation of the signal noise and the measurement noise. This process has incremental covariance matrix ${T}_n(\theta,\boldsymbol{o})\in\mathbb{R}^{(n+n_y)\times (n+n_y)}$, a block diagonal matrix with entries $\smash{\left[{T}_n(\theta,\boldsymbol{o})\right]_{1:n,1:n} = {Q}_n(\theta)}$ and $\smash{\left[{T}_n(\theta,\boldsymbol{o})\right]_{(n+1):(n+n_y),(n+1):(n+n_y)}= {R}(\boldsymbol{o})}$. For brevity, we have omitted the explicit forms of the matrices $\Phi_n(\theta,\boldsymbol{o})$ and $\Psi_n(\theta,\boldsymbol{o})$.
\end{proof}

\begin{lemma} \label{propA0a}
Suppose that Assumptions \ref{manual_assumption2i} and \ref{manual_assumption2ii} hold. Then $\Phi_n(\theta,\boldsymbol{o})$ is stable. 
\end{lemma}

\begin{proof}
The matrix $\Phi_n(\theta,\boldsymbol{o})$ is block lower-triangular, and so its eigenvalues are given by the eigenvalues of its block diagonal matrices. These block diagonal matrices are given by
\begin{align}
[\Phi_n(\theta,\boldsymbol{o})]_{1:n,1:n} &= A_n(\theta_0), \\
[\Phi_n(\theta,\boldsymbol{o})]_{in+(1:n),in+(1:n)} &= {A}_n(\theta)-\Sigma^{\alpha}_{n,\infty}(\theta,\boldsymbol{o}){C}_n^{*}(\theta,\boldsymbol{o}){R}^{-1}(\boldsymbol{o}){C}_n(\theta,\boldsymbol{o}),
\end{align}
where $i=1,\dots,n_{\theta}+n_{\boldsymbol{o}}+1$. By Assumption \ref{manual_assumption2ii}, the eigenvalues of the matrix $A_n(\theta_0)$ have strictly negative real parts. Meanwhile, by Assumption \ref{manual_assumption2i}, the eigenvalues of the matrix ${A}_n(\theta)-\Sigma^{\alpha}_{n,\infty}(\theta,\boldsymbol{o}){C}_n^{*}(\theta,\boldsymbol{o}){R}^{-1}(\boldsymbol{o}){C}_n(\theta,\boldsymbol{o})$ have strictly negative real parts (e.g., \cite{Kwakernaak1972}). Thus, the eigenvalues of $\Phi_n(\theta,\boldsymbol{o})$ have strictly negative real parts. That is, $\Phi_n(\theta,\boldsymbol{o})$ is stable.
\end{proof}

 It remains to prove that Assumptions \ref{manual_assumption2i} - \ref{manual_assumption2iii} are indeed sufficient for Assumptions \ref{assumption2} - \ref{assumption5}.

\begin{proposition} \label{propA1}
Assumptions \ref{manual_assumption2i} - \ref{manual_assumption2ii} imply Assumption \ref{assumption2}.
\end{proposition}
\begin{proof}
By Lemma \ref{propA0}, $\{\mathcal{X}_n(\theta,\boldsymbol{o},t)\}_{t\geq 0}$ is a multivariate OU process which evolves according to equation \eqref{OU_eq}. By Lemma \ref{propA0a}, the matrix $\Phi_n(\theta,\boldsymbol{o})$ is stable. It follows, using standard results, that $\{\mathcal{X}_n(\theta,\boldsymbol{o},t)\}_{t\geq 0}$ is ergodic (e.g., \cite[Theorem 6.7]{Karatzas1998}). 
\end{proof}

\begin{proposition} \label{propA2}
Assumptions \ref{manual_assumption2i} - \ref{manual_assumption2ii} %and \ref{manual_assumption3} 
imply Assumption \ref{assumption3}.
\end{proposition}
\begin{proof}
By Proposition \ref{propA1}, the process $\{\mathcal{X}_n(\theta,\boldsymbol{o},t)\}_{t\geq 0}$ is ergodic. The unique invariant probability measure of this process, $\mu_{n,\theta,\boldsymbol{o}}$, is Gaussian with zero mean and covariance matrix $K_{n,\infty}(\theta,\boldsymbol{o})$, given by the solution of (e.g., \cite[Theorem 6.7]{Karatzas1998})
\begin{align}
0&= \Phi_n(\theta,\boldsymbol{o}) K_{n,\infty}(\theta,\boldsymbol{o}) + K_{n,\infty}(\theta,\boldsymbol{o})\Phi_n^{*}(\theta,\boldsymbol{o})+\Psi_n(\theta,\boldsymbol{o}){T}_n(\theta,\boldsymbol{o})\Psi_n^{*}(\theta,\boldsymbol{o}). \label{asymptotic_cov}
\end{align}  It follows, using the fact that the moments of a multivariate normal distribution are bounded, that for all $q>0$, there exists a constant $K_q>0$ such that 
\begin{equation}
\int_{\mathbb{R}^{d}} (1+||x||^{q})\mu_{n,\theta,\boldsymbol{o}}(\mathrm{d}x) \leq K_q.
\end{equation}
 It also follows, assuming that the matrices $\Phi_n(\theta,\boldsymbol{o})$, $\Psi_n(\theta,\boldsymbol{o})$, $T_n(\theta,\boldsymbol{o})$ have bounded partial derivatives with respect to $\theta,\boldsymbol{o}$, that for all $q>0$, there exist constants $K^{\theta}_{q},K_{q}^{\boldsymbol{o}}>0$, such that 
\begin{align}
\int_{\mathbb{R}^{d}} (1+||x||^{q})|\nu^{(\theta)}_{n,\theta,\boldsymbol{o},i}(\mathrm{d}x)|\leq K^{\theta}_q~,~\int_{\mathbb{R}^{d}} (1+||x||^{q})|\nu^{(\boldsymbol{o})}_{n,\theta,\boldsymbol{o},i}(\mathrm{d}x)|\leq K^{\boldsymbol{o}}_q.
\end{align}
\end{proof}

\begin{proposition}
Assumption \ref{assumption8} holds.
\end{proposition}

\begin{proof}
The matrix $\Phi_n(\theta,\boldsymbol{o})$ is independent of $\mathcal{X}_n(\theta,\boldsymbol{o})$. Thus, trivially, it satisfies the PGP (component-wise). 
\end{proof}

\begin{proposition}
Assumptions \ref{manual_assumption2i} - \ref{manual_assumption2iii} imply Assumption \ref{assumption4a}.
\end{proposition}
\begin{proof}
By Lemma \ref{propA0a}, the matrix $\Phi_n(\theta,\boldsymbol{o})$ satisfies Condition ($A_b$) in \cite{Pardoux2005}.
\iffalse
It follows that there exists $a>0$ such that 
\begin{equation}
\lim_{|x|\rightarrow\infty}\sup_{\theta\in\Theta,\boldsymbol{o}\in\Pi^{n_y}}\langle \Phi_n(\theta,\boldsymbol{o})x,x\rangle = -\lim_{|x|\rightarrow\infty} a||x||^2=-\infty.
\end{equation}
\fi
By Lemmas \ref{propA0} and \ref{propA0a}, the process $\{\mathcal{X}_n(\theta,\boldsymbol{o},t)\}_{t\geq 0}$ satisfies Condition ($A_T$) in \cite{Pardoux2005}. 
By Assumption \ref{manual_assumption2iii}, the transition density of $\{\mathcal{X}_n(\theta,\boldsymbol{o},t)\}_{t\geq0}$ satisfies Condition ($D_{sl}$) in \cite{Pardoux2005}. The result follows from \cite[Theorem 1]{Pardoux2005}.
\end{proof}

\begin{proposition}
Assumption \ref{assumption6} holds. 
\end{proposition}  
\begin{proof}
In this proof, we will make repeated use of the following standard results (e.g., \cite[Chapter 1]{Isaacson1994}): \\[-3mm]
\begin{itemize}
\item[(\textbf{R1})] Let $x\in\mathbb{R}^{d}$, and let $x^{(i)}$ denote the $i^{\text{th}}$ component of $x$. Then $|x^{(i)}|\leq ||x||$. \\[-3mm]
\item[(\textbf{R2})] Let $x_1\in\mathbb{R}^{d_1}$, $x_2\in\mathbb{R}^{d_2}$, and $x=(x_1,x_2)^T\in\mathbb{R}^{d_1+d_2}$. Then $||x_1||\leq ||x||$. \\[-3mm]
\item[(\textbf{R3})] Let $x_1,x_2\in\mathbb{R}^{d}$. Then $|x_1^Tx_2|\leq ||x_1||\hspace{.5mm}||x_2||$. \\[-3mm]
\item[(\textbf{R4})] Let $A\in\mathbb{R}^{d_1\times d_2}$, let $A^{(i,j)}$ denote the $(i,j)^{\text{th}}$ entry of $A$, and let $A^{(\cdot,j)}$ denote the $j^{\text{th}}$ column of $A$. Then $|A^{(i,j)}|\leq ||A^{(\cdot,j)}||$. \\[-3mm]
\item[(\textbf{R5})] Let $A\in\mathbb{R}^{d_1\times d_2}$ and $x\in\mathbb{R}^{d_2}$. Then $||Ax||\leq ||A||\hspace{.5mm}||x||$. \\[-3mm]
\item[(\textbf{R6})] Let $A\in\mathbb{R}^{d_1\times d_2}$. Then there exists $0<K<\infty$ such that $||A||\leq K$. \\[-3mm]
\item[(\textbf{R7})] Let $A\in\mathbb{R}^{d_1\times d_2}$ and $B\in\mathbb{R}^{d_2\times d_3}$. Then $|\mathrm{Tr}(AB)|\leq ||A||_{\mathrm{HS}}\hspace{.5mm}||B||_{\mathrm{HS}}$. \\[-3mm]
\end{itemize} 
Using these results, we can verify directly that the functions $\varphi_n,\psi_n,\phi_n,\eta_n,\zeta_n,\xi_n$ have the PGP (component-wise). In particular, allowing the value of the constant $K>0$ to vary from line to line, we have
\allowdisplaybreaks{
\begin{alignat}{2} 
|\varphi^{(i)}_n(\theta,\boldsymbol{o},\mathcal{X})|&\leq ||\varphi_n(\theta,\boldsymbol{o},\mathcal{X})||=||{C}_n(\theta,\boldsymbol{o})\hat{\alpha}_n(\theta,\boldsymbol{o})|| &&(\textbf{R1}) \nonumber \\
&\leq ||{C}_n(\theta,\boldsymbol{o})||\hspace{.5mm} ||\hat{\alpha}_n(\theta,\boldsymbol{o})||\leq K||\hat{\alpha}_n(\theta,\boldsymbol{o})|| \leq K||\mathcal{X}(\theta,\boldsymbol{o})||&&(\textbf{R5},\textbf{R6},\textbf{R2}) \nonumber \\[2mm]
|\zeta_n(\theta,\boldsymbol{o},\mathcal{X})| &=\left| [{C}_n(\theta,\boldsymbol{o})\hat{\alpha}_n(\theta,\boldsymbol{o})]^T [C_n(\boldsymbol{o})\alpha_n-\tfrac{1}{2}C_n(\boldsymbol{o})\hat{\alpha}_n(\theta,\boldsymbol{o})] \right| \nonumber\\
&\leq || {C}_n(\theta,\boldsymbol{o})\hat{\alpha}_n(\theta,\boldsymbol{o})||\hspace{1mm} ||C_n(\theta,\boldsymbol{o})\alpha_n-\tfrac{1}{2}C_n(\theta,\boldsymbol{o})\hat{\alpha}_n(\theta,\boldsymbol{o})|| &&(\textbf{R3}) \hspace{-4mm} \nonumber \\
&\leq ||{C}_n(\theta,\boldsymbol{o})||^2\hspace{.5mm} ||\hat{\alpha}_n(\theta,\boldsymbol{o})||\hspace{.5mm} ||\alpha_n-\tfrac{1}{2}\hat{\alpha}_n(\theta,\boldsymbol{o})|| &&(\textbf{R5}) \nonumber \\
&\leq K ||\hat{\alpha}_n(\theta,\boldsymbol{o})||\hspace{.5mm} ||\alpha_n-\tfrac{1}{2}\hat{\alpha}_n(\theta,\boldsymbol{o})|| \leq K ||\mathcal{X}(\theta,\boldsymbol{o})||^2&&(\textbf{R6},\textbf{R2}) \nonumber \\[2mm]
|\eta^{(i,j)}_n(\theta,\boldsymbol{o},\mathcal{X})| &\leq ||\eta^{(\cdot,j)}_n(\theta,\boldsymbol{o},\mathcal{X})||=||{C}_n(\theta,\boldsymbol{o})\hat{\alpha}^{\theta_j}_n(\theta,\boldsymbol{o})|| &&(\textbf{R4}) \nonumber\\
&\leq ||{C}_n(\theta,\boldsymbol{o})||\hspace{.5mm}||\hat{\alpha}^{\theta_j}_n(\theta,\boldsymbol{o})|| \leq K ||\hat{\alpha}^{\theta_j}_n(\theta,\boldsymbol{o})||\leq K ||\mathcal{X}(\theta,\boldsymbol{o})|| \nonumber &&(\textbf{R5},\textbf{R6},\textbf{R2})\\[2mm]
|\psi^{(i)}_n(\theta,\boldsymbol{o},\mathcal{X})|&=\left|[{C}_n(\theta,\boldsymbol{o})\hat{\alpha}^{\theta_i}_n(\theta,\boldsymbol{o})]^T [C_n(\theta,\boldsymbol{o})\alpha_n-C_n(\theta,\boldsymbol{o})\hat{\alpha}_n(\theta,\boldsymbol{o})]\right|  \nonumber\\
& \leq ||{C}_n(\theta,\boldsymbol{o})\hat{\alpha}^{\theta_i}_n(\theta,\boldsymbol{o})|| \hspace{1mm} ||C_n(\theta,\boldsymbol{o})\alpha_n-C_n(\theta,\boldsymbol{o})\hat{\alpha}_n(\theta,\boldsymbol{o})|| &&(\textbf{R3}) \nonumber \\
& \leq ||{C}_n(\theta,\boldsymbol{o})||^2||\hat{\alpha}^{\theta_i}_n(\theta,\boldsymbol{o})|| \hspace{1mm} ||\alpha_n-\hat{\alpha}_n(\theta,\boldsymbol{o})|| &&(\textbf{R5}) \nonumber \\
&\leq K ||\hat{\alpha}^{\theta_i}_n(\theta,\boldsymbol{o})|| ||\alpha_n-\hat{\alpha}_n(\theta,\boldsymbol{o})||\leq K ||\mathcal{X}(\theta,\boldsymbol{o})||^2 &&(\textbf{R6},\textbf{R2}) \nonumber \\[2mm]
|\iota_n(\theta,\boldsymbol{o},\mathcal{X})|&= \left|\mathrm{Tr}\left[{M}_n\Sigma^{\alpha}_n(\theta,\boldsymbol{o})\right]\right|\leq ||M_n||_{\mathrm{HS}}\hspace{.5mm}||\Sigma^{\alpha}_n(\theta,\boldsymbol{o})||_{\mathrm{HS}} &&(\textbf{R7}) \nonumber \\
& \leq K ||\Sigma^{\alpha}_n(\theta,\boldsymbol{o})||_{\mathrm{HS}}= K ||\mathrm{vec}\left(\Sigma^{\alpha}_n(\theta,\boldsymbol{o})\right)\hspace{-.5mm}||\leq K||\mathcal{X}(\theta,\boldsymbol{o})|| &&(\textbf{R6},\textbf{R2}) \nonumber \\[2mm]
|\phi^{(i)}_n(\theta,\boldsymbol{o},\mathcal{X})|&= \left| \mathrm{Tr}^{\boldsymbol{o}_i}\left[{M}_n\Sigma^{\alpha}_n(\theta,\boldsymbol{o})\right]\right|= \left| \mathrm{Tr}\left[{M}_n\Sigma_n^{\alpha,\boldsymbol{o}_i}(\theta,\boldsymbol{o})\right]\right| \nonumber \\
&\leq ||M_n||_{\mathrm{HS}}\hspace{.5mm} ||\Sigma^{\alpha,\boldsymbol{o}_i}_n(\theta,\boldsymbol{o})||_{\mathrm{HS}}\leq K ||\Sigma^{\alpha,\boldsymbol{o}_i}_n(\theta,\boldsymbol{o})||_{\mathrm{HS}} &&(\textbf{R7},\textbf{R6}) \nonumber \\
& = K ||\mathrm{vec}\left(\Sigma^{\alpha}_n(\theta,\boldsymbol{o})\right)\hspace{-.5mm}||\leq K||\mathcal{X}(\theta,\boldsymbol{o})||  &&(\textbf{R2}). \nonumber
\end{alignat}
}
\end{proof}

\begin{proposition}
Assumptions \ref{manual_assumption2i} - \ref{manual_assumption2ii} imply Assumption \ref{assumption5}. 
\end{proposition}
\begin{proof}
By Lemma \ref{propA0}, $\{\mathcal{X}_n(\theta_n(t),\boldsymbol{o}_n(t),t)\}_{t\geq 0}$ is a multivariate OU process which evolves according to equation \eqref{OU_eq}.
%\begin{equation}
%\mathrm{d} \mathcal{X}_n(\theta(t),\boldsymbol{o}(t),t)  = \Phi_n(\theta(t),\boldsymbol{o}(t)) \mathcal{X}_n(\theta(t),\boldsymbol{o}(t),t)\mathrm{d}t + \Psi_n(\theta(t),\boldsymbol{o}(t))\mathrm{d}b_n(t).
%\end{equation}
The mean and covariance of this process can be obtained straightforwardly as
\begin{align}
m_n(\theta_n(t),\boldsymbol{o}_n(t),t)&= e^{\Phi_n(\theta_n(t),\boldsymbol{o}_n(t))t}\mathcal{X}_n(\theta_n(0),\boldsymbol{o}_n(0),0) \\
P_n(\theta_n(t),\boldsymbol{o}_n(t),t) &= \int_0^t e^{\Phi_n(\theta_n(t),\boldsymbol{o}_n(t))(t-s)}T_n(\theta_n(t),\boldsymbol{o}_n(t))e^{\Phi_n^T(\theta_n(t),\boldsymbol{o}_n(t))(t-s)}\mathrm{d}s.
\end{align}
By Lemma \ref{propA0a}, the matrix $\Phi_n(\theta,\boldsymbol{o})$ is stable for all $\theta\in\Theta$, $\boldsymbol{o}\in\Pi^{n_y}$. It follows that $\Phi_n(\theta_n(t),\boldsymbol{o}_n(t))$ is stable for all $t\geq 0$. Thus, there exist positive constants $\alpha,\beta>0$ such that $||e^{\Phi_n(\theta_n(t),\boldsymbol{o}_n(t))t}||\leq \alpha e^{-\beta t}$ (e.g., \cite[Chapter 3]{Hsu2013}). This implies, in particular, that there exist positive constants $K_1,K_2>0$ such that 
\begin{align}
||m_n(\theta_n(t),\boldsymbol{o}_n(t),t)&||\leq ||\mathcal{X}_n(\theta_n(0),\boldsymbol{o}_n(0),0)|| \alpha e^{-\beta t} \leq K_1 e^{-\beta t}, \\
||P_n(\theta_n(t),\boldsymbol{o}_n(t),t)&||\leq \sup_{s\in[0,t]}||T_n(\theta_n(s),\boldsymbol{o}_n(s))|| \int_0^t \alpha^2 e^{-2\beta(t-s)}\mathrm{d}s
%\leq \sup_{0\leq s\leq t} ||T_n(\theta(s),\boldsymbol{o}(s))|| \alpha^2 \left[1-e^{-2\beta t}\right] 
\leq K_2.%\left[1-e^{-2\beta t}\right].
\end{align}
where, in a slight abuse of notation, we use $||\cdot||$ to denote the standard Euclidean norm in the first line, and $||\cdot||$ to denote the induced matrix norm in the second line. The matrix $P_n$ is positive-definite, and thus its norm $||P_n||$ is equal to its largest eigenvalue, say $\lambda_n$ (e.g., \cite[Chapter 1]{Isaacson1994}). It follows, writing $\smash{P_n^{\frac{1}{2}}}$ to denote the principal square root of $P_n$, that $\smash{||P_n^{\frac{1}{2}}||}$ is equal to $\smash{\lambda_n^{\frac{1}{2}}}$. We thus have
\begin{align}
 ||P_n^{\frac{1}{2}}(\theta_n(t),\boldsymbol{o}_n(t),t)|| = \lambda^{\frac{1}{2}}(\theta_n(t),\boldsymbol{o}_n(t),t)&\leq \max\{1,\lambda(\theta_n(t),\boldsymbol{o}_n(t),t)\}%\\
 %&=\max\{1,||P_n(\theta(t),\boldsymbol{o}(t),t)||\} \nonumber \\
%&
\leq \max\{1,K_2\}. \nonumber 
\end{align}
%we h in addition, we have made use of the inequalities $||\mathcal{X}_{n}(\theta(0),\boldsymbol{o}(0),0)||\leq K$ and $\sup_{0\leq s\leq t}||T_n(\theta(s),\boldsymbol{o}(s))||\leq K$ for some $K>0$. %\footnote{The second of these bounds follows from the boundedness of $T(\theta,\boldsymbol{o})$, which itself follows from the boundedness of $Q_n(\theta)$ and $R(\boldsymbol{o})$.} 
It follows, defining $K = \max\{1,K_1^q,K_2^q\}>0$, that for all $q>0$, 
\begin{align}
||m_n(\theta_n(t),\boldsymbol{o}_n(t),t)||^{q}&\leq K e^{-\beta qt}, \\
||P^{\frac{1}{2}}_n(\theta_n(t),\boldsymbol{o}_n(t),t)||^{q}%\leq ||T_n(\theta(t),\boldsymbol{o}_{n}(t))||^q \int_0^t \alpha^2 e^{-2\beta(t-s)}\mathrm{d}s
%&\leq K\left[1-e^{-2\beta t}\right] 
&\leq K.
\end{align}
Now, using elementary properties of the multivariate normal distribution, it is possible to write $\smash{\mathcal{X}_n(\theta_n(t),\boldsymbol{o}_n(t),t) = m_n(\theta_n(t),\boldsymbol{o}_n(t),t) + P_n^{\frac{1}{2}}(\theta_n(t),\boldsymbol{o}_n(t),t)z_n}$,
\iffalse
\begin{equation}
\smash{\mathcal{X}_n(\theta(t),\boldsymbol{o}_{n}(t),t) = m_n(\theta(t),\boldsymbol{o}_{n}(t),t) + P_n^{\frac{1}{2}}(\theta(t),\boldsymbol{o}_{n}(t),t)z_n},
\end{equation}
\fi
where $z_n\sim\mathcal{N}(0,1_n)$. We thus have, also making use of the inequalities $||x+y||^q\leq ||x||^q+||y||^q$, and $||Ax||\leq ||A||\hspace{.2mm}||x||$, that
\begin{align}
\mathbb{E}\left[||\mathcal{X}_n(\theta_{n}(t),\boldsymbol{o}_{n}(t),t)||^{q}\right] &\leq ||m_n(\theta_{n}(t),\boldsymbol{o}_{n}(t),t)||^{q} + ||P_n^{\frac{1}{2}}(\theta_{n}(t),\boldsymbol{o}_{n}(t),t)||^{q}\mathbb{E}\left[||z_n||^{q}\right] \hspace{-2mm} \\
&\leq  ||m_n(\theta_{n}(t),\boldsymbol{o}_{n}(t),t)||^{q} + ||P_n^{\frac{1}{2}}(\theta_{n}(t),\boldsymbol{o}_{n}(t),t)||^{q} \nonumber \\[1mm]
&\leq K\left[1+e^{-\beta qt}\right] . \nonumber
\end{align}
It follows, in particular, that for all $q>0$, and for all $t\geq 0$, $\mathbb{E}[||\mathcal{X}_n(\theta_{n}(t),\boldsymbol{o}_{n}(t),t)||^q]<\infty$, and there exists $K>0$ such that for all $t$ sufficiently large, 
\begin{subequations}
\begin{align}
\mathbb{E}\left[\sup_{s\leq t}||\mathcal{X}_n(\theta,\boldsymbol{o},s)||^q\right]&\leq K\sqrt{t}~,~~~\forall \theta\in\Theta~,~\forall \boldsymbol{o}\in\Pi^{n_y}, \\
\mathbb{E}\left[\sup_{s\leq t}||\mathcal{X}_n(\theta_{n}(s),\boldsymbol{o}_{n}(s),s)||^q\right]&\leq K\sqrt{t}.
\end{align}
\end{subequations}
\end{proof}

\color{black}
\section{The Algebraic Ricatti Equation} \label{app:ARE}
In this Appendix, we recall a well known result on asymptotic solutions of the differential Ricatti equation, original due to Curtain \cite{Curtain1978a}. 

\begin{theorem*} \label{theorem:ARE}
Suppose that the pair $(\mathcal{A}(\theta),\mathcal{B}\mathcal{Q}^{\frac{1}{2}}(\theta))$ is exponentially stabilisable, and that the pair $(\mathcal{A}(\theta),\mathcal{C}(\theta,\boldsymbol{o}))$ is exponentially detectable (see, e.g., \cite{Zhang2018} for definitions of `exponentially stabilisable' and `exponentially detectable'). Then, as $t\rightarrow\infty$, 
the solution of the differential Ricatti equation %, namely 
$\Sigma(\theta,\boldsymbol{o},t)$
converges strongly to $\Sigma_{\infty}(\theta,\boldsymbol{o})$, the unique weak non-negative solution of the algebraic Ricatti equation
\begin{align}
0&=\mathcal{A}(\theta) \Sigma_{\infty}(\theta,\boldsymbol{o}) + \Sigma_{\infty}(\theta,\boldsymbol{o})\mathcal{A}^{\dagger}(\theta) \label{inf_dim_alg_ricatti} \\
&\hspace{5mm}+\mathcal{B}\mathcal{Q}(\theta)\mathcal{B}^{\dagger}-\Sigma_{\infty}(\theta,\boldsymbol{o}) \mathcal{C}^{\dagger}(
\theta,\boldsymbol{o})\mathcal{R}^{-1}(\boldsymbol{o})\mathcal{C}(\theta,\boldsymbol{o})\Sigma_{\infty}(\theta,\boldsymbol{o}). \nonumber 
 \end{align}
 \end{theorem*}
 The following result, which provides an explicit representation for the asymptotic sensor placement objective function in terms of the solution of the algebraic Ricatti equation, then follows straightforwardly. For brevity, the proof is omitted.
  
\begin{corollary}
Suppose that the assumptions of Theorem \ref{theorem:ARE} hold. Suppose also that $\mathcal{M}(t)$ converges strongly to a bounded linear operator $\mathcal{M}_{\infty}:\mathcal{H}\rightarrow\mathcal{H}$ as $t\rightarrow\infty$. Then
 \begin{equation}
 \tilde{\mathcal{J}}(\theta,\boldsymbol{o}) := \lim_{t\rightarrow\infty} \frac{1}{t} \mathcal{J}_t(\theta,\boldsymbol{o})  = \mathrm{Tr}\left[ M_{\infty}\Sigma_{\infty}(\theta,\boldsymbol{o})\right].
 \end{equation}
 \end{corollary}

\section{The Spatial Weighting Operator} \label{app:weighting}
In this Appendix, we provide an explicit definition of the spatial weighting operator $\mathcal{M}(t):\mathcal{H}\rightarrow\mathcal{H}$. Throughout this paper, we assume that this operator is defined according to
\begin{equation}
(\mathcal{M}(t)\varphi) (\boldsymbol{x}) = m(\boldsymbol{x},t)\varphi(\boldsymbol{x})~,~~~\varphi \in\mathcal{H},
\end{equation}
where $m(\cdot,t)\in\mathcal{H}$ is a spatial weighting function to be defined below. In this case, it is possible to show (e.g., \cite{Chen1975,Colantuoni1978,Kumar1978}), using Mercer's Theorem \cite{Mercer1909}, that 
\begin{equation}
\mathrm{Tr}\left[\mathcal{M}(t)\Sigma(\theta,\boldsymbol{o},t)\right] = \int_{\Pi} m(\boldsymbol{x},t) \tilde{\Sigma}(\theta,\boldsymbol{o},\boldsymbol{x},\boldsymbol{x},t)\mathrm{d}\boldsymbol{x},
\end{equation}
where $\tilde{\Sigma}(\theta,\boldsymbol{o},\cdot,\cdot,t):\Pi\times\Pi\rightarrow\mathbb{R}$ is the kernel operator (or covariance function) associated with the covariance operator $\Sigma(\theta,\boldsymbol{o},t)$. Thus, in particular, the sensor placement objective function can be written in the form 
\begin{equation}
{\mathcal{J}}_t(\theta,\boldsymbol{o}) = \int_0^t \left[\int_{\Pi} m(\boldsymbol{x},s) \tilde{\Sigma}(\theta,\boldsymbol{o},\boldsymbol{x},\boldsymbol{x},s)\mathrm{d}\boldsymbol{x}\right] \mathrm{d}s.
\end{equation}
We can now define the explicit form of the spatial weighting function $m(\cdot,t)\in\mathcal{H}$. In particular, we will assume that
\begin{equation}
m(\boldsymbol{x},t) =  c_0 \mathds{1}_{\boldsymbol{x}\in\Pi_{w}(t)} + c_1 \mathds{1}_{\boldsymbol{x}\in\Pi\setminus \Pi_{w}(t)}~,~~~\boldsymbol{x}\in\Pi,
\end{equation} 
where $\Pi_{w}(t)\subseteq\Pi$ denotes a `weighted' or `target' spatial region, which corresponds to the region in which we are most interested in minimising the uncertainty in the optimal state estimate, and $0\leq c_1\leq c_0\leq 1$ are positive constants. The choice of the constants $c_0$ and $c_1$, or equivalently the ratio $\frac{c_0}{c_1}\in[1,\infty)$, determines the extent to which the objective function will prioritise minimising the uncertainty in the state estimate in the region $\Pi_{w}(t)$, relative to the region $\Pi\setminus\Pi_{w}(t)$. In our numerics, we use the following specific definitions of $c_0$, $c_1$, and $\Pi_{w}(t)$. 

\subsubsection*{Simulation I} In this simulation, we set $c_0 = 1$, $c_1 = 0$, and $\Pi_{w}(t) = \Pi_{w} = \bigcup_{i=1}^{8}\{\boldsymbol{x}\in\Pi:|\boldsymbol{x} - \boldsymbol{x}_i|\leq r\}$, where $\{\boldsymbol{x}_i\}_{i=1}^{8}$ are the 8 `target' locations defined in \eqref{targets}, and $r>0$ is a small positive constant.
This yields $m(\boldsymbol{x},t) = \sum_{i=1}^{8} \mathds{1}_{\{\boldsymbol{x}\in\Pi:|\boldsymbol{x}- \boldsymbol{x}_i|\leq r\}}$, and thus
\begin{equation}
\mathrm{Tr}\left[\mathcal{M}(t)\Sigma(\theta,\boldsymbol{o},t)\right] = \sum_{i=1}^{8} \int\limits_{\{\boldsymbol{x}\in\Pi : |\boldsymbol{x}-\boldsymbol{x}_i|<r\}} \hspace{-5mm}\tilde{\Sigma}(\theta,\boldsymbol{o},\boldsymbol{x},\boldsymbol{x},t)\mathrm{d}\boldsymbol{x}
\end{equation}
so that the objective function only seeks to minimise the uncertainty in the state estimate close to the target locations $\{\boldsymbol{x}_i\}_{i=1}^{8}$. We remark that similar results are obtained if one sets $c_1 = \varepsilon$, for some $\varepsilon\ll 1$. In this case, the weighted trace of the covariance is given by a similar expression to \eqref{eqc07} (see below).

\subsubsection*{Simulation II, IV, V} In these simulations, we set  $c_0=c_1=1$, or equivalently $\frac{c_0}{c_1} = 1$. In this case, the spatial weighting function reduces to the identity,$m(\boldsymbol{x},t)= \mathds{1}_{\boldsymbol{x}\in\Pi_w} + \mathds{1}_{\boldsymbol{x}\in\Pi\setminus\Pi_{w}} = \mathds{1}_{\boldsymbol{x}\in\Pi}$, and we have 
\begin{equation}
\mathrm{Tr}\left[\mathcal{M}(t)\Sigma(\theta,\boldsymbol{o},t)\right] = \int_{\Pi} \tilde{\Sigma}(\theta,\boldsymbol{o},\boldsymbol{x},\boldsymbol{x},t)\mathrm{d}\boldsymbol{x}
\end{equation}
%\begin{equation}
%m(\boldsymbol{x},t) = \mathds{1}_{\boldsymbol{x}\in\Pi_w} + \mathds{1}_{\boldsymbol{x}\in\Pi\setminus\Pi_{w}} = \mathds{1}_{\boldsymbol{x}\in\Pi}
%\end{equation}
so that the objective function equally weights the uncertainty in the state estimate at all spatial locations. 

\subsubsection*{Simulation III}. In this simulation we set $c_0 = 1$, $c_1 = 0.01$, and $\Pi_w(t) = \bigcup_{i=1}^{4}\{\boldsymbol{x}\in\Pi:|\boldsymbol{x} - \boldsymbol{x}_i(t) |\leq r\}$, where $\{\boldsymbol{x}_i(t)\}_{i=1}^{4}$ are the 4 time-varying locations shown in purple in Figure \ref{fig2g}. This yields $m(\boldsymbol{x},t) = \sum_{i=1}^{4} \mathds{1}_{\{\boldsymbol{x}\in\Pi:|\boldsymbol{x}- \boldsymbol{x}_i(t)|\leq r\}} + 0.01 \mathds{1}_{\{\boldsymbol{x}\in\Pi:\cap_{i=1}^4|\boldsymbol{x}- \boldsymbol{x}_i(t)|> r\}}$ and
\begin{equation}
\mathrm{Tr}\left[\mathcal{M}(t)\Sigma(\theta,\boldsymbol{o},t)\right] = \sum_{i=1}^{4} \int\limits_{\{\boldsymbol{x}\in\Pi : |\boldsymbol{x}-\boldsymbol{x}_i(t)|<r\}} \hspace{-10mm}\tilde{\Sigma}(\theta,\boldsymbol{o},\boldsymbol{x},\boldsymbol{x},t)\mathrm{d}\boldsymbol{x} \hspace{2mm} + 0.01\hspace{-8mm} \int\limits_{\{\boldsymbol{x}\in\Pi:\cap_{i=1}^4|\boldsymbol{x}- \boldsymbol{x}_i(t)|> r\}}\hspace{-13mm}\tilde{\Sigma}(\theta,\boldsymbol{o},\boldsymbol{x},\boldsymbol{x},t)\mathrm{d}\boldsymbol{x} \label{eqc07}
\end{equation}
so that the objective function strongly weights the uncertainty in the state estimate in the regions close to the time-varying locations $\{\boldsymbol{x}_i(t)\}_{i=1}^4$, but also contains a contribution from the uncertainty in the state estimate at all other locations.

%We begin by recalling that the covariance operator $\Sigma(\theta,\boldsymbol{o},t):\mathcal{H}\rightarrow\mathcal{H}$ is an integral operator with kernel $\tilde{\Sigma}(\theta,\boldsymbol{o},\boldsymbol{x},\boldsymbol{y},t):\mathcal{H}\times\mathcal{H}\rightarrow\mathbb{R}$ (e.g., \cite{Colantuoni1978}). In particular, for any $\varphi\in\mathcal{H}$, we have that 
%\begin{equation}
%\left(\Sigma(\theta,\boldsymbol{o},t)\varphi\right) (\boldsymbol{x})  = \int_{\Pi} \tilde{\Sigma}(\theta,\boldsymbol{o},\boldsymbol{x},\boldsymbol{y},t)\varphi(\boldsymbol{y})\mathrm{d}\boldsymbol{y}.
%\end{equation}
%Using this representation, it is straightforward to show that 
%\begin{equation}
%\end{equation}

\section{Additional Details for Figure \ref{fig0}} \label{app:fig0}
In this Appendix, we provide additional details of the parameter values and the sensor placements used to generated Figure \ref{fig0}. The true parameters $\theta^{*}$ and the incorrect parameters $\theta_0$ are given respectively by 
\begin{align}
\theta^{*} &= (\rho_0 = 0.5,\sigma^2 = 0.2,\zeta=0.5,\rho_1=0.05,\gamma=2,\alpha=\tfrac{\pi}{4},\mu_x=0.3,\mu_y=-0.3), \\[1mm]
\theta_0 &= ({\rho}_{0} = 0.5,{\sigma}^2 = 0.2,{\zeta}=0.5,{\rho}_{1}=0.30,{\gamma}=2,{\alpha}=\tfrac{\pi}{4},{\mu}_{x}=0.3,{\mu}_{y}=-0.3). 
\end{align}
The optimal sensor placement consists of 400 sensors, uniformly distributed over the spatial domain $\Pi = [0,1]^2$. The sub-optimal sensor placement consists of 400 sensors, distributed at random over $[0.5,0.95]^2$.
\color{black}

%%%%%%%%%%%%

%%%%%%%%%%%%%%
%%% REFERENCES %%%
%%%%%%%%%%%%%%

\bibliographystyle{siamplain}
\bibliography{references}

\begin{thebibliography}{100}

\bibitem{Aidarous1978}
{\sc S.~Aidarous, M.~R. Gevers, and M.~J. Install{\'{e}}}, {\em {On the
  Asymptotic Behavious of Sensors' Allocation Algorithm in Stochastic
  Distributed Systems}}, in Distrib. Param. Syst. Model. Identification. Lect.
  Notes Control Inf. Sci. Vol. 1, A.~Ruberti, ed., Springer, Berlin,
  Heidelberg, 1978, \url{https://doi.org/10.1007/BFb0003732}.

\bibitem{Aidarous1975}
{\sc S.~E. Aidarous, M.~R. Gevers, and M.~J. Install{\'{e}}}, {\em {Optimal
  sensors' allocation strategies for a class of stochastic distributed
  systems}}, Int. J. Control, 22 (1975), pp.~197--213,
  \url{https://doi.org/10.1080/00207177508922076}.

\bibitem{Aihara1992}
{\sc S.~Aihara}, {\em {Regularized Maximum Likelihood Estimate for an
  Infinite-Dimensional Parameter in Stochastic Parabolic Systems}}, SIAM J.
  Control Optim., 30 (1992), pp.~745--764.

\bibitem{Aihara1988}
{\sc S.~Aihara and Y.~Sunahara}, {\em {Identification of an
  Infinite-Dimensional Parameter for Stochastic Diffusion Equations}}, SIAM J.
  Control Optim., 26 (1988), pp.~1062--1075,
  \url{https://doi.org/10.1137/0326058}.

\bibitem{Aihara1994}
{\sc S.~I. Aihara}, {\em {Maximum likelihood estimate for discontinuous
  parameter in stochastic hyperbolic systems}}, Acta Appl. Math., 35 (1994),
  pp.~131--151, \url{https://doi.org/10.1007/BF00994914}.

\bibitem{Aihara1991}
{\sc S.~I. Aihara and A.~Bagchi}, {\em {Parameter identification for hyperbolic
  stochastic systems}}, J. Math. Anal. Appl., 160 (1991), pp.~485--499,
  \url{https://doi.org/10.1016/0022-247X(91)90321-P}.

\bibitem{Amouroux1978}
{\sc M.~Amouroux, J.~P. Babary, and C.~Malandrakis}, {\em {Optimal location of
  sensors for linear stochastic distributed parameter systems}}, in Distrib.
  Param. Syst. Model. Identification. Lect. Notes Control Inf. Sci. Vol. 1,
  A.~Ruberti, ed., Springer, Berlin, Heidelberg, 1978, pp.~92--113,
  \url{https://doi.org/10.1007/BFb0003733}.

\bibitem{Athans1972}
{\sc M.~Athans}, {\em {On the Determination of Optimal Costly Measurement
  Strategies for Linear Stochastic Systems}}, Automatica, 8 (1972),
  pp.~397--412, \url{https://doi.org/10.1016/S1474-6670(17)68422-2}.

\bibitem{Bagchi1981}
{\sc A.~Bagchi and V.~Borkar}, {\em {Parameter identification in infinite
  dimensional linear systems}}, in 20th IEEE Conf. Decis. Control Incl. Symp.
  Adapt. Process., 1981, pp.~62--66,
  \url{https://doi.org/10.1109/CDC.1981.269443}.

\bibitem{Bagchi1984}
{\sc A.~Bagchi and V.~Borkar}, {\em {Parameter identification in infinte
  dimensional linear systems}}, Stochastics, 12 (1984), pp.~201--213,
  \url{https://doi.org/10.1080/17442508408833301}.

\bibitem{Bain2009}
{\sc A.~Bain and D.~Crisan}, {\em {Fundamentals of Stochastic Filtering}},
  Springer-Verlag New York, 1~ed., 2009,
  \url{https://doi.org/10.1007/978-0-387-76896-0}.

\bibitem{Balakrishnan1973}
{\sc A.~V. Balakrishnan}, {\em {Stochastic Differential Systems I}}, vol.~84 of
  Lecture Notes in Economics and Mathematical Systems, Springer, Berlin,
  Heidelberg, 1973, \url{https://doi.org/10.1007/978-3-642-80759-6}.

\bibitem{Balakrishnan1975}
{\sc A.~V. Balakrishnan}, {\em {Identification - Inverse Problems for Partial
  Differential Equations: A Stochastic Formulation}}, in Optim. Tech. IFIP
  Tech. Conf. Lect. Notes Comput. Sci., G.~I. Marchuk, ed., Springer, Berlin,
  Heidelberg, 1975, pp.~1--12,
  \url{https://doi.org/10.1007/978-3-662-38527-2_1}.

\bibitem{Balakrishnan1975a}
{\sc A.~V. Balakrishnan}, {\em {Identification and stochastic control of a
  class of distributed systems with boundary noise}}, in Control Theory, Numer.
  Methods Comput. Syst. Model., A.~Bensoussan and J.~Lions, eds., Springer
  Verlag, New York, 1975, pp.~163--178.

\bibitem{Banks1984}
{\sc H.~T. Banks and K.~Kunisch}, {\em {The Linear Regulator Problem for
  Parabolic Systems}}, SIAM J. Control Optim., 22 (1984), pp.~684--698,
  \url{https://doi.org/10.1137/0322043}.

\bibitem{Banks2012}
{\sc H.~T. Banks and K.~Kunisch}, {\em {Estimation Techniques for Distributed
  Parameter Systems}}, Springer Science {\&} Business Media, 2012.

\bibitem{Barrera1997}
{\sc P.~Barrera and B.~Spagnolo}, {\em {A Stochastic Model For The
  Advection-diffusion Equation Of Air Pollution In The Atmosphere}}, WIT Trans.
  Ecol. Environ., 21 (1997), \url{https://doi.org/10.2495/AIR970261}.

\bibitem{Bashirov2003}
{\sc A.~Bashirov}, {\em {Partially Observable Linear Systems Under Dependent
  Noises}}, Birkhauser Basel, 2003.

\bibitem{Baumeister1997}
{\sc J.~Baumeister, W.~Scondo, M.~A. Demetriou, and I.~G. Rosen}, {\em {On-Line
  Parameter Estimation for Infinite-Dimensional Dynamical Systems}}, SIAM J.
  Control Optim., 35 (1997), pp.~678--713,
  \url{https://doi.org/10.1137/S0363012994270928}.

\bibitem{Bensoussan1971}
{\sc A.~Bensoussan}, {\em {Filtrage optimal des syst{\`{e}}mes
  lin{\'{e}}aires}}, Dunod, Paris, 1971.

\bibitem{Bensoussan1972}
{\sc A.~Bensoussan}, {\em {Optimisation of Sensors' Location in a Distributed
  Filtering Problem}}, in Stab. Stoch. Dyn. Sysems. Lect. Notes Math.,
  vol.~294, Springer, Berlin, Germany, 1972, pp.~62--84,
  \url{https://doi.org/10.1007/BFb0064935}.

\bibitem{Bensoussan2007}
{\sc A.~Bensoussan, G.~{Da Prato}, M.~Delfour, and S.~Mitter}, {\em
  {Representation and Control of Infinite Dimensional Systems}},
  Birkh{\"{a}}user Basel, 2~ed., 2007,
  \url{https://doi.org/10.1007/978-0-8176-4581-6}.

\bibitem{Beskos2020}
{\sc A.~Beskos, D.~Crisan, A.~Jasra, N.~Kantas, and H.~Ruzayqat}, {\em
  {Score-Based Parameter Estimation for a Class of Continuous-Time State Space
  Models}}, arXiv Prepr.,  (2020), \url{https://arxiv.org/abs/2008.07803}.

\bibitem{Borkar2008}
{\sc V.~Borkar}, {\em {Stochastic Approximation: A Dynamical Systems
  Viewpoint}}, Cambridge University Press, 2008,
  \url{https://doi.org/10.1007/978-93-86279-38-5}.

\bibitem{Borkar1997}
{\sc V.~S. Borkar}, {\em {Stochastic approximation with two time scales}},
  Syst. Control Lett., 29 (1997), pp.~291--294,
  \url{https://doi.org/10.1016/S0167-6911(97)90015-3}.

\bibitem{Bottou2010}
{\sc L.~Bottou}, {\em {Large-Scale Machine Learning with Stochastic Gradient
  Descent}}, in Proc. COMPSTAT'2010, Y.~Lechevallier and G.~Saporta, eds.,
  Heidelberg, 2010, Physica-Verlag HD, pp.~177--186,
  \url{https://doi.org/10.1007/978-3-7908-2604-3_16}.

\bibitem{Breit2018}
{\sc D.~Breit}, {\em {An Introduction to Stochastic Navier–Stokes
  Equations}}, in New Trends Results Math. Descr. Fluid Flows,
  M.~Bul{\'{i}}{\v{c}}ek, E.~Feireisl, and M.~Pokorn{\'{y}}, eds., Springer
  International Publishing, Cham, 2018, pp.~1--51,
  \url{https://doi.org/10.1007/978-3-319-94343-5_1}.

\bibitem{Brown2001}
{\sc P.~E. Brown, P.~J. Diggle, M.~E. Lord, and P.~C. Young}, {\em
  {Space–time calibration of radar rainfall data}}, J. R. Stat. Soc. Ser. C
  (Applied Stat., 50 (2001), pp.~221--241,
  \url{https://doi.org/10.1111/1467-9876.00230}.

\bibitem{Burns2015a}
{\sc J.~Burns and C.~Rautenberg}, {\em {Solutions and Approximations to the
  Riccati Integral Equation with Values in a Space of Compact Operators}}, SIAM
  J. Control Optim., 53 (2015), pp.~2846--2877,
  \url{https://doi.org/10.1137/130948513}.

\bibitem{Burns2009}
{\sc J.~A. Burns, E.~M. Cliff, and C.~Rautenberg}, {\em {A distributed
  parameter control approach to optimal filtering and smoothing with mobile
  sensor networks}}, in 2009 17th Mediterr. Conf. Control Autom., 2009,
  pp.~181--186, \url{https://doi.org/10.1109/MED.2009.5164536}.

\bibitem{Burns2011}
{\sc J.~A. Burns and C.~N. Rautenberg}, {\em {Bochner integrable solutions to
  Riccati partial differential equations and optimal sensor placement}}, in
  Proc. 2011 Am. Control Conf., 2011, pp.~2368--2373,
  \url{https://doi.org/10.1109/ACC.2011.5991343}.

\bibitem{Burns2015}
{\sc J.~A. Burns and C.~N. Rautenberg}, {\em {The Infinite-Dimensional Optimal
  Filtering Problem with Mobile and Stationary Sensor Networks}}, Numer. Funct.
  Anal. Optim., 36 (2015), pp.~181--224,
  \url{https://doi.org/10.1080/01630563.2014.970647}.

\bibitem{Cannon1971}
{\sc J.~R. Cannon and R.~E. Klein}, {\em {Optimal Selection of Measurement
  Locations in a Conductor for Approximate Determination of Temperature
  Distributions}}, J. Dyn. Syst. Meas. Control, 93 (1971), pp.~193--199,
  \url{https://doi.org/10.1115/1.3426496}.

\bibitem{Cappe2005}
{\sc O.~Cappe, E.~Moulines, and T.~Ryden}, {\em {Inference in Hidden Markov
  Models}}, Springer, New York, 2005.

\bibitem{Caravani1975}
{\sc P.~Caravani, G.~{Di Pillo}, and L.~Grippo}, {\em {Optimal Location of a
  Measurement Point in a Diffusion Process}}, IFAC Proc. Vol., 8 (1975),
  pp.~62--68, \url{https://doi.org/10.1016/S1474-6670(17)67721-8}.

\bibitem{Chen1975}
{\sc W.~Chen and J.~H. Seinfeld}, {\em {Optimal location of process
  measurements}}, Int. J. Control, 21 (1975), pp.~1003--1014,
  \url{https://doi.org/10.1080/00207177508922052}.

\bibitem{Cialenco2018}
{\sc I.~Cialenco}, {\em {Statistical inference for SPDEs: an overview}}, Stat.
  Inference Stoch. Process., 21 (2018), pp.~309--329.

\bibitem{Colantuoni1978}
{\sc G.~Colantuoni and L.~Padmanabhan}, {\em {Optimal sensor selection in
  sequential estimation problems}}, Int. J. Control, 28 (1978), pp.~821--845,
  \url{https://doi.org/10.1080/00207177808922500}.

\bibitem{Collings1998}
{\sc I.~B. Collings and T.~Ryden}, {\em {A new maximum likelihood gradient
  algorithm for on-line hidden Markov model identification}}, in Proc. 1998
  IEEE Int. Conf. Acoust. Speech Signal Process., vol.~4, 1998, pp.~2261--2264
  vol.4, \url{https://doi.org/10.1109/ICASSP.1998.681599}.

\bibitem{Cressie1993}
{\sc N.~Cressie}, {\em {Statistics for Spatial Data}}, John Wiley {\&} Sons,
  Inc., Hoboken, NJ, USA, 1993.

\bibitem{Cressie2011a}
{\sc N.~Cressie and C.~K. Wikle}, {\em {Statistics for Spatio-Temporal Data}},
  John Wiley {\&} Sons, 2011.

\bibitem{Curtain1975}
{\sc R.~Curtain}, {\em {A Survey of Infinite-Dimensional Filtering}}, SIAM
  Rev., 17 (1975), pp.~395--411, \url{https://doi.org/10.1137/1017041}.

\bibitem{Curtain1978}
{\sc R.~F. Curtain, A.~Ichikawa, and E.~Ryan}, {\em {Optimal location of
  sensors for filtering for distributed systems}}, in Distrib. Param. Syst.
  Model. Identif., A.~Ruberti, ed., Springer, Berlin, Heidelberg, 1978,
  pp.~236--255, \url{https://doi.org/10.1007/BFb0003741}.

\bibitem{Curtain1978a}
{\sc R.~F. Curtain and A.~J. Pritchard}, {\em {Infinite Dimensional Linear
  Systems Theory}}, Springer, Berlin, Heidelberg, 1978,
  \url{https://doi.org/10.1007/BFb0006761}.

\bibitem{Curtain1995}
{\sc R.~F. Curtain and H.~Zwart}, {\em {An Introduction to Infinite-Dimensional
  Linear Systems Theory}}, Springer-Verlag, New York, 1995,
  \url{https://doi.org/10.1007/978-1-4612-4224-6}.

\bibitem{DaPrato2014}
{\sc G.~{Da Prato} and J.~Zabczyk}, {\em {Stochastic Equations in Infinite
  Dimensions}}, Cambridge University Press, Cambridge, 2~ed., 2014,
  \url{https://doi.org/10.1017/CBO9781107295513}.

\bibitem{DeSantis1993}
{\sc A.~{De Santis}, A.~Germani, and L.~Jetto}, {\em {Approximation of the
  Algebraic Riccati Equation in the Hilbert Space of Hilbert–Schmidt
  Operators}}, SIAM J. Control Optim., 31 (1993), pp.~847--874,
  \url{https://doi.org/10.1137/0331037}.

\bibitem{Demetriou2009}
{\sc M.~Demetriou and I.~Hussein}, {\em {Estimation of Spatially Distributed
  Processes Using Mobile Spatially Distributed Sensor Network}}, SIAM J.
  Control Optim., 48 (2009), pp.~266--291,
  \url{https://doi.org/10.1137/060677884}.

\bibitem{Doucet2001a}
{\sc A.~Doucet, N.~de~Freitas, and N.~Gordon}, {\em {Sequential Monte Carlo
  Methods in Practice}}, Springer, New York, NY, 2001,
  \url{https://doi.org/10.1007/978-1-4757-3437-9}.

\bibitem{Doucet2003}
{\sc A.~Doucet and V.~B. Tadi{\'{c}}}, {\em {Parameter estimation in general
  state-space models using particle methods}}, Ann. Inst. Stat. Math., 55
  (2003), pp.~409--422, \url{https://doi.org/10.1007/BF02530508}.

\bibitem{Duchi2011}
{\sc J.~Duchi, E.~Hazan, and Y.~Singer}, {\em {Adaptive Subgradient Methods for
  Online Learning and Stochastic Optimization}}, J. Mach. Learn. Res., 12
  (2011), pp.~2121--2159.

\bibitem{Elliott1995}
{\sc R.~J. Elliott and J.~B. Moore}, {\em {Hidden Markov models : estimation
  and control}}, Springer-Verlag, New York, 1995.

\bibitem{Falb1967}
{\sc P.~L. Falb}, {\em {Infinite-dimensional filtering: The Kalman-Bucy filter
  in Hilbert space}}, Inf. Control, 11 (1967), pp.~102--137,
  \url{https://doi.org/10.1016/S0019-9958(67)90417-2}.

\bibitem{Ge2015}
{\sc R.~Ge, F.~Huang, C.~Jin, and Y.~Yuan}, {\em {Escaping from saddle points -
  Online stochastic gradient for tensor decomposition}}, in COLT '15 Proc. 28th
  Annu. Conf. Learn. Theory, 2015.

\bibitem{Gerencser1984}
{\sc L.~Gerencs{\'{e}}r, I.~Gy{\"{o}}ngy, and G.~Michaletzky}, {\em
  {Continuous-Time Recursive Maximum Likelihood Method. A New Approach to
  Ljung's Scheme}}, IFAC Proc. Vol., 17 (1984), pp.~683--686,
  \url{https://doi.org/10.1016/S1474-6670(17)61050-4}.

\bibitem{Gerencser2009}
{\sc L.~Gerencs{\'{e}}r and V.~Prokaj}, {\em {Recursive identification of
  continuous-time linear stochastic systems - Convergence w.p.1 and in Lq}}, in
  2009 Eur. Control Conf., 2009, pp.~1209--1214,
  \url{https://doi.org/10.23919/ECC.2009.7074570}.

\bibitem{Germani1988}
{\sc A.~Germani, L.~Jetto, and M.~Piccioni}, {\em {Galerkin Approximation for
  Optimal Linear Filtering of Infinite-Dimensional Linear Systems}}, SIAM J.
  Control Optim., 26 (1988), pp.~1287--1305,
  \url{https://doi.org/10.1137/0326072}.

\bibitem{Handcock1993}
{\sc M.~S. Handcock and M.~L. Stein}, {\em {A Bayesian Analysis of Kriging}},
  Technometrics, 35 (1993), pp.~403--410.

\bibitem{Herring1974}
{\sc K.~Herring and J.~Melsa}, {\em {Optimum Measurements for Estimation}},
  IEEE Trans. Autom. Control., 19 (1974), pp.~264--266,
  \url{https://doi.org/10.1109/TAC.1974.1100568}.

\bibitem{Hintermuller2017}
{\sc M.~Hinterm{\"{u}}ller, C.~Rautenberg, M.~Mohammadi, and M.~Kanitsar}, {\em
  {Optimal Sensor Placement: A Robust Approach}}, SIAM J. Control Optim., 55
  (2017), pp.~3609--3639, \url{https://doi.org/10.1137/16M1088867}.

\bibitem{Hsu2013}
{\sc S.-B. Hsu}, {\em {Ordinary Differential Equations with Applications}},
  vol.~Volume 21, World Scientific, 2013, \url{https://doi.org/10.1142/8744}.

\bibitem{Isaacson1994}
{\sc E.~Isaacson and H.~B. Keller}, {\em {Analysis of Numerical Methods}},
  Dover Publications, New York, 1994.

\bibitem{Jardak2010}
{\sc M.~Jardak, I.~M. Navon, and M.~Zupanski}, {\em {Comparison of sequential
  data assimilation methods for the Kuramoto–Sivashinsky equation}}, Int. J.
  Numer. Methods Fluids, 62 (2010), pp.~374--402,
  \url{https://doi.org/10.1002/fld.2020}.

\bibitem{Kalman1961}
{\sc R.~E. Kalman and R.~S. Bucy}, {\em {New Results in Linear Filtering and
  Prediction Theory}}, J. Basic Eng., 83 (1961), pp.~95--103,
  \url{https://doi.org/10.1115/1.3658902}.

\bibitem{Kantas2015}
{\sc N.~Kantas, A.~Doucet, S.~S. Singh, J.~Maciejowski, and N.~Chopin}, {\em
  {On Particle Methods for Parameter Estimation in State-Space Models}}, Stat.
  Sci., 30 (2015), pp.~328--351, \url{https://doi.org/10.1214/14-STS511}.

\bibitem{Karatzas1998}
{\sc I.~Karatzas and S.~E. Shreve}, {\em {Brownian Motion and Stochastic
  Calculus}}, Springer-Verlag New York, 2nd~ed., 1998,
  \url{https://doi.org/10.1007/978-1-4612-0949-2}.

\bibitem{Kim1996}
{\sc Y.~Kim}, {\em {Parameter estimation for an infinite dimensional stochastic
  differential equation}}, J. Korean Stat. Soc., 25 (1996), pp.~161--173.

\bibitem{Kingma2015}
{\sc D.~Kingma and J.~Ba}, {\em {Adam: a Method for Stochastic Optimisation}},
  in Int. Conf. Learn. Represent., 2015, pp.~1--13.

\bibitem{Korbicz1994}
{\sc J.~Korbicz and D.~Uci{\'{n}}ski}, {\em {Sensor Allocation for State and
  Parameter Estimation of Distributed Systems}}, in Discret. Struct. Optim.,
  W.~Gutkowski and J.~Bauer, eds., Springer, Berlin, Heidelberg, 1994,
  pp.~178--189, \url{https://doi.org/10.1007/978-3-642-85095-0_18}.

\bibitem{Koski1986}
{\sc T.~Koski and W.~Loges}, {\em {On minimum-contrast estimation for hilbert
  space-valued stochastic differential equations}}, Stochastics, 16 (1986),
  pp.~217--225, \url{https://doi.org/10.1080/17442508608833374}.

\bibitem{Krishnamurthy2002}
{\sc V.~Krishnamurthy and G.~G. Yin}, {\em {Recursive algorithms for estimation
  of hidden Markov models and autoregressive models with Markov regime}}, IEEE
  Trans. Inf. Theory, 48 (2002), pp.~458--476,
  \url{https://doi.org/10.1109/18.979322}.

\bibitem{Kubrusly1977}
{\sc C.~S. Kubrusly}, {\em {Distributed Parameter System Identification: A
  Survey}}, Int. J. Control, 26 (1977), pp.~509--535,
  \url{https://doi.org/10.1080/00207177708922326}.

\bibitem{Kubrusly1985}
{\sc C.~S. Kubrusly and H.~Malebranche}, {\em {Sensors and controllers location
  in distributed systems—A survey}}, Automatica, 21 (1985), pp.~117--128,
  \url{https://doi.org/10.1016/0005-1098(85)90107-4}.

\bibitem{Kumar1991}
{\sc P.~Kumar, T.~E. Unny, and K.~Ponnambalam}, {\em {Stochastic partial
  differential equations in groundwater hydrology}}, Stoch. Hydrol. Hydraul., 5
  (1991), pp.~239--251, \url{https://doi.org/10.1007/BF01544060}.

\bibitem{Kumar1978}
{\sc S.~Kumar and J.~Seinfeld}, {\em {Optimal location of measurements for
  distributed parameter estimation}}, IEEE Trans. Automat. Contr., 23 (1978),
  pp.~690--698, \url{https://doi.org/10.1109/TAC.1978.1101803}.

\bibitem{Kwakernaak1972}
{\sc H.~Kwakernaak and R.~Sivan}, {\em {Linear Optimal Control Systems}},
  Wiley-Interscience, New York, 1972.

\bibitem{LeGland1995}
{\sc F.~LeGland and L.~Mevel}, {\em {Recursive identification of HMMs with
  observations in a finite set}}, in Proc. 1995 34th IEEE Conf. Decis. Control,
  vol.~1, 1995, pp.~216--221 vol.1,
  \url{https://doi.org/10.1109/CDC.1995.478681}.

\bibitem{LeGland1997}
{\sc F.~LeGland and L.~Mevel}, {\em {Recursive Identification in Hidden Markov
  Models}}, in Proc. 1997 36th IEEE Conf. Decis. Control, 1997, pp.~3468--3473,
  \url{https://doi.org/10.1109/CDC.1997.652384}.

\bibitem{Levanony1994}
{\sc D.~Levanony, A.~Shwartz, and O.~Zeitouni}, {\em {Recursive identification
  in continuous-time stochastic processes}}, Stoch. Process. their Appl., 49
  (1994), pp.~245--275, \url{https://doi.org/10.1016/0304-4149(94)90137-6}.

\bibitem{Lindgren2011}
{\sc F.~Lindgren, H.~Rue, and J.~Lindstrom}, {\em {An explicit link between
  Gaussian fields and Gaussian Markov random fields: the stochastic partial
  differential equation approach}}, J. R. Stat. Soc. Ser. B (Statistical
  Methodol., 73 (2011), pp.~423--498,
  \url{https://doi.org/10.1111/j.1467-9868.2011.00777.x}.

\bibitem{Lions1971}
{\sc J.~L. Lions}, {\em {Optimal Control of Systems Governed by Partial
  Differential Equations}}, Springer-Verlag, Berlin, Heidelberg, 1971.

\bibitem{Liptser2001}
{\sc R.~S. Liptser and A.~N. Shiryaev}, {\em {Statistics of Random Processes}},
  Springer, Berlin, Heidelberg, 2~ed., 2001,
  \url{https://doi.org/10.1007/978-3-662-13043-8}.

\bibitem{Liu2016}
{\sc X.~Liu, K.~Yeo, Y.~Hwang, J.~Singh, and J.~Kalagnanam}, {\em {A
  statistical modeling approach for air quality data based on physical
  dispersion processes and its application to ozone modeling}}, Ann. Appl.
  Stat., 10 (2016), pp.~756--785, \url{https://doi.org/10.1214/15-AOAS901}.

\bibitem{Liu2019}
{\sc X.~Liu, K.~Yeo, and S.~Lu}, {\em {Statistical Modeling for Spatio-Temporal
  Data from Physical Convection-Diffusion Processes}}, arXiv Prepr.,  (2019),
  \url{https://arxiv.org/abs/1910.10375}.

\bibitem{Ljung1987}
{\sc L.~Ljung, S.~K. Mitter, and J.~M.~F. Moura}, {\em {Optimal Recursive
  Maximum Likelihood Estimation}}, IFAC Proc. Vol., 20 (1987), pp.~241--242,
  \url{https://doi.org/10.1016/S1474-6670(17)55040-5}.

\bibitem{Llopis2017}
{\sc F.~P. Llopis, N.~Kantas, A.~Beskos, and A.~Jasra}, {\em {Particle
  Filtering for Stochastic Navier-Stokes Signal Observed with Linear Additive
  Noise}}, arXiv Prepr.,  (2017), \url{http://arxiv.org/abs/1710.04586},
  \url{https://arxiv.org/abs/1710.04586}.

\bibitem{Lototsky2009}
{\sc S.~V. Lototsky}, {\em {Statistical inference for stochastic parabolic
  equations: a spectral approach}}, Publicacions Matem{\`{a}}tiques, 53 (2009),
  pp.~3--45.

\bibitem{Man2007}
{\sc C.~Man and C.~W. Tsai}, {\em {Stochastic Partial Differential
  Equation-Based Model for Suspended Sediment Transport in Surface Water
  Flows}}, J. Eng. Mech., 133 (2007), pp.~422--430,
  \url{https://doi.org/10.1061/(ASCE)0733-9399(2007)133:4(422)}.

\bibitem{Matsumoto2002}
{\sc Y.~Matsumoto}, {\em {An Introduction to Morse Theory}}, in Trans. Math.
  Monogr. No. 208, Providence, Rhode Island, 2002, American Mathematical
  Society.

\bibitem{Mercer1909}
{\sc J.~Mercer and A.~R. Forsyth}, {\em {XVI. Functions of positive and
  negative type, and their connection the theory of integral equations}},
  Philos. Trans. R. Soc. London. Ser. A, Contain. Pap. a Math. or Phys.
  Character, 209 (1909), pp.~415--446,
  \url{https://doi.org/10.1098/rsta.1909.0016}.

\bibitem{Mohapl1994}
{\sc J.~Mohapl}, {\em {Maximum likelihood estimation in linear infinite
  dimensional models}}, Commun. Stat. Stoch. Model., 10 (1994), pp.~781--794,
  \url{https://doi.org/10.1080/15326349408807322}.

\bibitem{Mohapl2000}
{\sc J.~Mohapl}, {\em {A Stochastic Advection-Diffusion Model for the Rocky
  Flats Soil Plutonium Data}}, Ann. Inst. Stat. Math., 52 (2000), pp.~84--107,
  \url{https://doi.org/10.1023/A:1004137016101}.

\bibitem{Morris2011}
{\sc K.~Morris}, {\em {Linear-Quadratic Optimal Actuator Location}}, IEEE
  Trans. Automat. Contr., 56 (2011), pp.~113--124,
  \url{https://doi.org/10.1109/TAC.2010.2052151}.

\bibitem{Moura1986}
{\sc J.~M.~F. Moura and S.~K. Mitter}, {\em {Identification and Filtering -
  Optimal Recursive Maximum Likelihood Approach}}, tech. report, Massachusetts
  Institute of Technology and Laboratory for Information and Decision Systems,
  1986.

\bibitem{Omatu1978}
{\sc S.~Omatu, S.~Koide, and T.~Soeda}, {\em {Optimal sensor location problem
  for a linear distributed parameter system}}, IEEE Trans. Automat. Contr., 23
  (1978), pp.~665--673, \url{https://doi.org/10.1016/S1474-6670(17)66839-3}.

\bibitem{Omatu1989}
{\sc S.~Omatu and J.~Seinfeld}, {\em {Distributed Parameter Systems. Theory and
  Applications.}}, Oxford University Press, Oxford, 1989.

\bibitem{Pardoux2005}
{\sc E.~Pardoux and A.~Y. Veretennikov}, {\em {On the Poisson equation and
  diffusion approximation 3}}, Ann. Probab., 33 (2005), pp.~1111--1133,
  \url{https://doi.org/10.1214/009117905000000062}.

\bibitem{Patan2012}
{\sc M.~Patan}, {\em {Optimal Sensor Networks Scheduling in Identification of
  Distributed Parameter Systems}}, Springer, Berlin, Heidelberg, 2012.

\bibitem{Pazy1983}
{\sc A.~Pazy}, {\em {Semigroups of Linear Operators and Applications to Partial
  Differential Equations}}, Springer, New York, NY, 1983,
  \url{https://doi.org/10.1007/978-1-4612-5561-1}.

\bibitem{Polis1982}
{\sc M.~Polis}, {\em {The distributed system parameter identification
  problem}}, in Proc. 3rd IFAC Sympsoium Control Distrib. Param. Syst.,
  Toulouse, France, 1982.

\bibitem{Poyiadjis2011}
{\sc G.~Poyiadjis, A.~Doucet, and S.~S. Singh}, {\em {Particle approximations
  of the score and observed information matrix in state space models with
  application to parameter estimation}}, Biometrika, 98 (2011), pp.~65--80,
  \url{https://doi.org/10.1093/biomet/asq062}.

\bibitem{Qian1999}
{\sc N.~Qian}, {\em {On the momentum term in gradient descent learning
  algorithms}}, Neural Networks, 12 (1999), pp.~145--151,
  \url{https://doi.org/https://doi.org/10.1016/S0893-6080(98)00116-6}.

\bibitem{Rafajlowicz1984}
{\sc E.~Rafajlowicz}, {\em {Optimisation of Measurements for State Estimation
  in Parabolic Distributed Systems}}, Kybernetika, 20 (1984), pp.~413--422.

\bibitem{Rautenberg2010}
{\sc C.~N. Rautenberg}, {\em {A Distributed Parameter Approach to Optimal
  Filtering and Estimation with Mobile Sensor Networks}}, phd thesis, Virginia
  Polytechnic Institute and State University, 2010.

\bibitem{Reddi2018}
{\sc S.~J. Reddi, S.~Kale, and S.~Kumar}, {\em {On the Convergence of Adam and
  Beyond}}, in Int. Conf. Learn. Represent., 2018.

\bibitem{Rosen1991}
{\sc I.~G. Rosen}, {\em {Convergence of Galerkin approximations for operator
  Riccati equations—A nonlinear evolution equation approach}}, J. Math. Anal.
  Appl., 155 (1991), pp.~226--248,
  \url{https://doi.org/10.1016/0022-247X(91)90035-X}.

\bibitem{Schwartz2019}
{\sc B.~Schwartz, S.~Gannot, E.~A.~P. Habets, and Y.~Noam}, {\em {Recursive
  Maximum Likelihood Algorithm for Dependent Observations}}, IEEE Trans. Signal
  Process., 67 (2019), pp.~1366--1381,
  \url{https://doi.org/10.1109/TSP.2018.2889945}.

\bibitem{Sell2002}
{\sc G.~R. Sell and Y.~You}, {\em {Dynamics of Evolutionary Equations}},
  Springer, New York, NY, 2002.

\bibitem{Sharrock2020a}
{\sc L.~Sharrock and N.~Kantas}, {\em {Two-Timescale Stochastic Gradient
  Descent in Continuous Time with Application to Joint Online Parameter
  Estimation and Optimal Sensor Placement}}, arXiv Prepr.,  (2020),
  \url{https://arxiv.org/abs/2007.15998}.

\bibitem{Sigrist2015}
{\sc F.~Sigrist, H.~R. Kunsch, and W.~A. Stahel}, {\em {Stochastic partial
  differential equation based modelling of large space–time data sets}}, J.
  R. Stat. Soc. Ser. B (Statistical Methodol., 77 (2015), pp.~3--33,
  \url{https://doi.org/10.1111/rssb.12061}.

\bibitem{Sirignano2017a}
{\sc J.~Sirignano and K.~Spiliopoulos}, {\em {Stochastic Gradient Descent in
  Continuous Time}}, SIAM J. Financ. Math., 8 (2017), pp.~933--961,
  \url{https://doi.org/10.1137/17M1126825}.

\bibitem{Soderstrom1983}
{\sc T.~Soderstrom and L.~Ljung}, {\em {Theory and Practice of Recursive
  Identification}}, MIT Press, 1983.

\bibitem{Stein1999}
{\sc M.~L. Stein}, {\em {Interpolation of Spatial Data: Some Theory of
  Kriging}}, Springer-Verlag New York, 1999.

\bibitem{Surace2019}
{\sc S.~C. Surace and J.~Pfister}, {\em {Online Maximum-Likelihood Estimation
  of the Parameters of Partially Observed Diffusion Processes}}, IEEE Trans.
  Automat. Contr., 64 (2019), pp.~2814--2829,
  \url{https://doi.org/10.1109/TAC.2018.2880404}.

\bibitem{Tadic2010}
{\sc V.~B. Tadic}, {\em {Analyticity, Convergence, and Convergence Rate of
  Recursive Maximum-Likelihood Estimation in Hidden Markov Models}}, IEEE
  Trans. Inf. Theory, 56 (2010), pp.~6406--6432,
  \url{https://doi.org/10.1109/TIT.2010.2081110}.

\bibitem{Tadic2018}
{\sc V.~B. Tadi{\'{c}} and A.~Doucet}, {\em {Asymptotic properties of recursive
  maximum likelihood estimation in non-linear state-space models}}, arXiv
  Prepr.,  (2018), \url{https://arxiv.org/abs/1806.09571}.

\bibitem{Tanabe1960}
{\sc H.~Tanabe}, {\em {On the equations of evolution in a Banach space}}, Osaka
  J. Math., 12 (1960), pp.~363--376.

\bibitem{Tang2017}
{\sc S.~Tang and K.~A. Morris}, {\em {Optimal sensor design for infinite-time
  Kalman filters}}, in 2017 IEEE 56th Annu. Conf. Decis. Control, 2017,
  pp.~64--69, \url{https://doi.org/10.1109/CDC.2017.8263644}.

\bibitem{Ucinski2005}
{\sc D.~Ucinski}, {\em {Optimal Measurement Methods for Distributed Parameter
  System Identification}}, CRC Press, Boca Raton, FL, 2005.

\bibitem{Unny1989}
{\sc T.~E. Unny}, {\em {Stochastic partial differential equations in
  groundwater hydrology}}, Stoch. Hydrol. Hydraul., 3 (1989), pp.~135--153,
  \url{https://doi.org/10.1007/BF01544077}.

\bibitem{Ventzel1965}
{\sc A.~Ventzel}, {\em {On equations of the theory of conditional Markov
  processes}}, Probab. Theory Its Appl., 10 (1965), pp.~357--361.

\bibitem{Whittle1954}
{\sc P.~Whittle}, {\em {On stationary processes in the plane}}, Biometrika, 41
  (1954), pp.~434--449, \url{https://doi.org/10.1093/biomet/41.3-4.434}.

\bibitem{Wu2016}
{\sc X.~Wu, B.~Jacob, and H.~Elbern}, {\em {Optimal Control and Observation
  Locations for Time-Varying Systems on a Finite-Time Horizon}}, SIAM J.
  Control Optim., 54 (2016), pp.~291--316,
  \url{https://doi.org/10.1137/15M1014759}.

\bibitem{Yu1973}
{\sc T.~K. Yu and J.~H. Seinfeld}, {\em {Observability and optimal measurement
  location in linear distributed parameter systems}}, Int. J. Control, 18
  (1973), pp.~785--799, \url{https://doi.org/10.1080/00207177308932556}.

\bibitem{Zeiler2012}
{\sc M.~D. Zeiler}, {\em {ADADELTA: An Adaptive Learning Rate Method}}, arXiv
  Prepr.,  (2012), \url{https://arxiv.org/abs/1212.5701}.

\bibitem{Zhang2018}
{\sc M.~Zhang and K.~Morris}, {\em {Sensor Choice for Minimum Error Variance
  Estimation}}, IEEE Trans. Automat. Contr., 63 (2018), pp.~315--330,
  \url{https://doi.org/10.1109/TAC.2017.2714643}.

\end{thebibliography}

\end{document}